\documentclass[10pt,a4paper,reqno,english]{amsart}
\title[]{Global Strichartz estimates for an inhomogeneous Maxwell system}
\def\PpP{-}% <<< pagine da mostrare
\usepackage{%
    amssymb%
    ,babel%
    ,csquotes%
    ,graphicx%
    ,mathrsfs%
    ,eucal%
    ,xcolor%
    ,enumitem%
    ,geometry%
    ,esint}
\usepackage[files,\PpP]{pagesel}%
 \usepackage[%
    texencoding=ascii%
    ,backend=biber%
     ,url=false%
    ,doi=false%
    ,isbn=false%
     ,giveninits=true%
    ,maxnames=5%
    ,hyperref]{biblatex}
\usepackage[%
    pdfauthor={Piero D'Ancona}%
    ,colorlinks=true%
    ,linkcolor=magenta%
    ,citecolor=cyan%
    ,urlcolor=blue%
    ,hyperfootnotes=false%
]{hyperref}

\newcommand{\bra}[1]{\langle #1 \rangle}
\newcommand{\one}[1]{\mathbf{1}_{#1}}

\usepackage{stmaryrd}
  {\baselineskip0pt\normalfont\footnotesize%
  \abovedisplayskip=1pt\abovedisplayshortskip=1pt%
  \belowdisplayskip=1pt\belowdisplayshortskip=1pt%
  \vskip2pt\par\noindent$\llbracket$\ignorespaces}%
  {\ignorespaces$\rrbracket$\vskip2pt}
\numberwithin{equation}{section}
\newtheorem{theorem}{Theorem}[section]
\newtheorem{corollary}[theorem]{Corollary}
\newtheorem{lemma}[theorem]{Lemma}
\newtheorem{proposition}[theorem]{Proposition}
\theoremstyle{remark}
\newtheorem{remark}[theorem]{Remark}

\theoremstyle{definition}

\def\B{\mathbf{B}}
\def\E{\mathbf{E}}

\def\H{\mathbf{H}}
\def\D{\mathbf{D}}

\def\J{\mathbf{J}}

\allowdisplaybreaks

\date{\today}

\author[P.~D'Ancona]{Piero D'Ancona}
\address{Piero D'Ancona:
Dipartimento di Matematica\\
Sapienza Universit\`{a} di Roma\\
Piazzale A.~Moro 2\\
00185 Roma\\
Italy}
\email{dancona@mat.uniroma1.it}
\author[R.~Schnaubelt]{Roland Schnaubelt}
\address{Roland Schnaubelt:
Karlsruhe Institute of Technology\\
Department of Mathematics\\
76128 Karlsruhe\\
Germany}
\email{schnaubelt@kit.edu}
\thanks{Funded by the Deutsche Forschungsgemeinschaft (DFG, German Research Foundation) 
--Project-ID 258734477 -- SFB 1173%The first author is partially supported by the Project 
}

\begin{document}

\begin{abstract}
We show global-in-time Strichartz estimates for the isotropic Maxwell system
with divergence free data.  On the scalar permittivity and permeability
we impose decay assumptions as $|x|\to\infty$ and a non-trapping condition.
The proof is based on smoothing estimates in weighted $L^2$ spaces which
follow from corresponding resolvent estimates for the underlying Helmholtz problem.
\end{abstract}

\maketitle

%%% >>> INDEX (toc)
%\tableofcontents

% e_f_pre  <<<<<<<<< PREAMBOLO

\section{Introduction}\label{sec:intr}

This paper investigates a model for the propagation of
electromagnetic waves in continuous media, the
\emph{Maxwell equations} 
\begin{equation}\label{eq:maxw1}
  \D_{t}=\nabla \times\H-\J,
  \qquad
  \B_{t}=-\nabla \times\E,
  \qquad
  \nabla \cdot \D=\nabla \cdot \B=0,
\end{equation}
on $\mathbb{R}_{t}\times \mathbb{R}^{3}_{x}$
with \emph{linear inhomogeneous} material laws
\begin{equation}\label{eq:constrel}
  \D= \epsilon(x)\E,
  \qquad
  \B=\mu(x)\H,
\end{equation}
and the (divergence free) current density $\J=\J(t,x)$.
Here, $\E$ and $\D$ are the electric fields, 
$\B$ and $\H$ are the magnetic fields, and the permittivity
$\epsilon$ and the permeability $\mu$ are positive scalar functions
on $\mathbb{R}^{3}$. Hence the model is \emph{isotropic},
i.e., the interaction of fields with matter depends on
the location but not on the direction of the fields
$\D,\H,\E,\B:\mathbb{R}\times \mathbb{R}^{3}\to \mathbb{R}^{3}$.
We note that the divergence constraints follow from the evolution
equations if the initial data $\D(0)$ and $\B(0)$ and the current
$\J$ are divergence free.

The Maxwell system is the foundation of electromagnetic theory 
so that it is not necessary to recall the importance of model
\eqref{eq:maxw1} and \eqref{eq:constrel} in applications,
including the classical case $\epsilon,\mu=\mathrm{const}$.
Despite the large literature devoted to the subject,
see e.g.\ the monographs \cite{Cessenat96} and  \cite{FabrizioMorro03},
many important questions are still unclear.

Global well posedness in Sobolev spaces
$H^{s}$ of the Cauchy problem for \eqref{eq:maxw1}
follows from the general theory of hyperbolic systems,
under rather weak conditions on the coefficients $\epsilon$ and $\mu$.
Here we are 
mainly interested in the asymptotic properties of solutions.
Besides its inherent importance,
information on the decay of the solutions
is essential for the study of the corresponding nonlinear
problems. In the constant coefficient case
\begin{equation*}
  \E_{t}=\nabla \times\B -\J,
  \qquad
  \B_{t}=-\nabla \times\E,
  \qquad
  \nabla \cdot \E=\nabla \cdot \B=0,
\end{equation*}
with data
\begin{equation*}
  \E(0,x)=\E_{0},\qquad \B(0,x)=\B_{0},
\end{equation*}
solutions are easily seen to satisfy diagonal systems of
wave equations
\begin{equation*}
  \square \E=-\J_t,\qquad
  \square \B=\nabla\times \J.
\end{equation*}
Hence one can apply the well established theory
on dispersive properties of wave equations.
The strongest property is the pointwise decay
\begin{equation}\label{eq:ptwdec1}
  \|\E(t,\cdot)\|_{L^{\infty}}
  +
  \|\B(t,\cdot)\|_{L^{\infty}}
  \lesssim
  \bigl(\|\nabla\E_{0}\|_{L^{1}}+
    \|\nabla\B_{0}\|_{L^{1}}\bigr) \cdot |t|^{-1},
\end{equation}
where we set $\J=0$.
From \eqref{eq:ptwdec1} \emph{Strichartz estimates} can be deduced.
For all couples of \emph{wave admissible}
indices $(p,q)$ and $(r,s)$, that is to say
\begin{equation}\label{eq:waint}
  \textstyle
  \frac 1p+\frac 1q=\frac 12,
  \qquad
  p\in[2,\infty],
  \qquad
  q\in[2,\infty)
\end{equation}
 in dimension $3$, we have
\begin{equation*}%\label{eq:strich1}
  \||D|^{-\frac 2p}D_{t,x}\E\| _{L^{p}L^{q}}
  +
  \||D|^{-\frac 2p}D_{t,x}\B\| _{L^{p}L^{q}}
  \lesssim
  \|\nabla\E_{0}\|_{L^2}+\|\nabla\B_{0}\|_{L^2}+\|\J(0,\cdot)\|_{L^2}
   + \||D|^{\frac 2r}D_{t,x}\J\| _{L^{r'}L^{s'}}
\end{equation*}
(see \cite{GinibreVelo95-b}, \cite{KeelTao98-a}).
Here we are using the notations 
$|D|^{s}u=\mathcal{F}^{-1}(|\xi|^{s}\widehat{u}(\xi))$,
where $\mathcal{F}u=\widehat{u}$ is the Fourier transform,
and $L^{p}L^{q}=L^{p}(\mathbb{R};L^{q}(\mathbb{R}^{3}))$.
An even weaker form of dispersion is expressed by the
so called \emph{smoothing estimates}
\begin{equation}\label{eq:smooest1}
  \|\bra{x}^{-1/2-}\E\|_{L^{2}L^{2}}+
  \|\bra{x}^{-1/2-}\B\|_{L^{2}L^{2}}
  \lesssim
  \|\E_{0}\|_{L^{2}}+
  \|\B_{0}\|_{L^{2}}
\end{equation}
for $\J=0$.
(See e.g.~\cite{DAncona15-a} for a comprehensive framework
for such estimates.)

Substantial work has been devoted in recent years to extend
dispersive estimates to more general equations, including
in particular equations with electromagnetic potentials
or variable coefficients, and equations on manifolds
(see among many others
\cite{Goldberg04},
\cite{RodnianskiSchlag04-a}, \cite{Schlag05-b}
for the Schr\"odinger equation;
\cite{Cuccagna00-a}, \cite{GeorgievVisciglia03-a},
\cite{DAnconaPierfelice05-a} for the wave equation;
for wave equations with variable coefficients in highest order,
\cite{Tataru08-a}, \cite{SoggeWang10}, \cite{MetcalfeTataru12};
concerning dispersive estimates,
\cite{Yajima95-b},
\cite{Yajima95-a}, \cite{Yajima99-a}, \cite{ArtbazarYajima00},
\cite{GoldbergSchlag04}, 
\cite{DAnconaFanelli06-a}). 

Astonishingly, only little is known about such estimates for 
the Maxwell system \eqref{eq:maxw1} and \eqref{eq:constrel}. 
In \cite{DumasSueur12} \emph{local-in-time} Strichartz estimates
were shown for smooth scalar coefficients $\epsilon$ and $\mu$ 
being constant outside a compact set. For matrix valued 
coefficients the situation seems to be much more complicated, 
as already for constant matrices $\epsilon$ and $\mu$ the 
dispersive decay depends on the multiplicity of their 
eigenvalues, see \cite{Liess91}, \cite{LucenteZiliotti00} and also \cite{MandelSchippa21}. 
Very recently, local-in-time  Strichartz estimates with matrix valued 
(anisotropic) coefficients were shown  in the two dimensional case, \cite {SchippaSchnaubelt21}.
In the present work we are concerned with  \emph{global-in-time} 
Strichartz estimates for scalar  $\epsilon$ and $\mu$ in $C^2$ 
under some decay assumptions as $|x|\to\infty$.

In our arguments we use  a second-order formulation of \eqref{eq:maxw1} and \eqref{eq:constrel}. 
By a computation similar to the constant coefficient
case, any solution
$\D(t,x)$ to the problem \eqref{eq:maxw1} with \eqref{eq:constrel}
also  solves the system
\begin{equation}\label{eq:waveD}
\textstyle
  \D_{tt}+
  \nabla \times\frac 1\mu \nabla \times \frac1\epsilon \D=-\J_t,
  \quad
  \nabla \cdot\D=0,
  \qquad
  \D(0,x)=\D_{0},
  \quad
  \D_{t}(0,x)=\nabla \times \frac1\mu \B_{0}-\J(0).
\end{equation}
The other fields satisfy similar  equations, e.g., $\B$ 
satisfies an analogous system with  $\epsilon$ and $\mu$ interchanged and modified data, namely 
\begin{equation}\label{eq:waveB}
  \textstyle
  \B_{tt}+
  \nabla \times\frac 1 \epsilon \nabla \times \frac1\mu \B=\nabla \times \frac1\epsilon\J,
  \quad
  \nabla \cdot\B=0,
  \qquad
  \B(0,x)=\B_{0},
  \quad
  \B_{t}(0,x)=-\nabla \times \frac1\epsilon\D_{0}.
\end{equation}
The material laws \eqref{eq:constrel} then imply
\begin{align}\label{eq:waveE}
  &\E_{tt}+
  \tfrac1\epsilon\nabla \times\tfrac 1\mu \nabla \times  \E=-\tfrac1\epsilon\J_t,
  \quad
  \nabla \cdot(\epsilon\E)=0,
  \quad
  \E(0)=\E_{0},
  \quad
  \E_{t}(0)=\tfrac1\epsilon\nabla \times \H_{0}-\tfrac1\epsilon\J(0),\\
 \label{eq:waveH}
  &\H_{tt}+
  \tfrac1\mu\nabla \times\tfrac 1 \epsilon \nabla \times \H= \tfrac1\mu\nabla \times \tfrac1\epsilon\J,
  \quad
  \nabla \cdot(\mu\H)=0,
  \quad
  \H(0)=\H_{0},
  \quad
  \H_{t}(0)=-\tfrac1\mu\nabla \times \E_{0}.
\end{align}
In this work we focus on  \eqref{eq:waveD}.
Equations \eqref{eq:waveD} and \eqref{eq:waveB}
are essentially systems of wave equations with 
variable coefficients. Indeed, one can write
\begin{equation*}
  \textstyle
  \epsilon \mu
  \nabla \times\frac 1 \mu \nabla \times \frac1{\epsilon} U
  =
  \nabla \times \nabla \times U-
  b(x,\partial)U
\end{equation*}
where $b(x,\partial)$ is the first-order matrix operator
\begin{equation}\label{eq:defbxD}
  \textstyle
  b(x,\partial)U=
  (p+q)\times(\nabla \times U)+
  \nabla \times(p \times U)
  - (p+q) \times(p \times U)   
\end{equation}
with coefficients
\begin{equation*}
  p=\nabla\log \epsilon,
  \qquad
  q=\nabla\log \mu.
\end{equation*}
Here we heavily use that $\epsilon$ and $\mu$ are scalar.
We also denote by $\widetilde{b}(x,\partial)$ the operator
as in \eqref{eq:defbxD} with $p$ and $q$ interchanged:
\begin{equation*}
  \textstyle
  \widetilde{b}(x,\partial)U=
  (p+q)\times(\nabla \times U)+
  \nabla \times(q \times U)
  -(p+q) \times(q \times U).    
\end{equation*}
Since 
$\nabla \times \nabla \times\D=
  -\Delta\D+\nabla(\nabla \cdot\D)=-\Delta\D$,
we see that \eqref{eq:waveD} can be written as
\begin{equation}\label{eq:waveD1}
  \epsilon \mu\D_{tt}-\Delta\D-b(x,\partial)\D =-\epsilon \mu\J_t, 
  \qquad \nabla\cdot \D=0,
\end{equation}
and similarly \eqref{eq:waveB} is equivalent to
\begin{equation}\label{eq:waveB1}
  \epsilon \mu\B_{tt}-\Delta\B-\widetilde{b}(x,\partial)\B
            =  \epsilon \mu\nabla \times \tfrac1\epsilon\J, 
  \qquad \nabla\cdot \B=0.
\end{equation}
In other words, for scalar $\epsilon$ and $\mu$ the divergence constraint allows us
to reduce \eqref{eq:maxw1} and \eqref{eq:constrel}
to a wave system with uncoupled principal part $(\epsilon \mu \partial_{tt} -\Delta) I_{3\times 3}$.

The main goal of the paper is to prove the following estimates, 
which apply in particular to the fields solving the Maxwell system
\eqref{eq:maxw1} and \eqref{eq:constrel}.

\begin{theorem}[]\label{the:strichmaxw}
  Let $\epsilon(x),\mu(x):\mathbb{R}^{3}\to \mathbb{R}$
  and assume for some $\delta\in(0,1/2)$ that
  \begin{enumerate}
    \item 
    $\inf\epsilon \mu>0$
   \ and \
    $(\epsilon \mu)'_{-}\le \frac 14 (1-2^{-\delta})^{-1} \epsilon \mu
     \bra{x}^{-1-\delta}$,
    \item 
    $|\epsilon-1|+|\mu-1|\lesssim \bra{x}^{-2-\delta}$, \
    $|\nabla\epsilon|+|\nabla \mu|\lesssim \bra{x}^{-\frac32-\delta}$, \
    and \
    $|D^{2}\epsilon|+|D^{2}\mu|
      \lesssim \bra{x}^{-\frac52-\delta}$.
  \end{enumerate}
  Let $\D_0=\epsilon \E_0$, $\B_0=\mu\H_0$, and $\J$ be divergence free.
  Then the solution $\D$ to \eqref{eq:waveD} satisfies the Strichartz estimate
  \begin{equation*}
    \||D|^{-\frac 2p}D_{t,x}\D\| _{L^{p}L^{q}}
    \lesssim
    \|\nabla \D_{0}\|_{L^2}+\|\nabla\B_{0}\|_{L^2}
    + \|J(0)\|_{L^2} + \||D|^{\frac 2r}\J_t\|_{L^{r'}L^{s'}}
  \end{equation*}
  for all wave admissible $(p,q)$ and $(r,s).$
  The solution $\B$   to \eqref{eq:waveB} fulfill 
   \begin{equation*}
    \||D|^{-\frac 2p}D_{t,x}\B\| _{L^{p}L^{q}} 
    \lesssim
    \|\nabla \D_{0}\|_{L^2}+\|\nabla\B_{0}\|_{L^2}
    + \||D|^{\frac 2r}\nabla \J\|_{L^{r'}L^{s'}}.
  \end{equation*}
  Here we can replace $\D$ by $\E$ and $\B$ by $\H$, solving \eqref{eq:waveE}
  respectively  \eqref{eq:waveH}.
\end{theorem}

% In the inhomogeneous estimate we have to replace the spatial gradient $D_x$
% by the square root of the operators
% \begin{equation*}
%   \textstyle
%   H_{\D}\D=\nabla \times\frac 1\mu \nabla \times \frac1{\epsilon}\D, \qquad 
%    H_{\E}\E=\tfrac1\epsilon\nabla \times\frac 1\mu \nabla \times \E,
% \end{equation*}
% which are self adjoint on divergence free vector fields with a weighted $L^2$
% scalar product, see \eqref{eq:HS-D} and \eqref{eq:HS-E}, respectively.
% One further defines the operators $H_{\B}$, repectively $H_{\H}$, by 
% interchanging $\epsilon$ and $\mu$  in $H_{\D}$, respectively $H_{\E}$.
% 
% \begin{theorem}[]\label{the:strichmaxw-in}
% Under the assumptions of Theorem~\ref{the:strichmaxw},
%  let $\D_0=\E_0=\B_0=\H_0=0$ and $\J$ be divergence free.
%   Then each solution $\D$ to \eqref{eq:waveD} satisfies the Strichartz estimate
%   \begin{equation*}
%     \||D|^{-\frac 2p}D_{t}\D\| _{L^{p}L^{q}} +\||D|^{-\frac 2p}\sqrt{H_\D}\D\| _{L^{p}L^{q}}
%     \lesssim
%    % \|\nabla \D_{0}\|_{L^2}+\|\nabla\B_{0}\|_{L^2}+ \|J(0)\|_{L^2}+ 
%    \||D|^{\frac 2p}\J_t\|_{L^{a'}L^{b'}}
%   \end{equation*}
%   for all wave admissible $(p,q)$ and $(a,b).$
%   The solutions $\B$   to \eqref{eq:waveB} fulfill 
%    \begin{equation*}
%     \||D|^{-\frac 2p}D_{t}\B\| _{L^{p}L^{q}} +\||D|^{-\frac 2p}\sqrt{H_\B}\B\| _{L^{p}L^{q}}
%     \lesssim
%     %\|\nabla \D_{0}\|_{L^2}+\|\nabla\B_{0}\|_{L^2}+ 
%     \||D|^{\frac 2p} \nabla \J\|_{L^{a'}L^{b'}}
%   \end{equation*}
%   Here we can replace $\D$ by $\E$ and $\B$ by $\H$, solving \eqref{eq:waveE}
%   respectively  \eqref{eq:waveH}.
% \end{theorem}

We briefly discuss the previous statements. In (1),  the symbol
$(a)'_{-}=\max\{-\partial_{r}a,0\}$ denotes the negative part
of the radial derivative, and 
$\bra{x}=(1+|x|^{2})^{1/2}$.
Wave admissible couples and the notations $L^{p}L^{q}$
and $|D|^{s}$ have been defined above (see \eqref{eq:waint}).

The second assumption in (1) is our non-trapping condition. 
Note that this is a one--sided condition, affecting only the negative
part of the radial derivative of $\epsilon \mu$; it is a kind
of `repulsivity' of the coefficients.
It is well known that
some hypothesis of this type is necessary to exclude trapped rays,
which are an obstruction to global decay in time
and even to the much weaker local energy decay.
Many of our intermediate results are true under weaker decay 
assumptions than (2). For instance, our basic smoothing estimate  
\eqref{eq:globresL2s}
for the wave equation and the corresponding resolvent 
bound \eqref{eq:largefreqfin}
are shown assuming condition (1), the decay
\begin{equation}\label{eq:decay0}
|\epsilon-1|+|\mu-1|+|D^{2}\epsilon|+|D^{2}\mu|\lesssim \bra{x}^{-2-\delta}, \qquad 
    |\nabla\epsilon|+|\nabla \mu|\lesssim \bra{x}^{-1-\delta},
\end{equation}    
and a non-resonance condition for the frequency $z=0$ stated before
Proposition~\ref{pro:fullopres}. The extra decay in the above hypothesis (2) 
is needed to remove this non-resonance condition in Proposition~\ref{pro:spectralass3},
and also to establish certain Riesz-type bounds  in Lemma~\ref{lem:riesz} in (weighted) 
$L^2$ spaces which are crucial to derive the Strichartz estimates.

The proof of Theorem \ref{the:strichmaxw} is given at the end of the paper.
It follows the general principle, pionereed in
\cite{RodnianskiSchlag04-a} and further developed 
in many works
(e.g., \cite{DAnconaFanelli08-a},
\cite{ErdoganGoldbergSchlag09-a},
\cite{MetcalfeTataru09-a},
\cite{SoggeWang10},
\cite{Tataru08-a}),
that weak decay properties of solutions can
be upgraded to much stronger decay, under suitable 
regularity and localization information on the coefficients.
The main  novelty of our paper is that we treat a system
with variable coefficients in higher order terms.
We explain our proofs in more detail.

For scalar wave equations, the paper \cite{MetcalfeTataru12} 
gives global Strichartz estimates if the coefficients 
are close to constants and decay as $|x|\to\infty$. (For derivatives
the decay assumptions  are similar to \eqref{eq:decay0}.) 
Moreover,  local-in-time estimates are proven 
without the smallness condition. As we can put our
problem in the form \eqref{eq:waveD1}, we are able to apply 
these results after suitable localizations of our solution.
Recall that the possibility to
deduce global Strichartz estimates from local estimates
combined with global local energy decay was discovered in
\cite{Burq03-a}. 
The localization procedure introduces commutator terms 
which we must estimate in $L^2L^2$.
These are controlled using the smoothing estimates in
Propositions \ref{pro:estMaxw} and \ref{pro:estMaxwhom}
which are based on \eqref{eq:largefreqfin}.
In this analysis, one must switch between homogeneous 
and inhomogeneous estimates; this requires $TT^*$ arguments and
suitable Riesz-type inequalities, see Lemma~\ref{lem:riesz}.
To prove the latter, we use crucially the divergence conditions
of the Maxwell system. On the other hand,
we must avoid the usual $TT^*$ argument since it would need 
Riesz' bounds in $\dot{H}^{-2/p}_q$ which are not available 
for our operator.

The necessary smoothing estimates are deduced directly
from the resolvent bound \eqref{eq:largefreqfin} for the
stationary problem, which also involves weighted $L^2$ norms,
via Plancherel's Theorem.
In principle, here we follow the general framework
of Kato smoothing (see \cite{DAncona15-a}). However we cannot 
apply the general theory since we have to work with the operator
$L(z)=\epsilon\mu z^2 +\Delta +b(x,\partial)$ without 
divergence constraint when showing the resolvent estimates.
Since the operator $\Delta +b(x,\partial)$ is not self adjoint, 
the Kato theory can not be applied directly.

We prove the resolvent estimates by splitting into
three different regimes: 
bounded frequencies, which we handle via compactness arguments,
see Section \ref{sec:low_freq};
large frequences and large $x$, via Morawetz type estimates,
see Section \ref{sub:mora_esti};
and large frequences on a compact region of space
via Carleman estimates,
see Section \ref{sub:carl_esti}.
In the step for small frequencies one has to exclude eigenvectors and 
resonances of $L(z)$. Here it is crucial to show that such functions
have to be divergence free, which is proved in the 
relevant Propositions~\ref{pro:spectralass1}, 
\ref{pro:spectralass2}, and \ref{pro:spectralass3}
using the structure of \eqref{eq:waveD1}.

% Indeed, 
% in Section \ref{sec:stri} we deduce Strichartz
% estimates from smoothing estimates of the
% form \eqref{eq:smooest1}, \eqref{eq:globresL2s}by applying the results
% of \cite{MetcalfeTataru09-a}.
% Thus the main bulk of the paper is devoted to the
% proof of smoothing estimates, see
% Section \ref{sec:smoo_esti} and in particular
% Proposition \ref{pro:estMaxw} and
% Proposition \ref{pro:estMaxwhom}.
% The smoothing estimates, in turn, are deduced
% from corresponding resolvent estimates for the
% stationary problem, following the general framework
% of Kato smoothing (see \cite{DAncona15-a}).
% We prove the resolvent estimates by splitting into
% three different regimes: 
% small frequencies via compactness arguments,
% see Section \ref{sec:low_freq};
% large frequences and large $x$ via Morawetz type estimates,
% see Section \ref{sub:mora_esti}
% and large frequences on a compact region of space
% via Carleman estimates,
% see Section \ref{sub:carl_esti}.

\section{Low frequencies}\label{sec:low_freq}

We first prove a resolvent estimate which is valid
for all values of the complex frequency, but with a
constant $C(z)$ which may grow as $|z|\to \infty$.
Hence, we will use this estimate only for
$z$ in a suitably chosen compact region.
In the next section we shall prove a
uniform estimate for large $|z|$.
Except for the final result, in the present section the space dimension 
is $n\ge3$, however in this paper we shall only need $n=3$.

We shall apply a few variations of the following standard argument. 
Suppose a reference operator $H_{0}$ satisfies, for $z$ 
in an open domain $\Omega \subseteq\mathbb{C}$, 
a resolvent estimate
\begin{equation*}
  \|R_{0}(z)v\|_{B_{1}}\le C(z)\|v\|_{B_{2}},
  \qquad
  R_{0}(z)=(H_{0}+z)^{-1},
\end{equation*}
where $B_{1}$ and $B_{2}$ are some Banach spaces.
Suppose also that
\begin{itemize}
  \item $H$ is a relatively compact perturbation of $H_{0}$,
  meaning that the operator
  $K(z)=(H-H_{0})R_{0}(z)$ extends to a bounded and
  compact operator on $B_{2}$,
  \item $z \mapsto K(z)$ is continuous in the
operator norm.
\end{itemize}
Then we can write
\begin{equation*}
  H+z=(H-H_{0})+H_{0}+z=(I+(H-H_{0})R_{0}(z))(H_{0}+z)=
  (I+K(z))(H_{0}+z).
\end{equation*}
Let the operator $I+K:B_{2}\to B_{2}$ be
\emph{injective}. Then it is also bijective since it is
Fredholm. Moreover, the operator norm of $(I+K(z))^{-1}$
is locally bounded for $z\in \Omega$. This type of
argument is classical on weighted $L^{2}$ spaces,
see e.g.~Theorem VI.14 in \cite{ReedSimon75-a},
and it holds more
generally in Banach spaces (a fact likely rediscovered
several times, see e.g~Lemma 3.4 in \cite{DAncona20}).
As a consequence, we can invert
$H+z$ for all values of $z\in \Omega$ and the resolvent
estimate holds also for $H$, in the form
\begin{equation*}
  \|(H+z)^{-1}v\|_{B_{1}}\le C'(z)\|v\|_{B_{2}},
\end{equation*}
with a different $C'(z)$, which is locally bounded
for $z\in \Omega$, but otherwise undetermined.

We first look at the operator without lower order terms 
$b(x,\partial)$
starting with a basic resolvent estimate outside the spectrum for 
\begin{equation*}
  R(z)=(\Delta+az)^{-1}, \qquad 
  z\in \mathbb{C}\setminus[0,+\infty).
\end{equation*}

\begin{proposition}[]\label{pro:deltaa}
  Assume that $a\in L^{\infty}$, $a>0$, $\lim_{|x|\to+ \infty}a(x)=1$ and 
  $z\in \mathbb{D}=\mathbb{C}\setminus[0,+\infty)$.
  Then $\Delta+az:H^{2}\to L^{2}$ is a bijection and
  $R(z):=(\Delta+az)^{-1}$ satisfies
  \begin{equation*}
    \|R(z)f\|_{H^{2}}\le C(z)\|f\|_{L^{2}}
  \end{equation*}
  for some continuous function $C:\mathbb{D}\to \mathbb{R}^{+}$.
\end{proposition}

\begin{proof}%[Proof of ...]
 Let $z\in \mathbb{D}$ and $R_0(z)=(z+\Delta)^{-1}$. We can write
  \begin{equation}\label{eq:idres}
    \Delta+az=\Delta+z+(a-1)z=
    (I+(a-1)zR_{0}(z))(\Delta+z).
  \end{equation}
  The operator $K(z)=(a-1)zR_{0}(z)$ is bounded and compact
  on $L^{2}$. We prove that $I+K(z)$ is injective for each 
  $z\in \mathbb{D}$.
  Assume that $(I+K)u=0$. Setting $v=R_{0}(z)u$,
  we have $v\in H^{2}$ and
  \begin{equation*}
    (\Delta+za)v=0
    \quad\text{which implies}\quad 
    \textstyle
    \int|\nabla v|^{2}- z\int a|v|^{2}=0.
  \end{equation*}
  If $\Im z\neq0$, taking the imaginary part we infer
  $v=0$ and hence $u=(\Delta+z)v=0$.
  If $\Im z=0$ so that $z=-\lambda\in(-\infty,0)$,
  we obtain
  \begin{equation*}
    \textstyle
    \int|\nabla v|^{2}+\lambda\int a|v|^{2}=0
  \end{equation*}
  and this implies again $v=0$.

  Thus by analytic Fredholm theory we can invert $I+K(z)$
  on $L^{2}$ and the operator norm of $(I+K(z))^{-1}$
  is locally bounded in $z\in \mathbb{D}$. 
  The claim follows writing
  \begin{equation*}
    (\Delta+az)^{-1}=R_{0}(z)(I+K(z))^{-1}
  \end{equation*}
  and using the elementary estimate
  \begin{equation*}%\label{eq:laplresolv}
    \|R_{0}(z)v\|_{H^{2}}\le
    C(z)\|v\|_{L^{2}},
    \qquad
    C(z)=Cd(z,\mathbb{R}^{+})^{-1},
  \end{equation*} 
  and the bound on $(I+K(z))^{-1}$.
  Note that $C(z)$ blows up as $z\to \mathbb{R}^{+}$.
\end{proof}

The next step is a \emph{limiting absorption principle}
for $R(z)$, where the limits of $R(z)$ as $\pm\Im z \downarrow0$
exist in a suitable topology.
In the following, we commit a slight abuse of notation
since for $\lambda\in \sigma(-\Delta)=[0,+\infty)$ there are
\emph{two} extensions $R_{0}(\lambda\pm i0)$ of the
resolvent, and we shall denote both limits with the same notation
$R_{0}(z)$ for the sake of terseness.
The limiting absorption principle for the free Laplacian is
expressed by the uniform estimate
\begin{equation*}%\label{eq:sharpR0}
  \|R_{0}(z)f\|_{X}+|z|^{1/2}\|R_{0}(z)f\|_{Y}
  +\|\nabla R_{0}(z)f\|_{\dot Y}
  \le C\|f\|_{Y^{*}}
\end{equation*}
valid for all $z\in \mathbb{C}$, with a constant independent
of $z$. Here the norms of
$X$, $Y$ and $Y^{*}$  are defined as follows:
$Y^{*}$ is the predual of $Y$, while
\begin{equation*}
  \textstyle
  \|v\|_{X}^{2}:=
  \sup\limits_{R>0}\frac{1}{\bra{R}^{2}}\int_{\{|x|=R\}}
  |v|^{2}dS,
  \qquad
  \|v\|_{Y}^{2}:=
  \sup\limits_{R>0}\frac{1}{\bra{R}}\int_{\{|x|\le R\}}
  |v|^{2}dx.
\end{equation*}
We shall also need the (stronger) homogeneous norms
\begin{equation}\label{eq:defdotX}
  \textstyle
  \|v\|_{\dot X}^{2}=\sup_{R>0}\frac{1}{R^{2}}
  \int_{\{|x|=R\}} |v|^{2}dS
  ,\qquad
  \|v\|_{\dot Y}^{2}= \sup_{R>0}\frac 1R\int_{\{|x|\le R\}}|v|^{2}dx.
\end{equation}
We note the
equivalent expressions in terms of dyadic norms 
\begin{equation}\label{eq:dyadic}
  \|v\|_{Y}\eqsim\|\bra{x}^{-\frac 12}v\|_{\ell^{\infty}L^{2}},
  \qquad
  \|v\|_{Y^{*}}\eqsim\|\bra{x}^{\frac 12}v\|_{\ell^{1}L^{2}},
   \qquad
  \|v\|_{X}\eqsim\|\bra{x}^{-1}v\|_{\ell^{\infty}L^{\infty}L^{2}},
\end{equation}
writing (using polar coordinates in the last term)
\begin{align*}
  \textstyle
  \|v\|_{\ell^{\infty}L^{2}}&=
  \sup\limits_{j\ge0}\|v\|_{L^{2}(A_{j})},
  \qquad
  A_{0}=\{|x|\le1\},
  \quad
  A_{j}=\{2^{j-1}\le|x|\le 2^{j}\},\\
  \textstyle
  \|v\|_{\ell^{1}L^{2}}&=\sum_{j\ge0}\|v\|_{L^{2}(A_{j})}, \qquad
   \|v\|_{\ell^{\infty}L^\infty L^{2}}
      =  \sup\limits_{j\ge0} \|v\|_{L^\infty_{|x|}L^{2}_\omega (A_{j})}.
\end{align*}
These  norms can be considered as sharp versions of weighted $L^{2}$
norms. Indeed, it is easy to check the inequalities
\begin{equation}\label{eq:XY-bounds} \begin{split}
  \|\bra{x}^{-\frac 12-\delta}v\|_{L^{2}}&\le C(\delta)\|v\|_{Y},
  \qquad
  \|v\|_{Y^{*}}\le C(\delta)\|\bra{x}^{\frac 12+\delta}v\|_{L^{2}},\\
  \|\bra{x}^{-\frac 32-\delta} v\|_{L^{2}}&\le C(\delta)\|v\|_{X}, \qquad
  \|\bra{x}^{-1} v\|_Y\le \|v\|_X
\end{split}\end{equation}
for all $\delta>0$.

In the next lemma we collect the relevant estimates for the
free Laplacian. We write them at the point $z^{2}$ with
$\Im z\ge0$, thus covering the entire complex plane for
both sides of $[0,+\infty)$. (In later sections it will be convenient to use $z^2$.)
We set $\widehat{x}=|x|^{-1}x$ for $x\in\mathbb{R}^n\setminus\{0\}$.

\begin{lemma}[]\label{lem:resetsR0}
  Let $z\in \mathbb{C}$ with $\Im z\ge0$.
  Then we have,
  with constants independent of $z$,
  \begin{align}\label{eq:R0first}
    \|R_{0}(z^{2})f\|_{X}+\|z R_{0}(z^{2})f\|_{Y}+
    \|\nabla R_{0}(z^{2})f\|_{Y}&\le
    C\|f\|_{Y^{*}},\\
 \label{eq:R0rads0}
    \|(\nabla-i \widehat{x}z)R_{0}(z^{2})f\|_{L^{2}}
    &\le C\||x|f\|_{L^{2}}.
  \end{align}
  Moreover, for $s\in[\frac 12,1]$ we have,
  with $C$ independent of $s$ and $z$,
  \begin{equation}\label{eq:R0rad}
    \|\bra{x}^{s-1}(\nabla-i \widehat{x}z)R_{0}(z^{2})f\|
      _{\ell^{\infty} L^{2}}
    \le C\|\bra{x}^{s}f\|_{\ell^{1} L^{2}}.
  \end{equation}
\end{lemma}

\begin{proof}%[Proof of ...]
  Estimate \eqref{eq:R0first} is essentially the
  classical Agmon--H\"{o}rmander estimate, which is uniform 
  in $z$ in the special case of the operator $\Delta$.
  See e.g.~\cite{CassanoDAncona15-a} for a complete proof.

  Consider now \eqref{eq:R0rads0}. Take $f\in L^2$ with $|x|f\in L^2$.
  The restriction that $f\in L^2$ can be removed by approximation. Define
  $u=R_{0}(\lambda+i \eta)f$, so that
  $(\Delta+\lambda+i \eta)u=f$.
 We  multiply this equation by 
  $\overline{u}$, take the imaginary and the real
  part of the resulting identity, and integrate over
  $\mathbb{R}^{n}$. We then obtain (see \eqref{eq:firstst} and
  \eqref{eq:secest} below for a similar computation)
  \begin{equation}\label{eq:easyest}
    \textstyle
    \eta\|u\|_{L^{2}}^{2}=\Im\int f \overline{u},
    \qquad
    \|\nabla u\|_{L^{2}}^{2}=\lambda\|u\|_{L^{2}}^{2}-
    \Re\int f \overline{u}.
  \end{equation}
  If $\lambda\le 2|\eta|$, these equations imply
  \begin{equation*}
    \|\nabla u\|_{L^{2}}^{2}\le
    2|\eta|\|u\|_{L^{2}}^{2}+\|f \overline{u}\|_{L^{1}}
    \le3\|f \overline{u}\|_{L^{1}},
  \end{equation*}
  and hence
  \begin{equation*}
    (|\eta|+|\lambda|)\|u\|_{L^{2}}^{2}+
    \|\nabla u\|_{L^{2}}^{2}
    \lesssim\|f \overline{u}\|_{L^{1}}
    \le\||x|f\|_{L^{2}}\||x|^{-1}u\|_{L^{2}}.
  \end{equation*}
  Using Hardy's inequality 
  $\||x|^{-1}u\|_{L^{2}}\lesssim\|\nabla u\|_{L^{2}}$,
  we conclude
  \begin{equation*}
    (|\eta|+|\lambda|)\|u\|_{L^{2}}^{2}+
    \|\nabla u\|_{L^{2}}^{2}
    \lesssim\||x|f\|_{L^{2}}^{2}.
  \end{equation*}
  We write this estimate in terms of $z^{2}=\lambda+i \eta$.
  Note that if $\arg z\in [\frac \pi8,\pi-\frac \pi8]$, 
  then $\arg z^{2}\in [\frac \pi4,2\pi-\frac \pi4]$, i.e.,
  $\lambda\le|\eta|$. We have thus  proved
  \begin{equation}\label{eq:lminep}
    \textstyle
    \|zR_{0}(z^{2})f\|_{L^{2}}+
    \|\nabla R_{0}(z^{2})f\|_{L^{2}}
    \lesssim\||x|f\|_{L^{2}}
    \quad\text{provided}\quad 
    \arg z \in [\frac \pi8,\pi-\frac \pi8].
  \end{equation}
  This estimate obviously yields
  \begin{equation}\label{eq:power1}
    \|(\nabla-i \widehat{x}z)
    R_{0}(z^{2})f\|_{L^{2}}
    \lesssim \||x|f\|_{L^{2}}
  \end{equation}
  for the same values of $z$.
  Next, we consider the region 
  $\arg z\in[0,\frac \pi 8]\cup[\pi-\frac \pi8,\pi]$,
  i.e., $\arg z^{2}=\lambda+i \eta\in[0,\frac \pi 4]
    \cup[2\pi-\frac \pi4,2\pi]$  or equivalently 
  $0\le|\eta|\le \lambda$.
  Proposition~3.1 in \cite{BarceloVegaZubeldia13-a}
  shows that
  \begin{equation}\label{eq:BVZ}
    \|(\nabla-i \widehat{x}\sqrt{\lambda})
    R_{0}(\lambda+i \eta)f\|_{L^{2}}
    \lesssim\||x|f\|_{L^{2}}
  \end{equation}
  with a constant independent of $\eta$ and $\lambda$.
  Setting $u=R_{0}(\lambda+i \eta)f$ and
  $v=e^{-i|x| \sqrt{\lambda}}u$, we have     
  $\nabla v=e^{-i|x| \sqrt{\lambda}}  (\nabla-i \widehat{x}\sqrt{\lambda})u$. 
  By Hardy's inequality,  estimate \eqref{eq:BVZ} implies
  \begin{equation*}
    \||x|^{-1}R_{0}(\lambda+i \eta)f\|_{L^{2}}=
    \||x|^{-1}u\|_{L^{2}}=
    \||x|^{-1}v\|_{L^{2}}\lesssim
    \|\nabla v\|_{L^{2}}\lesssim\||x|f\|_{L^{2}}.
  \end{equation*}
  From the first part of \eqref{eq:easyest} we then deduce
  \begin{equation*}
    |\eta|\,\|R_{0}(\lambda+i \eta)f\|^{2}_{L^{2}}
    \le
    \||x|f\|_{L^{2}}\||x|^{-1}u\|_{L^{2}}
    \lesssim\||x|f\|_{L^{2}}^{2}.
  \end{equation*}
  Observe that for $\lambda+i \eta=z^{2}$ and
  $0\le |\eta|\le \lambda$ we have
  \begin{equation*}
    |\sqrt{\lambda}-z|=|(\Re z^{2})^{1/2}-z|\le \sqrt{|\eta|}.
  \end{equation*}
The previous estimates thus lead to
  \begin{equation*}
    \|(\nabla-iz \widehat{x})R_{0}(z^{2})f\|_{L^{2}}
    \le
    \|(\nabla-i \sqrt{\lambda} \widehat{x})R_{0}(z^{2})f\|_{L^{2}}
    +
    \sqrt{|\eta|}\|R_{0}(z^{2})f\|_{L^{2}}
    \lesssim\||x|f\|_{L^{2}}.
  \end{equation*}
Combined with \eqref{eq:power1}, we see that \eqref{eq:R0rads0}
  holds uniformly in $z$ for all $\Im z\ge0$.

  For the last assertion, we note that \eqref{eq:R0rad} for $s=1$
follows from \eqref{eq:R0rads0}. If $s=\frac 12$, inequalities  
 \eqref{eq:dyadic}  and \eqref{eq:R0first} yield
  \begin{equation*}
\|\bra{x}^{-\frac12}(\nabla-i\widehat{x}z)R_{0}(z^{2})f\|_{\ell^{\infty}L^{2}}
   \le    C\|(\nabla-i \widehat{x}z)R_{0}(z^{2})f\|_{Y}
    \le C\|f\|_{Y^{*}}\le C \|\bra{x}^{\frac 12}f\|_{\ell^{1}L^{2}}.
  \end{equation*}
  Real interpolation between the cases $s=\frac 12$ and $s=1$
  then gives \eqref{eq:R0rad}.
\end{proof}

We now prove the limiting absorption principle for $\Delta+az^2$. 
As for $R_{0}(z)$, the two extensions on the positive reals
for $\Im z \downarrow0$ and for $\Im z \uparrow 0$ are different,
but for simplicity we will use the same notation $R(z)$ for both.
The weighted $L^2$ space with norm
$\|\bra{x}^{s}u\|_{L^{2}}$ is denoted by $L^2_s$.

\begin{proposition}[]\label{pro:LAP}
  Assume $\bra{x}^{2+\delta}(a-1)\in L^{\infty}$ for some $\delta>0$.
  Then $R(z)$ satisfies the estimate
  \begin{equation}\label{eq:Rzlarge}
    \|R(z^{2})f\|_{X}+\|zR(z^{2})f\|_{Y}
    +\|\nabla R(z^{2})f\|_{Y}
    \le C(z)\|f\|_{Y^{*}}
  \end{equation}
  for all $\Im z\ge0$ and  for some continuous $C(z)$.
  Let $s'<s$ in $(1/2,1]$ and 
  $\bra{x}^{s+\frac 32+\delta}(a-1)\in L^{\infty}$. We then have
  \begin{equation}\label{eq:radest}
    \|\bra{x}^{s'-1}(\nabla-i \widehat{x}z)R(z^{2})f\|_{L^{2}}
    \le
    C(s',s,z)\|\bra{x}^{s} f\|_{L^{2}}.
  \end{equation}
  Moreover, for $f\in L^{2}_{s}$ there exists $g\in L^{2}_{s}$ 
  with $R(z^{2})f=R_{0}(z^{2})g$.
\end{proposition}

\begin{proof}%[Proof of ...]
  We shall use the inequalities
  \begin{equation}\label{eq:XY-bounds1}
    \|u\|_{Y^{*}}\lesssim\|\bra{x}^{1+\delta}u\|_{Y},
    \qquad
    \|u\|_{Y^{*}}\lesssim\|\bra{x}^{2+\delta}u\|_{X}.
  \end{equation}
valid for   any $\delta>0$, see \eqref{eq:XY-bounds}.
  Let $K(z)=(a-1)z^{2}R_{0}(z^{2})$.  The operator 
  $\bra{x}^{-2-\delta}zR_{0}(z^{2})$ is compact on $Y^{*}$
  and bounded uniformly in $z$,  as it follows from estimates  
  \eqref{eq:R0first} and  \eqref{eq:XY-bounds1} 
  (or as a special case of Lemma 3.1 in \cite{DAncona20}).
  Writing $K(z)=\bra{x}^{2+\delta}
    (a-1)z\cdot\bra{x}^{-2-\delta}zR_{0}(z^{2})$
  we see that $K(z):Y^{*}\to Y^{*}$ is also a compact operator
  for each $z\in \mathbb{C}$
  whose operator norm is locally bounded in $z\in \mathbb{C}$.
  % By estimate \eqref{eq:R0first} for the free resolvent
  % we know that $zR_{0}(z^{2}):Y^{*}\to Y$ is (uniformly) bounded.
  % This shows that the operator
  % $K(z)=(a-1)z^{2}R_{0}(z^{2})$ is bounded on $Y^{*}$
  % under our decay assumption on $a$
  % (it would be enough to assume $\bra{x}^{1+}(a-1)\in L^{\infty}$)
  % with a constant depending continuously on $z$.
  % Since the
  % operator $\bra{x}^{-2-\delta}R_{0}(z)$ is compact on $Y^{*}$
  % (see e.g.~Lemma 3.1 in \cite{DAncona20}),
  % hence $K:Y^{*}\to Y^{*}$ is also compact.

  We next prove that $I+K(z):Y^{*}\to Y^{*}$ is injective.
  Thus assume $(I+K(z))v=0$ for some $v\in Y^*\hookrightarrow L^2$.
 Let $u=R_{0}(z^{2})v$ so that $u\in Y\cap H^{2}_{loc}$ if $z\neq 0$, 
 $u\in X\cap H^{2}_{loc}$ if $z=0$, and $u$
  satisfies $\Delta u+az^{2}u=0$. If $z=0$ this means that $u\in X$
  is harmonic, hence $v=0$. If $\Im z^2\neq0$ or $z^{2}<0$, 
  we have  $u=R_{0}(z^{2})v\in H^2$.     
Proposition \ref{pro:deltaa} now yields $u=0=v$.
Finally, if $z^{2}=\lambda>0$  then $u$ satisfies
  \begin{equation*}
    (\Delta +\lambda)u+\lambda(a-1)u=0.
  \end{equation*}
 Regarding $W(x)=\lambda(a-1)$ as a potential with 
  $|x|^{2}\bra{x}^{\delta/2}W\in \ell^{1}L^{\infty}$,
 Lemma~3.3 in \cite{DAncona20} shows that $v=0$. Then
  \eqref{eq:Rzlarge} follows from \eqref{eq:R0first} as before 
	by analytic Fredholm theory  and the representation
  $R(z^{2})=R_{0}(z^{2})(I+K(z))^{-1}$.

  Consider now the radiation estimate \eqref{eq:radest}
assuming  $\bra{x}^{s+\frac 32+\delta}(a-1)\in L^{\infty}$.
 We transfer estimate \eqref{eq:R0rad} for $R_{0}$
  to the perturbed resolvent $R$, using the
  representation $R(z^{2})=R_{0}(z^{2})(I+K(z))^{-1}$. 
  In view of \eqref{eq:dyadic} and \eqref{eq:R0rad},  
  we only have to prove that $I+K(z)$ is an invertible operator on the
  weighted space $L^{2}_{s}$ with norm $\|\bra{x}^{s}f\|_{L^{2}}$. 
   Note that we have already shown that $I+K(z)$
  is injective on the larger space $Y^{*}$. It thus it remains
  to check that $K(z)$ is compact on $L^{2}_{s}$. We can write
  % \begin{equation*}
  %   \|Kf\|_{L^{2}_{s}}\le
  %   \|\bra{x}^{s+\frac 32+\delta}(a-1)z\|_{L^{\infty}}
  %   \|\bra{x}^{-\frac 32-\delta}R_{0}(z)f\|_{L^{2}}
  %   \le C(a,z)\|\bra{x}^{-2-\delta/2}R_{0}(z)f\|_{Y^{*}}
  % \end{equation*}
  \begin{equation*}
    K(z)=\bra{x}^{\frac 32+\delta}(a-1)z
    \cdot
    \bra{x}^{\frac 12}
    \cdot
    \bra{x}^{-2-\delta }zR_{0}(z^{2}).   
 %   \bra{x}^{\frac 12-\frac \delta 2}
 %   \cdot
 %   \bra{x}^{-2-\frac \delta 2}zR_{0}(z^{2}).
  \end{equation*}
Observe that $\bra{x}^{\frac 32+\delta}(a-1)z$
  is a bounded operator from $L^{2}$ to $L^{2}_{s}$ since
  $\bra{x}^{s+\frac 32+\delta}(a-1)\in L^{\infty}$,
  $\bra{x}^{\frac 12}$ is bounded from
  $Y^{*}$ to $L^{2}$ by \eqref{eq:dyadic}, and 
 $\bra{x}^{-2-\delta }R_{0}(z)$ is compact on $Y^{*}$
  because of \eqref{eq:R0first} and \eqref{eq:XY-bounds1}. 
  Summing up, $K(z):Y^{*}\to L^{2}_{s}$ is compact and 
  due to the embedding $L^{2}_{s}\hookrightarrow Y^{*}$ 
  it is also compact on $L^{2}_{s}$.

  The final claim is a consequence of the representation
  $R(z^{2})=R_{0}(z^{2})(I+K(z))^{-1}$ and of the bijectivity
  of $I+K(z)$ on  $L^{2}_{s}$ for the above values of $s$.
\end{proof}

Note that writing $\Delta R(z)f=f-azR(z)f$,
Proposition \ref{pro:LAP} also yields
\begin{equation*}
  \|\Delta R(z^{2})f\|_{Y}\le\|f\|_{Y}+|z|C(z)\|f\|_{Y^{*}}
  \le C_{1}(z)\|f\|_{Y^{*}}
\end{equation*}
where we used the inequality $\|f\|_{Y}\lesssim\|f\|_{Y^{*}}$,
cf.\ \eqref{eq:dyadic}.
This gives the complete estimate
\begin{equation}\label{eq:H2estR}
  \|R(z^{2})f\|_{X}+\|zR(z^{2})f\|_{Y}
  +\|\nabla R(z^{2})f\|_{Y}+\|\Delta R(z^{2})f\|_{Y}
  \le C(z)\|f\|_{Y^{*}}.
\end{equation}

Finally we consider the case of the full operator
\begin{equation*}
  L(z)=\Delta+a(x)z^{2}+b(x,\partial).
\end{equation*}
In the following, we actually treat a more general 
\emph{matrix operator}
\begin{equation*}
  L(z)=I_{3}\Delta+I_{3}a(x)z^{2}+b(x,\partial).
\end{equation*}
Here $I_{3}$ is the $3\times3$ identity matrix so that
the principal part is a diagonal Laplacian
operator. Moreover, $b(x,\partial)$ is a $3\times 3$ matrix 
first-order operator subject to conditions as in the scalar
case. It will be clear from the proofs that  in our setting
no change is required in the matrix case.

In order to perform the usual injectivity step,
we shall make the following spectral assumption
saying that $L(z)$ has no resonances or eigenvalues.
 See Remark~\ref{rem:spec} and 
Propositions~\ref{pro:spectralass1}, \ref{pro:spectralass2},
and \ref{pro:spectralass3}
below for a closer examination of these conditions.
There we show that these conditions only lead to mild extra conditions
when  establishing our main results on the Strichartz estimates for the Maxwell 
system. Actually, these extra conditions are only needed to exclude a resonance 
at $z=0$, see Proposition~\ref{pro:spectralass3}.

\textbf{Spectral assumption (S)}.
Let $\Im z\ge0$.
Then $L(z)u=0$ implies $u=0$, provided
\begin{enumerate}
  \item 
  either $z\not\in \mathbb{R}$ and $u\in H^{2}$
  (no eigenvalues)
  \item 
  or
  $z\in \mathbb{R}$ and $u=R_{0}(z^{2})f$ for some 
  $\bra{x}^{\frac 12+}f\in L^{2}$
  (no embedded resonances).
\end{enumerate}
Note that $u\in R_{0}(z^{2})Y^{*}$ satisfies
$\nabla u, \Delta u\in Y$, and $u\in X$
(and $u\in Y$ if $z\neq0$) by Lemma~\ref{lem:resetsR0}.

We briefly discuss condition (2) for $z=0$ (no resonances at 0).
It is necessary since the presence of resonances
competes with dispersion, a well studied effect since
\cite{JensenKato79-a}. 
%Here we use the following notion.
% \begin{definition}[]\label{def:res0}
%   We say that an operator $\Delta+b(x,\partial)$
%   has a \emph{resonance at 0} if there exists $u$
%   such that $(\Delta+b)u=0$, where
%   $u=\Delta^{-1}f$ and
%   $\bra{x}^{\frac 12+\sigma}f\in L^{2}(\mathbb{R}^{3})$
%   for some $\sigma>0$.
% \end{definition}
If $\bra{x}^{\frac 12+\sigma}f\in L^{2}$ then
$u=\Delta^{-1}f$ satisfies 
$\bra{x}^{-\frac 12-\sigma'}u\in L^{2}$ for all $\sigma'>0$,
thus our non-resonance assumption is slightly weaker than
the usual one.

\begin{remark}[]\label{rem:spec}
  Assumption (S) is satisfied for $z$ sufficiently large
  with respect to the coefficients. This is a consequence
  of estimate \eqref{eq:largefreqfin} in the next section.
\end{remark}

Moreover, the non-resonance assumption
is \emph{generic} in the following sense.
We take a parameter $\omega\in \mathbb{R}\setminus0$ and
consider the modified operator $\Delta+\omega b$.
Under the previous assumptions on $\epsilon$ and $\mu$, then the
set of values $\omega$ such that $\Delta+\omega b$ has
a resonance at 0 is \emph{discrete}. Indeed,
one easily  checks that
$0$ is a resonance for $\Delta+\omega b$ if and only if
$-\omega^{-1}$ is an eigenvalue for the compact
operator $b(x,\partial)\Delta^{-1}$ on the weighted
$L^{2}$ space with weight $\bra{x}^{\frac 12+\sigma}$.

\begin{proposition}[]\label{pro:fullopres}
  Let $L(z)=I_{3}\Delta+I_{3}a(x)z^{2}+b(x,\partial)$ with
  $|x|^{2}\bra{x}^{\delta}(a-1)\in L^{\infty}$ and
  $b(x,\partial)$ a first-order matrix
  differential operator satisfying
  \begin{equation}\label{eq:assbpacom}
    |b(x,\partial)v|\le
    C_{b}(\bra{x}^{-2-\delta}|v|+\bra{x}^{-1-\delta}|\nabla v|)
  \end{equation}
  for some $C_{b},\delta>0$.
  Assume $L(z)$ satisfies the spectral assumption (S).
  Then for $\Im z\ge0$ we have
  \begin{equation}\label{eq:LzestR}
    \|u\|_{X}+\|zu\|_{Y}
    +\|\nabla u\|_{Y}+\|\Delta u\|_{Y}
    \le C(z)\|\bra{x}^{\frac 12+}L(z)u\|_{L^{2}}.
  \end{equation}
\end{proposition}

\begin{proof}%[Proof of ...]
  As before we write
  \begin{equation}\label{eq:bsL}
    L(z)=(I+K(z))(\Delta+az^{2}),
    \qquad
    K(z)=b(x,\partial)R(z^{2}),
  \end{equation}
  where $R(z)=(\Delta+az^{2})^{-1}$ is the operator 
  constructed in Proposition \ref{pro:LAP}.
  Estimates  \eqref{eq:XY-bounds} and \eqref{eq:H2estR} and
  the assumptions on the coefficients imply the compactness
  of $K(z)$ as an operator on $L^{2}_{1/2+}$
  and the continuity of the map
  $z \mapsto K(z)$ in the operator norm.

  To prove injectivity of $I+K(z)$, assume
  $f+K(z)f=0$ for some $f\in L^{2}_{1/2+}$.
  Let $u=R(z^{2})f$ so that $u$ solves $L(z)u=0$.
  Note that by the final claim of Proposition \ref{pro:LAP}
  we also have $u=R_{0}(z^{2})g$ for some $g\in L^{2}_{1/2+}$.
  If $z\in \mathbb{R}$,  assumption (S) yields $u=0$
  and hence $f=(\Delta+az^{2})u= 0$.
  If $z\not\in \mathbb{R}$, since $Y^{*}\subset L^{2}$
  and $R(z^{2}):L^{2}\to H^{2}$, we see that $u$ is actually an
  eigenfunction of $L(z)$, and again by (S) we deduce
  $u=0$. The rest of the proof is similar to the previous ones.
\end{proof}

 The spectral assumption (S) holds
 if $a$ and $b$ have some additional structure that is present in our 
 main goal, the Maxwell system in the second-order form \eqref{eq:waveD1}.
 We first consider part (1) of (S) and exclude eigenvalues in the next result. Observe that
 the assumptions \eqref{eq:choab} and \eqref{eq:decayem} imply condition
 \eqref{eq:assbpacom} from  Proposition \ref{pro:fullopres}, cf.\  \eqref{eq:defbxD}. 
 This fact is used below several times.

\begin{proposition}[]\label{pro:spectralass1}
  Assume that the coefficients in  Proposition 
  \ref{pro:fullopres} have the form
  \begin{equation}\label{eq:choab}
      \textstyle
      a(x)=\epsilon(x)\mu(x),
      \qquad
      b(x,\partial)u=
      \nabla \times \nabla \times u-
      \epsilon(x)\mu(x)\nabla \times
      (\frac 1{\mu(x)} \nabla \times \frac1{\epsilon(x)}u),
    \end{equation}
  where $\epsilon$ and $\mu$ are bounded and 
  uniformly strictly positive. 
  Then property (1) in the spectral assumption (S) is satisfied.
\end{proposition}

\begin{proof}  
In the present case the equation $L(z)u=0$ can be rewritten as
  \begin{equation*}
    \textstyle
    z^{2}\epsilon \mu u+\Delta u+
    \nabla \times \nabla \times u-
    \epsilon\mu\nabla \times
    (\frac 1\mu \nabla \times \frac1{\epsilon}u)
    =0
  \end{equation*}
  or equivalently
  \begin{equation}\label{eq:resrewr}
    \textstyle
    z^{2} u
    +\frac{1}{\epsilon \mu}
    \nabla(\nabla \cdot u)-
    \nabla \times
    (\frac 1\mu \nabla \times \frac1{\epsilon}u)
    =0.
  \end{equation}
Assume that $z\not\in \mathbb{R}$ and
  $u\in H^{2}$ is a solution of \eqref{eq:resrewr}.
  By taking the divergence of the equation, we see that the 
  function $\phi=\nabla \cdot u$ satisfies
  \begin{equation*}
    \textstyle
    z^{2} \phi+\nabla \cdot(\frac{1}{\epsilon \mu}\nabla \phi)=0.
  \end{equation*}
  As $z\not\in \mathbb{R}$, this equation implies
  $\phi=0$ (i.e., $u$ is divergence free)
  since the operator 
  $\nabla \cdot(\frac{1}{\epsilon \mu}\nabla \phi)$
  is selfadjoint and non negative
  as soon as the (real valued) coefficient
  $\epsilon\mu$ is bounded and strictly positive.
  Thus the equation $L(z)u=0$ reduces to
  \begin{equation}\label{eq:resrewr2}
    \textstyle
    z^{2} u
    =
    \nabla \times
    (\frac 1\mu \nabla \times \frac1{\epsilon}u),
    \qquad
    \nabla \cdot u=0,
    \qquad
    u\in H^{2}.
  \end{equation}
  It is now convenient to set
  \begin{equation}\label{eq:defH}
    \E=u/\epsilon,
    \qquad
    \H=-(i\mu z)^{-1}\nabla \times \E,
  \end{equation}
  so that $(\E,\H)$ are $H^{1}$ solutions of the
  stationary Maxwell system
  \begin{equation}\label{eq:maxwellhelm}
    i\epsilon z \E=\nabla \times \H,
    \quad
    i\mu z\H=-\nabla \times \E,
    \quad
    \nabla \cdot(\epsilon \E)=\nabla \cdot(\mu \H)=0.
  \end{equation}
 We integrate the identity
    \begin{equation*}
      |\widehat{x}\times \E|^{2}+|\H|^{2}-
      |\widehat{x}\times \E+\H|^{2}=
      -2\Re(\widehat{x}\cdot \E \times \overline{\H})
    \end{equation*}
    over a sphere $|x|=R$. The divergence theorem then yields
    \begin{equation*}
      \textstyle
      \int_{|x|=R}
      [|\widehat{x}\times \E|^{2}+|\H|^{2}-
      |\widehat{x}\times \E+\H|^{2}]dS
      =-2\Re
      \int_{|x|\le R}\nabla \cdot(\E \times \overline{\H})dx.
    \end{equation*}
 Writing
    \begin{equation*}
      \nabla \cdot(\E \times \overline{\H})=
      \overline{\H} \cdot(\nabla \times {\E})
      -\E \cdot(\nabla \times \overline{\H})=
      -iz \mu |\H|^{2}+i \overline{z} \epsilon|\E|^{2},
    \end{equation*}
  we deduce 
    \begin{equation}\label{eq:partid}
      \textstyle
      \int_{|x|=R}
      [|\widehat{x}\times \E|^{2}+|\H|^{2}]d S
      \textstyle
      +
      2\Im z
      \int_{|x|\le R}[\epsilon|\E|^{2}+\mu|\H|^{2}]dx
      =
      \int_{|x|= R}
      |\widehat{x}\times \E+\H|^{2}d S.
    \end{equation}
    If we integrate in $R$ from 0 to $+\infty$, the RHS
    gives a finite contribution since $\E,\H\in L^{2}$.
    As a consequence the second integral on the LHS
    must be 0 (recall that $\Im z>0$). We have proved that
    $\E=\H=0$ and in particular $u=0$.
\end{proof}
We next treat resonances at $z^2>0$ which requires more sophisticated tools.

\begin{proposition}[]\label{pro:spectralass2}
Assume that the coefficients in  Proposition \ref{pro:fullopres} have the form
\eqref{eq:choab} and satisfy  $\epsilon,\mu>0$ as well as
  \begin{equation}\label{eq:decayem}
    \bra{x}^{2+\delta}(|\epsilon-1|+|\mu-1|+|D^{2}\epsilon|
      +|D^{2}\mu|)
    +
    \bra{x}^{1+\delta}(|\nabla\epsilon|+|\nabla\mu|)
    \in L^{\infty}.
  \end{equation}
Then also property (2) in the 
spectral assumption  (S) is satisfied if $z\in \mathbb{R}\setminus \{0\}$.
\end{proposition}
\begin{proof}
Let  $z\in \mathbb{R}\setminus \{0\}$ so that $\lambda=z^{2}>0$.
We take a solution
  $u$ of $L(z)u=0$ of the form $u=R_0(\lambda)f$ for some  
  $f\in L^{2}_{1/2+}\hookrightarrow Y^*$. 
  In particular, from \eqref{eq:H2estR} with $a=1$ we know
  that $u,\nabla u,\Delta u\in Y$.

  Proceeding as in the  previous proposition, we see that 
  $\phi=\nabla \cdot u\in Y$ satisfies
  \begin{equation*}
    \textstyle
    \lambda \phi+\nabla \cdot(\frac{1}{\epsilon \mu}\nabla \phi)=0
  \end{equation*}
  which can be written as
  \begin{equation*}
    (\Delta+\lambda) \phi
    -\nabla \beta \cdot \nabla \phi+
    \lambda(\epsilon \mu-1) \phi=0,
    \qquad
    \beta=\ln(\epsilon \mu).
  \end{equation*}
  Setting $\phi=\sqrt{\epsilon \mu}\psi$, this equation is
  transformed into
  \begin{equation}\label{eq:psi}
    \textstyle
    (\Delta+\lambda)\psi+c(x)\psi=0,
    \qquad
    c(x)=\frac 12 \Delta \beta-\frac 14|\nabla \beta|^{2}+
    \lambda(\epsilon \mu-1).
  \end{equation}
  Condition \eqref{eq:decayem}  for some $\delta'>\delta$
  implies that 
  $|x|^{2}\bra{x}^{\delta}c(x)\in \ell^{1}L^{\infty}$
  and $c\psi\in Y^*$.
   Lemma 3.3 in \cite{DAncona20} thus yields 
  $\psi=0$ and hence $\phi=0$.

 We next show some decay of $u$.
  Since $u$ is divergence free, 
  as in Proposition~\ref{pro:spectralass1} the equation
  $L(z)u=0$ is  reduced to \eqref{eq:resrewr2} with
  $z^{2}=\lambda$. Defining $(\E,\H)$ as in \eqref{eq:defH},
  with $\sqrt{\lambda}$ in place of $z$,
  we see that $(\E,\H)$ satisfy
  the Maxwell system
  \eqref{eq:maxwellhelm} with $z=\sqrt{\lambda}>0$.
   Since $\Im z=0$, equations \eqref{eq:partid} and  \eqref{eq:defH} imply
   \begin{align*}
    \textstyle
    \int_{|x|=R}
    [|\widehat{x}\times \E|^{2}+|\H|^{2}]d S
    &=
    \textstyle
    \int_{|x|= R}
    |\widehat{x}\times \E+\H|^{2}d S\\
    &=
   \textstyle
    \int_{|x|= R}
    |\mu \sqrt{\lambda}|^{-2}
    |\nabla \times \E- i \mu \sqrt{\lambda}\widehat{x}\times \E|^{2}
    d S.
  \end{align*}
  Multiplying both sides by the (radial) function
  $\bra{x}^{s-1}$ and integrating
  in the radial variable, we arrive at
  \begin{equation*}
    \|\bra{x}^{s-1}\widehat{x}\times \E\|_{L^{2}}
    +\|\bra{x}^{s-1}\H\|_{L^{2}}
    \le C(\mu)\lambda^{-1/2}
    \|\bra{x}^{s-1}
    (\nabla \times \E- i \mu \sqrt{\lambda}
    \widehat{x}\times \E)\|_{L^{2}}.
  \end{equation*}
  Now the radiation estimate \eqref{eq:R0rad} with  $s=\frac 12+$ yields
  \begin{equation}\label{eq:R0z}
    \|\bra{x}^{-\frac 12+}
    (\nabla-i \widehat{x}\sqrt{\lambda})R_0(\lambda)f\|_{L^{2}}         
    \le C    \|\bra{x}^{\frac 12+} f\|_{L^{2}}.
  \end{equation}
  By means of $\E=u/\epsilon$, we write
  \begin{equation}\label{eq:Etou}
    \textstyle
    \nabla \times \E- i \mu \sqrt{\lambda}\widehat{x}\times \E=
    (\nabla\frac1\epsilon)\times  u+
    \frac{i \sqrt{\lambda}}{\epsilon}(1-\mu)\widehat{x}\times u+
    \frac 1 \epsilon(\nabla \times  u-i \sqrt{\lambda}
    \widehat{x}\times u).
  \end{equation}
  We know that $u=R_0(\lambda)f$ for some $f\in Y^{*}$,      
  so that $u\in X$ and $\sqrt{\lambda}u\in Y$ by
  \eqref{eq:R0first}. Condition \eqref{eq:decayem} and \eqref{eq:XY-bounds} 
then imply that the first two terms on the RHS of
$\bra{x}^{-\frac 12+}$ times \eqref{eq:Etou}  are bounded 
by $\|\bra{x}^{\frac 12+}f\|_{L^{2}}$.  
Using also \eqref{eq:R0z},   we derive
  \begin{equation*}
    \|\bra{x}^{-\frac 12+}
    \widehat{x}\times \E\|_{L^{2}}
    +\|\bra{x}^{-\frac 12+}\H\|_{L^{2}}
    \le 
    C(\epsilon,\mu,\lambda)
    \|\bra{x}^{\frac 12+}f\|_{L^{2}}<\infty.
  \end{equation*}
  This proves that 
  $\bra{x}^{-\frac 12+}\H$ and hence 
  $\bra{x}^{-\frac 12+}\nabla \times\E$ are contained in
  $L^{2}$. On the other hand, $\E=\epsilon^{-1} R_0(\lambda)f$ satisfies   
  $\bra{x}^{-1/2-}\E\in L^{2}$ by \eqref{eq:XY-bounds}. 
  The condition $\nabla \cdot(\epsilon\E)=0$
  and the decay of $\nabla \epsilon$
  thus give $\bra{x}^{-\frac 12+}\nabla \cdot\E\in L^{2}$.
  It follows that $\bra{x}^{-\frac 12+}\nabla \E$ is an element of
  $L^{2}$,  which leads to
  $\bra{x}^{-\frac 12+}\nabla u\in L^{2}$ and the estimate
  \begin{equation*}
    \|\bra{x}^{-\frac 12+}\nabla u\|_{L^{2}}
    \le C(\epsilon,\mu,\lambda)
    \|\bra{x}^{\frac 12+}f\|_{L^{2}}<\infty.
  \end{equation*}
 
 Recalling the original equation satisfied by $u$, we have
  \begin{equation*}
    (\Delta+\lambda)u=-g    
  \end{equation*}
  with $g=\lambda(a-1)u+b(x,\partial)u$ and
  $a=\epsilon \mu$.  Since $u,\nabla u\in Y$, 
  the decay assumption \eqref{eq:decayem}
  and \eqref{eq:XY-bounds}  yield
  $\bra{x}^{\frac 12+}g\in L^{2}$. 
  By the radiation estimate \eqref{eq:R0rad}
  for $R_{0}(\lambda)$, we obtain that
  $\bra{x}^{-\frac 12+}(\nabla u-i \sqrt{\lambda}\widehat{x}u)\in L^{2}$ 
  and in conclusion $\bra{x}^{-\frac 12+}u\in L^{2}$.   Note that also $|x|^{-\frac 12+}u$
  belongs to $L^{2}$.

  To prove that $u=0$, we use a Carleman estimate
  from Proposition 5 of \cite{KochTataru06-a} 
  for the special case of the operator $\Delta+\lambda$
  and of a function with $|x|^{-1/2+}u\in L^{2}$.  
  There it is shown that 
  \begin{equation*}
    \textstyle
    \|w\rho u\|_{L^{2}}
    +
    \|\frac{|x|w \rho}{h'(\ln|x|)+|x|}\nabla
       u\|_{L^{2}}
    \lesssim
    \|w(x)\rho^{-1} (\Delta+\lambda)u\|_{L^{2}}
  \end{equation*}
  where $w(x)=e^{h(\ln|x|)}$, $\varepsilon,\tau_1>0$ are small but fixed, and
  \begin{equation*}
    \textstyle
    h'(t)=\tau_1+(\tau e^{t/2}-\tau_1)                 
    \frac{\tau^{2}}{\tau^{2}+\varepsilon e^{t}},
    \qquad
    \rho(|x|)=
    \Big(\frac{h'(\ln|x|)}{|x|^{2}}\Big(1+\frac{h'(\ln|x|)^{2}}{|x|^{2}}
       \Big)\Big)^{1/4}.
  \end{equation*}
  The estimate is uniform in   $\tau\ge \hat{\tau}$ for some $\hat\tau\ge 1$.
  We further set $\varphi(r)=h(\ln r)$ and note 
   \[ \rho(r)= r^{-\frac14}(\varphi'(r)+ \varphi'(r)^3)^{\frac14}.\]
 We can write
  \begin{equation*}
    \textstyle
    \|w\rho^{-1} (\Delta+\lambda) u\|_{L^{2}}
    \lesssim 
    \|w\rho^{-1} (\frac{1}{\epsilon \mu}
     \Delta+\lambda) u\|_{L^{2}}
    +
    \lambda
    \|w\rho^{-1} (\epsilon \mu-1)
     u\|_{L^{2}}
  \end{equation*}
  and also
  \begin{equation*}
    \textstyle
    -\frac{1}{\epsilon \mu}\Delta u=
    \nabla \times \frac 1\mu \nabla \times \frac1{\epsilon}u
    +L.O.T.
    =
    \lambda u+L.O.T.
  \end{equation*}
  Here the lower order terms are bounded by 
  $\bra{x}^{-2-\delta} |u| + \bra{x}^{-1-\delta}|\nabla u|$
  due to \eqref{eq:defbxD} and \eqref{eq:decayem}. We obtain
  \begin{equation*}
    \textstyle
    \|w\rho u\|_{L^{2}}
    +
    \|\frac{|x|w \rho}{h'(\ln|x|)+|x|} 
    \nabla u\|_{L^{2}}
    \lesssim
    \|w\rho^{-1}(L.O.T.)\|_{L^{2}}
    +
    \lambda
    \|w\rho^{-1} (\epsilon \mu-1)
    u\|_{L^{2}}.
  \end{equation*}
  
  To absorb the RHS by the left, we have to prove that the functions
  $m_1=\rho^{-2}\bra{x}^{-2-\delta}$ and
   $m_2=\rho^{-2}(1+\varphi')\bra{x}^{-1-\delta}$ are smaller 
   than a certain constant $\eta>0$  uniformly in $x$
  for a fixed large $\tau$. This will yield $u=0$ and thus the result. 
  Let $r=|x|$. We first observe that
\begin{align*}
 \varphi'(r)
  &= \frac{h'(\ln r)}{r}
  = \frac{\tau^3+\tau_1 \varepsilon r^{\frac12}}{\tau^2 r^{\frac12}+ \varepsilon r^{\frac32}},\\
m_1(x)&\le \bra{x}^{-\frac32-\delta} (\varphi'(r)+\varphi'(r)^3)^{-\frac12}
   \le \bra{x}^{-\frac32-\delta} \varphi'(r)^{-\frac12}
   \le \bra{x}^{-\frac12-\delta}\varphi'(r)^{-\frac12},\\
  m_2(x) &\le \bra{x}^{-\frac12-\delta} 
  \textstyle
     \frac{1+\varphi'(r)}{(\varphi'(r)+ \varphi'(r)^3)^{\frac12}} 
      = \bra{x}^{-\frac12-\delta}\varphi'(r)^{-\frac12}
     \le C \bra{x}^{-\frac12-\delta}
     \frac{\tau r^{\frac14}+ \varepsilon^{\frac12} r^{\frac34}}
        {\tau^{\frac32}+(\tau_1 \varepsilon)^{\frac12} r^{\frac14}}=:m(x).
\end{align*}
Let $r\ge r_0$ for some $r_0\ge 1$ to be fixed below. We compute
  \[ m(x)\le  \frac{r^{-\frac14-\delta}}{\tau^{\frac12}} 
         + \frac{\varepsilon^{\frac12} r^{\frac14-\delta}}{(\tau_1\varepsilon)^{\frac12}r^{\frac14}}
         \lesssim \tau^{-\frac12} + r_0^{-\delta}. \]
uniformly   for  $\tau\ge 1$ and $r\ge r_0$. 
We can fix $r_0\ge 1$ and $\tau_0\ge\hat\tau$ such that $m(x)\le \eta$ for all  $\tau\ge \tau_0$
and $|x|\ge r_0$. Let now $|x|=r\le r_0$. In similar way we  estimate
\[m(x)\le \frac{\tau r^{\frac14}+ \varepsilon^{\frac12} r^{\frac34}}
        {\tau^{\frac32}+(\tau_1 \varepsilon)^{\frac12} r^{\frac14}}
        \le \tau^{-\frac12} r_0^{\frac14}+ \tau^{-\frac32}\varepsilon^{\frac12} r_0^{\frac34}.\] 
Fixing a large $\tau\ge\tau_0$, we conclude that $m(x)\le \eta$  and hence
$m_1(x), m_2(x)\le \eta$ for all $x$.
\end{proof}

It is possible to exclude also a resonance at $z^2=0$, 
provided the first derivatives of the coefficients decay a bit faster. 
We now use that the space dimension is $n=3$
which did not play a role so far.

\begin{proposition}\label{pro:spectralass3}
Assume the real-valued coefficients 
$\epsilon,\mu>0$ satisfy \eqref{eq:choab} and
\begin{equation}\label{eq:decayem1}
  |\epsilon-1|+|\mu-1|+|D^{2}\mu|+|D^{2}\epsilon|
  \lesssim \bra{x}^{-2-\delta},\qquad
 %$ |D^{2}\epsilon|
  % \lesssim \bra{x}^{-\frac52-\delta},\qquad
   |\nabla \epsilon|+|\nabla \mu|
  \lesssim \bra{x}^{-\frac32-\delta}
\end{equation}
for some $\delta\in(0,\frac12)$. 
Let $L(0)u=0$ for some $u=\Delta^{-1} f$ and 
$f\in L^2_{1/2+}$.
Then $u=0$, so that spectral assumption  (S) is true in view of 
Propositions~\ref{pro:spectralass1} and \ref{pro:spectralass2}.
\end{proposition}

\begin{proof}
1) We have $\Delta u=f\in  L^2_{1/2+}\hookrightarrow Y^*$ 
and hence $D^2 u\in L^2$.
Moreover, Lemma~\ref{lem:resetsR0} yields $\nabla u\in Y$ 
and $u\in X$. 
As before, we first consider the function $\phi=\nabla\cdot u$
which now fulfills the equation
\begin{equation*}
  \textstyle
  \nabla \cdot(\frac{1}{\epsilon \mu}\nabla \phi)=0,
  \qquad\text{i.e.,}\qquad 
  \Delta \phi=\nabla \beta \cdot \nabla \phi,
  \qquad
  \beta=\ln(\epsilon \mu).
\end{equation*}
Starting from $\nabla \phi\in L^{2}$, we get 
$\Delta\phi\in L^{2}$ and then $\nabla\phi\in H^2_{loc}$, 
so that $\phi\in C^{1}$. By \eqref{eq:decayem1},
\begin{equation*}
  g=\nabla \beta \cdot \nabla \phi
  \quad\text{satisfies}\quad 
  g\in L^{2}_{\frac 32-\frac\delta2},\quad              
  \nabla g\in L^{2}_{\frac 32+\delta}.
\end{equation*}
Note that this implies $\bra{x}^{\frac 32+\delta}g\in L^{6}$ 
because of
\begin{equation*}
  \|\bra{x}^{\frac 32+\delta}g\|_{L^{6}}
  \lesssim
  \|\nabla(\bra{x}^{\frac 32+\delta}g)\|_{L^{2}}
  \lesssim
  \|\bra{x}^{\frac 12+\delta}g\|_{L^{2}}+
  \|\bra{x}^{\frac 32+\delta}\nabla g\|_{L^{2}}<\infty.
\end{equation*}
Since $\phi=\Delta^{-1}g$,
we can estimate
\begin{equation*}
  \textstyle
  |\phi(x)|
  \lesssim\int \frac{|g(y)|}{|x-y|}dy
  \lesssim
  \|\bra{x}^{\frac 32-\delta/2}g\|_{L^{2}}
  (\int \frac{dy}{\bra{y}^{3- \delta}|x-y|^{2}})^{1/2}
  \lesssim
  |x|^{-1+\delta/2}                  
\end{equation*}
using the standard inequality
\begin{equation*}
  \textstyle
  \int_{\mathbb{R}^n} 
    \frac{dy}{\bra{y}^{a}|x-y|^b}\lesssim |x|^{n-a-b} 
\end{equation*}
for $a,b\in (0,n)$ with $a+b>n$. 
In a similar way we obtain
\begin{equation*}
  \textstyle
  |\nabla\phi(x)|
  \lesssim\int \frac{|g(y)|}{|x-y|^{2}}dy
  \lesssim\|\bra{x}^{\frac 32+\delta}g\|_{L^{6}}
  (\int \frac{dy}{\bra{y}^{(3+2 \delta)3/5}|x-y|^{12/5}})^{5/6}
  \lesssim
  |x|^{-1-\delta}.
\end{equation*}
Together we have proved the decay
\begin{equation}\label{eq:est-psi}
  |\phi(x)|\lesssim \bra{x}^{-1+\delta/2},       
  \qquad
  |\nabla\phi(x)|\lesssim \bra{x}^{-1-\delta}.
\end{equation}
Let $\chi$ be a radial cut-off function equal to 1 on $B(0,1)$ 
and with support in $B(0,2)$. Set $\chi_R(x)=\chi(R^{-1}x)$ 
for $R\ge1$ and $\phi_R= \chi_R \phi$. We compute 
\begin{equation*}
  \nabla (\alpha \nabla \phi_{R})=
  2 \alpha \nabla \chi_{R}\cdot \nabla \phi+
  \nabla \alpha \cdot \nabla \chi_{R}\phi+
  \alpha \phi \Delta \chi_{R},
  \qquad
  \alpha=(\epsilon \mu)^{-1}.
\end{equation*}
Multiply by $\phi_R$ and integrate by parts. 
The above estimates then imply
\begin{equation*}
  \textstyle
  \int_{|x|\le R}\alpha|\nabla \phi_{R}|^{2}
  \lesssim
  \int_{R\le|x|\le2R}
  (
  R^{-1}\bra{x}^{-2-\delta/2}+             
  \bra{x}^{-1-\delta}R^{-1}\bra{x}^{-2+\delta}+
  \bra{x}^{-2+\delta}R^{-2}
  )
\end{equation*}
(we used again \eqref{eq:decayem1}) and we deduce
that for $R\to \infty$
\begin{equation}\label{eq:est-psiR}
  \textstyle
  \int_{|x|\le R}\alpha|\nabla \phi_{R}|^{2}
  \lesssim
  R^{-\delta/2}\to0. 
\end{equation}
We conclude that $\nabla \phi=0$, and by the decay
of $\phi$ we have $\nabla \cdot u=\phi=0$.

\smallskip

2) Using  $\nabla \cdot u=0$, 
as in Proposition~\ref{pro:spectralass1} 
the equation $L(0)u=0$ is  reduced to
\begin{equation}\label{eq:resrewr0}
  \nabla \times
  (\tfrac 1\mu \nabla \times \tfrac1{\epsilon}u) =0
  \qquad\text{or equivalently}\qquad 
  \Delta u=-b(x,\partial)u=:F.
\end{equation}
We can write
\begin{equation}\label{eq:expbu}
  |F|=|bu|\lesssim
  (|\nabla \mu|^{2}+|\nabla \epsilon|^{2}+|D^{2}\epsilon|)|u|+
  (|\nabla \epsilon|+|\nabla \mu|)|\nabla u|.
\end{equation}
We have $\nabla u\in Y \subset L^2_{-1/2-\sigma}$,
$u\in X\subset L^2_{-3/2-\sigma}$  (see \eqref{eq:XY-bounds}),
and by assumption $\Delta u\in L^{2}_{1/2+\sigma}$, for some $\sigma>0$.
Hypothesis \eqref{eq:decayem1} then yields  that
\[\bra{x}^{\frac12+\delta-\sigma} |\Delta u|
   \lesssim \bra{x}^{-\frac32-\sigma} |u| 
    + \bra{x}^{-\frac12-\sigma} |\nabla u| \ \in L^2.\]
    (Actually, we only use condition \eqref{eq:decay0} here.)
 We fix numbers $\frac\delta2>\gamma'>\gamma >\sigma$.  By H\"older's inequality,
 $\Delta u=F$  belongs  $L^{p}$ with
 \[\frac1{p}= \frac12 + \frac{\frac12+\gamma'-\sigma}3<1.\]
 Sobolev's embedding thus implies that 
\[ \nabla u \in L^q \quad \text{with \ }\frac1{q}= \frac13 + \frac{\gamma'-\sigma}3<\frac12,
\qquad u\in L^r \quad  \text{with \ } \frac1{r}=  \frac{\gamma'-\sigma}3.\]
We infer $\bra{x}^{\gamma-\sigma -\frac12}\nabla u\in L^2$ 
and $\bra{x}^{\gamma-\sigma -\frac32}u\in L^2$, so that
 $\bra{x}^{\frac12+\delta+\gamma-\sigma} \Delta u\in L^2$
by  \eqref{eq:decay0}.

We can repeat the argument replacing $\frac12-\sigma+\delta$ by 
$\frac12-\sigma+\gamma+\delta$, and derive that $\bra{x}^{2\gamma-\sigma -\frac12}\nabla u$ 
and $\bra{x}^{2\gamma-\sigma -\frac32}u$ belong to  $L^2$. This procedure can be started 
as long as $\frac12-\sigma+k\gamma+\gamma'<\frac32$. If $\frac12-\sigma+k\gamma+\gamma'\ge 1$
we obtain $\nabla u\in L^2$ where the improvement stops for $\nabla u$. For $u$ we can achieve
$\bra{x}^{-\frac12-\frac\delta2}u \in L^2$.

Assumption \eqref{eq:decayem1} now gives $F \in  L^2_{(3+\delta)/2}$ and 
\begin{equation*}
  \bra{x}^{\frac 32+\delta}|F|\lesssim
  \bra{x}^{-\frac12}|u| + |\nabla u|.
\end{equation*}

The second term at the right belongs to $L^{6}$
since $\|\nabla u\|_{L^{6}}\lesssim\|\Delta u\|_{L^{2}}$
and $\Delta u\in L^{2}_{1/2+}$. For the first term we have
\begin{equation*}
  \|\bra{x}^{-\frac12}u\|_{L^{6}}
  \lesssim
  \|\nabla(\bra{x}^{-\frac12}u)\|_{L^{2}}
  \lesssim
  \|\bra{x}^{-\frac32}u\|_{L^{2}}+
  \|\bra{x}^{-\frac12}\nabla u\|_{L^{2}}
  <\infty
\end{equation*}
by the above decay properties. We infer that
$\bra{x}^{\frac 32+\delta}F\in L^{6}$.
Thus we can repeat the argument in Step 1) and we obtain
\begin{equation}\label{eq:est-u}
  |u(x)|\lesssim \bra{x}^{-1+\delta/2},\qquad  
  |\nabla u(x)|\lesssim \bra{x}^{-1-\delta}.
\end{equation}
 
For $\chi_R$ as above, the map $u_R=\chi_R u$ satisfies
\[\nabla \times (\tfrac1\mu \nabla \times \tfrac1\epsilon u_R)  
 = \nabla \times (\tfrac1a \nabla\chi_R \times u) 
   + \tfrac1\mu \nabla\chi_R \times (\nabla \times \tfrac1\epsilon u)\]
because of \eqref{eq:resrewr0}. 
Similar to \eqref{eq:est-psiR}, we conclude that
\[ \int_{ |x|\le R} |\nabla\times 
  \tfrac1{\epsilon} u_R|^2  \tfrac1\mu dx
  \lesssim R^{-\delta/2}\]             
and hence $\nabla \times \tfrac1{\epsilon} u=0$. The Helmholtz decomposition 
thus yields $ \tfrac1{\epsilon} u=\nabla \varphi$ with the potential
\[ \varphi = \Delta^{ -1} \nabla  \cdot \tfrac1{\epsilon}u 
       =  \Delta^{ -1}(\nabla \tfrac1{\epsilon}\cdot u),\]
where we employed again $\nabla \cdot u =0$. 
Estimates \eqref{eq:decayem1} and \eqref{eq:est-u} imply
\begin{equation*}
  |\varphi(x)|\lesssim 
  \int \frac{dy}{\bra{y}^{\frac52+\frac{\delta}2}|x-y|}\lesssim |x|^{-\frac12 -\frac{\delta}2}, \qquad
  |\nabla \varphi(x)|\lesssim 
   \int \frac{dy}{\bra{y}^{\frac52+\frac\delta2}|x-y|^2}\lesssim |x|^{-\frac32-\frac\delta2}.
\end{equation*}
On the other hand, we have $0=\nabla \cdot u = \nabla\cdot (\epsilon\nabla \varphi)$
which leads to
\[ \int_{|x|\le R} \epsilon|\nabla\varphi|^2dx 
  =\Big|\int_{|x|=R} \epsilon \overline{\varphi}\widehat{x}\cdot\nabla\varphi dS\Big|
  \lesssim R^2 R^{-\frac12 -\frac{\delta}2}R^{-\frac32-\frac\delta2} \lesssim R^{-\delta}.\]
 As $R\to\infty$, we infer that $u=\epsilon \nabla \varphi=0$.
\end{proof}

\section{High frequencies}\label{sec:hi_freq}

In the high frequency regime $|z|\gg1$ we can prove more precise
estimates, with the correct dependence on $z$ of the constants.
This will require a splitting of space variables in two domains:
for large $x$ we can use a Morawetz type estimate since
lower order terms are small there, while 
for bounded $x$ a modified Carleman estimate is sufficient.
This kind of splitting has been used by several authors
(see e.g.~\cite{CardosoCuevasVodev13-a}).

\subsection{Morawetz estimate}\label{sub:mora_esti}

Assume $a(x)>0$ and let
\begin{equation}\label{eq:helmeq}
  f=\Delta v+z^{2}a(x)v,
  \qquad
  z^{2}=\lambda+i \eta.
\end{equation}
Here we may assume $\eta\ge0$ since the case $\eta<0$
is reduced to the first one by conjugating the equation.
Then for all real valued $\phi$ and $\psi$ we have the 
well known identities
\begin{align}\label{eq:id1}
  \begin{split}
    \Re \nabla \cdot\{Q+P\}&= -\textstyle \frac12 \Delta(\Delta
    \psi+\phi)|v|^{2} +
    2 \partial_{j}v \partial_{j}\partial_{k}\psi \partial_{k}\bar{v}
    -\lambda a(x) \phi|v|^{2}
    + \lambda \nabla \psi \cdot \nabla a|v|^{2}
    \\
    & \qquad +\phi|\nabla v|^{2} 
    +2 \eta a(x)\Im[v\nabla \psi \cdot\nabla
    \bar{v}] +\Re ([\Delta,\psi]+\phi)\overline{v}f],
  \end{split}\\
\label{eq:id2}
  \nabla \cdot P&=
  \phi|\nabla v|^{2}
  -z^{2} a(x)|v|^{2}\phi
  +f \overline{v}\phi
  -\textstyle \frac12 \Delta\phi|v|^{2}
  +i\Im (\overline{v}\nabla v \cdot \nabla \phi)
\end{align}
for the functions
\begin{align*}
  Q&=
  \nabla v \  
  [\Delta,\psi]\overline{v}
  -\textstyle\frac12 \nabla \Delta \psi|v|^{2}
  -\nabla\psi |\nabla v|^{2}
  +\nabla\psi a(x)\lambda|v|^{2},\\
  \textstyle 
  P&=
  \nabla v\phi \overline{v}
  -\frac 12\nabla\phi|v|^{2}.
\end{align*}
The quick way to check these identities is by expanding the
derivatives of $P$ and $Q$ at the left hand side.
In these computations we assume that the functions 
are sufficiently regular, and below we also need 
some integrability; these technical assumptions can be 
removed by approximation arguments.
We rewrite \eqref{eq:id1} in the form
\begin{equation}\label{eq:fundid}
  \textstyle
  \Re \nabla\{Q+P\}+I_{\eta}+I_{f}
  =
  I_{\nabla v}+I_{v}
\end{equation}
where
\begin{align*}
  I_{\nabla v}&=
  2 \partial _{j}v
  \, (\partial_{j}\partial_{k}\psi) \,
    \overline{\partial _{k}v}
  + \phi|\nabla v|^{2},
  \qquad
  \textstyle
  I_{v}=
  -\frac12 \Delta(\Delta \psi+\phi)|v|^{2}
  -\lambda a(x) \phi|v|^{2}
  + \lambda \nabla \psi \cdot \nabla a|v|^{2},\\
  I_{f}&=
  -\Re(f\,[\Delta,\psi]\overline{v}
  +f\,\overline{v}\,\phi),
  \qquad
  I_{\eta}=
  -2 \eta a(x)\Im(\overline{v}\,     
  \nabla \psi \cdot \nabla v).
\end{align*}

1) We first deduce from \eqref{eq:id2} some easy estimates, where
we now work in three space dimensions for simplicity.
We take the imaginary part in \eqref{eq:id2} 
and integrate on $\mathbb{R}^{3}$. It follows
\begin{equation}\label{eq:localimv2}
  \textstyle
  \eta \int a(x)|v|^{2}\phi=
  \Im\int f \overline{v}\phi
  +\Im \int\overline{v}\nabla v \cdot \nabla \phi.
\end{equation}
Choosing $\phi=1$, we infer
\begin{equation}\label{eq:firstst}
  \textstyle 
  \eta\|a^{1/2}v\|^{2}=\Im\int f \overline{v}.
\end{equation}
Similarly, the real part of \eqref{eq:id2} yields 
\begin{equation}\label{eq:localrev2}
  \textstyle
  \int \phi|\nabla v|^{2}=
  \lambda\int a|v|^{2}\phi-
  \Re\int f \overline{v}\phi+
  \frac 12\int \Delta \phi|v|^{2}
\end{equation}
and with $\phi=1$
\begin{equation}\label{eq:secest}
  \textstyle
  \|\nabla v\|^{2}=
  \lambda\|a^{1/2}v\|^{2}-\Re\int f \overline{v}.
\end{equation}
In order to estimate the term $I_{\eta}$ 
in \eqref{eq:fundid}, we use
\eqref{eq:firstst} and \eqref{eq:secest} to deduce
\begin{equation*}
  \textstyle
  \int I_{\eta}
  \le
  2 |\eta|\|a^{1/2}\nabla \psi\|_{L^{\infty}}
  \|a^{1/2}v\|_{L^{2}}\|\nabla v\|_{L^{2}}
  \le
  C|\eta|^{1/2}(\int|f\overline{v}|)^{1/2}
  (|\lambda|\|a^{1/2}v\|^{2}_{L^{2}}+\int|f\overline{v}|)
  ^{1/2}
\end{equation*}
with $C=2\|a^{1/2}\nabla \psi\|_{L^{\infty}}$.
Equation \eqref{eq:firstst} then leads to 
\begin{equation*}
  \textstyle
  \int I_\eta \le
  C(\int|f\overline{v}|)^{1/2}
  (|\lambda|\int|f\overline{v}|
  +|\eta|\int|f\overline{v}|)^{1/2},
\end{equation*}
and we arrive at the estimate
\begin{equation}\label{eq:estIep}
  \textstyle
  \int I_{\eta}
  \le
  2\|a^{1/2}\nabla \psi\|_{L^{\infty}}(|\lambda|+|\eta|)^{1/2}
  \|f\overline{v}\|_{L^{1}}.
\end{equation}

2) In \eqref{eq:localimv2} we choose $\phi$ as 
\begin{equation}\label{eq:choicephi}
  \textstyle
  \phi(x)=1 \ \text{ \ if}\ |x|\le R,\quad
  \phi(x)=2-\frac{|x|}{R} \ \text{ \ if}\ R\le|x|\le2R,\quad
  \phi(x)=0 \ \text{ \ if}\ |x|\ge 2R.
\end{equation}
We compute
\begin{equation}\label{eq:interm}
\begin{split}
  \textstyle
  |\eta|
  \int_{|x|\le R}
  a|v|^{2}                
  &\le 
  \textstyle
  \int_{|x|\le2R}
  |f \overline{v}|
  +\frac{1}{R}
  \int_{R\le|x|\le 2R}
  |v||\nabla v|
    \\
  &\lesssim 
  \textstyle
  \int_{|x|\le 2R}|f \overline{v}|+
  R\|v\|_{\dot X}\|\nabla v\|_{\dot Y}.
\end{split}
\end{equation}
Observe that we have used the homogeneous norms \eqref{eq:defdotX}.
Dividing by $R$ and taking the supremum over $R>0$,
we obtain the estimate
\begin{equation}\label{eq:third}
  |\eta|\|a^{1/2}v\|_{\dot Y}^2    
  \lesssim
  \||x|^{-1}f \overline{v}\|_{L^{1}}^2
  +\|v\|_{\dot X}
  +\|\nabla v\|_{\dot Y}^2.
\end{equation}
Next, take $\phi=\frac{1}{|x|\vee R}$ and note that
\begin{equation*}
  \textstyle
  \Delta \phi=-\frac{1}{R^{2}}\delta_{|x|=R}.
\end{equation*}
For this $\phi$, equation \eqref{eq:localrev2} implies 
\begin{equation}\label{eq:intermb}
  \textstyle
  \int \frac{|\nabla v|^{2}-\lambda a |v|^{2}}{|x|\vee R}dx
  + \frac 1{2 R^{2}}\int_{|x|=R}|v|^{2}dS
  \le
  \int \frac{|f \overline{v}|}{|x|\vee R}
  \le\||x|^{-1}f \overline{v}\|_{L^{1}}.
\end{equation}

To proceed, we have to distinguish three cases for $\lambda$.
First, let $\lambda\le 0$. 
We deduce
\begin{equation*}
  \textstyle
  \frac1R\int_{|x|\le R} (|\nabla v|^{2}+a|\lambda||v|^{2})dx   
  + \frac 1{2 R^{2}}\int_{|x|=R}|v|^{2}dS
  \le\||x|^{-1}f \overline{v}\|_{L^{1}},
\end{equation*}
and thus, taking the supremum over $R>0$,
\begin{equation*}%\label{eq:interm1}
  \|\nabla v\|_{\dot Y}^{2}+|\lambda|\|a^{1/2}v\|_{\dot Y}^{2}   
  +\|v\|_{\dot X}^{2}\lesssim
  \||x|^{-1}f \overline{v}\|_{L^{1}}.
\end{equation*}
Combined with \eqref{eq:third}, this relation shows
\begin{equation}\label{eq:lambaneg0}
  \|\nabla v\|_{\dot Y}^{2}
  +\|za^{1/2}v\|_{\dot Y}^{2}  
  +\|v\|_{\dot X}^{2}\lesssim
  \||x|^{-1}f \overline{v}\|_{L^{1}}
  \qquad\text{for}\quad \lambda\le0.
\end{equation}

If $\lambda\ge0$,  with a similar computation,
from \eqref{eq:intermb} we infer the inequality
\begin{equation*}
  \|\nabla v\|_{\dot Y}^{2}
  +\|v\|_{\dot X}^{2}
  \le C_{0}(
  \||x|^{-1}f \overline{v}\|_{L^{1}}
  +|\lambda|\|a^{1/2}v\|_{\dot Y}^{2})   
  \qquad\text{for}\quad \lambda\ge0
\end{equation*}
for a suitable constant $C_{0}>0$.
Let now $\lambda\le (2C_{0})^{-1}|\eta|$. 
As $\|\cdot\|_{\dot Y}\ge\|\cdot\|_{Y}$,
the above estimate, \eqref{eq:third} and 
\eqref{eq:lambaneg0} imply
\begin{equation}\label{eq:lambaneg}
  \|\nabla v\|_{Y}^{2}+\|za^{1/2}v\|_{Y}^{2}  
  +\|v\|_{\dot X}^{2}\lesssim
  \||x|^{-1} f\overline{v}\|_{L^{1}}
  \qquad\text{for}\quad \lambda\le C_{1}|\eta|
\end{equation}
where $C_{1}=(2C_{0})^{-1}$.

Recall now that $f=(\Delta+z^{2}a)v$. In the desired result
we also have a first-order operator  $b=b(x,\partial)$
satisfying \eqref{eq:assb} below, with a sufficiently small
constant $\sigma$. To include this term,
we write $f=(\Delta+z^{2}a+b)v-bv$. We can control
the new term with $bv$ via
\begin{equation*}
  \||x|^{-1} \overline{v}b(x,\partial)v\|_{L^{1}}
  \lesssim
  \sigma\||x|^{-1/2}\bra{x}^{-1-\delta/2}v\|^{2}_{L^{2}}
  +
  \sigma\|\bra{x}^{-(1+\delta)/2}\nabla v\|_{L^{2}}
  \|\bra{x}^{-(1+\delta)/2}|x|^{-1}v\|_{L^{2}}
\end{equation*}
so that (a variant of) \eqref{eq:XY-bounds} shows
\begin{equation*}
  \||x|^{-1} \overline{v}b(x,\partial)v\|_{L^{1}}
  \lesssim
  \sigma\|v\|_{\dot X}^{2}+\sigma\|\nabla v\|_{Y}^{2}.
\end{equation*}
These terms can be absorbed at the left if $\sigma>0$
is small enough. Inserting  $f=(\Delta+z^{2}a+b)v-bv$ in
\eqref{eq:lambaneg},  we conclude
\begin{equation}\label{eq:lambdaep}
  \|\nabla v\|_{Y}^{2}+\|za^{1/2}v\|_{Y}^{2}   
  +\|v\|_{\dot X}^{2}\lesssim
  \||x|^{-1} \overline{v}(\Delta+z^{2}a+b)v\|_{L^{1}}
  \quad\text{for}\quad \lambda\le C_{1}|\eta|.
\end{equation}
Observe that
\[\||x|^{-1}g \overline{v}\|_{L^{1}} 
  \le \|g\|_{Y^{*}}\||x|^{-1} v\|_{\dot{Y}}
  \lesssim \|g\|_{Y^{*}} \|v\|_{\dot X}.\]
 again by a variant of \eqref{eq:XY-bounds}.
Absorbing a $\|v\|_{\dot X}^2$ term, we conclude
\begin{equation}\label{eq:final-mor1}
  \|\nabla v\|_{Y}+\|za^{1/2}v\|_{Y}  
  +\|v\|_{\dot X}\lesssim
  \|(\Delta+z^{2}a+b)v\|_{Y^{*}}
  \qquad\text{for}\quad \lambda\le C_{1}|\eta|.
\end{equation}

3) It remains to consider the case $0\le C_{1}|\eta|\le \lambda$,
for which we need  \eqref{eq:fundid}.
For arbitrary $R>0$, we now employ the functions
\begin{equation}\label{eq:ourpsi}
  \psi=
  \frac{R^{2}+|x|^{2}}{2R}\one{|x|\le R}+|x|\one{|x|>R},
  \qquad
  \phi=-\frac{1}{R}\one{|x|\le R}.
\end{equation}
One calculates
\begin{equation}\label{eq:Apsifi}\begin{split}
  \psi'&=\frac{|x|}{|x|\vee R},
  \qquad
  \psi''=
  \frac1R\one{|x|\le R},\\
  \Delta \psi+\phi&=\frac{2}{|x|\vee R},\qquad
  \textstyle
  \Delta(\Delta \psi+\phi) =
    -\frac{2}{R^{2}}
    \delta_{|x|=R}.
\end{split}\end{equation}
We assume
\begin{equation*}%\label{eq:assax}
  \textstyle
  0<\alpha\le a(x)\le M,\qquad
  \|\bra{x}a'_{-}a^{-1}\|_{\ell^{1}L^{\infty}}\le \frac 14.   
\end{equation*}
Using these relations and  the inequality
\begin{equation*}
  \textstyle
  \int \nabla \psi \cdot \nabla a|v|^{2}
  \ge-\|a'_{-}|v|^{2}\|_{L^{1}}
  \ge-2\|a^{1/2}v\|_{Y}^{2}
  \|\bra{x}a'_{-}a^{-1}\|_{\ell^{1}L^{\infty}},
\end{equation*}
cf.\ \eqref{eq:dyadic}, we derive
\begin{equation}\label{eq:estIv}
  \textstyle
  \sup_{R>0}\int I_{v}
  \ge
  \|v\|_{\dot X}^{2}
  + \frac\lambda2 \|a^{1/2}v\|_{Y}^{2}.
\end{equation}
(Recall \eqref{eq:defdotX}.)
Since $\psi$ is radial, we can write
\begin{equation*}
  \textstyle
  2 \partial _{j}v\, (\partial_{j}\partial_{k}\psi) \,
    \overline{\partial _{k}v}
  =
  2\psi''
  \left|\widehat{x}\cdot \nabla v\right| ^{2}
  +
  2\frac{\psi'}{|x|}
  \left[|\nabla v|^{2}
    -\left|\widehat{x}\cdot \nabla v\right| ^{2}\right]
  \ge
  \frac{2}{R}\one{|x|\le R}|\nabla v|^{2},
\end{equation*}
so that
\begin{equation}\label{eq:estInav}
  \textstyle
  \sup_{R>0}
  \int I_{\nabla v}
  \ge
  \|\nabla v\|_{Y}^{2}.
\end{equation}
Integrating \eqref{eq:fundid}, the lower bounds
\eqref{eq:estIv} and \eqref{eq:estInav} show
\begin{equation}\label{eq:intermest}
  \textstyle
  \|v\|_{\dot X}^{2}+
  \lambda \|a^{1/2}v\|_{Y}^{2}
  +\|\nabla v\|_{Y}^{2}
  \lesssim
  \sup_{R>0}\int I_{f}
  +
  \sup_{R>0}\int I_{\eta}.
\end{equation}
In view of  $|\Delta \psi+\phi|\le2/|x|$ and
$|\nabla \psi|\le1$, we have
\begin{equation*}
  \textstyle
  \int I_{f} \le
  2\||x|^{-1}f\overline{v}\|_{L^{1}}
  +
  2\|f\nabla \overline{v}\|_{L^{1}}.
\end{equation*}
Because of  $0\le C_{1}|\eta|\le \lambda$,
estimate \eqref{eq:estIep} for the above $\psi$
yields
\begin{equation}\label{eq:estIep2}
  \textstyle
  \int I_{\eta}
  \lesssim
  (M \lambda)^{1/2}
  \|f\overline{v}\|_{L^{1}}.
\end{equation}
for every $R>0$. We thus arrive at
\begin{equation*}
  \textstyle
  \|v\|_{\dot X}^{2}+
  \lambda \|a^{1/2}v\|_{Y}^{2}
  +\|\nabla v\|_{Y}^{2}
  \lesssim
  \||x|^{-1}f\overline{v}\|_{L^{1}}
  +
  \|f\nabla \overline{v}\|_{L^{1}}
  +
  (M \lambda)^{1/2}
  \|f\overline{v}\|_{L^{1}} \qquad \text{for}\quad \lambda\ge C_{1}|\eta|.
\end{equation*}
We now use the inequalities
\begin{equation*}
  \||x|^{-1}f \overline{v}\|_{L^{1}}
  \le\|f\|_{Y^{*}}\|v\|_{\dot X},
  \qquad
  \|f \overline{v}\|_{L^{1}}\le\|f\|_{Y^{*}}\|v\|_{Y}
\end{equation*}
as well as $|\eta|\le \frac{1}{C_{1}}\lambda$ and 
$a\ge  \alpha$, to obtain
\begin{equation}\label{eq:lambdalarge}
  \textstyle
  \|v\|_{\dot X}^{2}+ \|zv\|_{Y}^{2}
  +\|\nabla v\|_{Y}^{2}
  \le
  C(M,\alpha)\|f\|_{Y^{*}}^{2}.
\end{equation}

Recall that $f=(\Delta+z^{2}a(x))v$. As in \eqref{eq:lambdaep}, 
in \eqref{eq:lambdalarge} one can now add and subtract the term
$bv$ on the right hand side and absorb error terms for 
a small $\sigma>0$ (w.r.t.\ $\alpha$ and $M$).
We conclude that
\begin{equation}\label{eq:finalsmoo}
  \textstyle
  \|v\|_{\dot X}^{2}+\|zv\|_{Y}^{2}
  +\|\nabla v\|_{Y}^{2}
  \le
  c(\alpha,M)\|(\Delta+z^{2}a+b) v\|_{Y^{*}}^{2}.
\end{equation}

Putting the pieces together, \eqref{eq:final-mor1} 
and \eqref{eq:finalsmoo} we have proved the following
uniform resolvent estimate under a smallness condition 
on the coefficients of $b(x,\partial)$.

\begin{proposition}[]\label{pro:smoothingest}
  Let $z\in \mathbb{C}$ with $\Im z\ge0$.
  Assume that for some $M,\alpha>0$
  \begin{equation}\label{eq:assa1}
    \textstyle
    \alpha\le a(x)\le M,\qquad
    \|\bra{x}a^{-1}a'_{-}\|_{\ell^{1}L^{\infty}}\le \frac 14,
  \end{equation}
  while the first-order operator
  $b(x,\partial)$ satisfies for some $\sigma,\delta>0$
  \begin{equation}\label{eq:assb}
    \textstyle
    |b(x,\partial)v|\le
    \sigma
    (\bra{x}^{-2-\delta}|v|+\bra{x}^{-1-\delta}|\nabla v|).
  \end{equation}
  Let $\sigma$ be sufficiently small with respect to
  $\alpha$ and $M$. We then have
  \begin{equation}\label{eq:gensmoo}
    \textstyle
    \|v\|_{\dot X}
    + \|zv\|_{Y}
    +\|\nabla v\|_{Y}
    \le
    c(\alpha,M,\sigma,\delta)
    \|(\Delta+z^{2}a+b) v\|_{Y^{*}}.
  \end{equation}
\end{proposition}

We now localize estimate \eqref{eq:gensmoo} 
to a region $|x|\ge S$, where $S>1$ is fixed but arbitrary.
We shall assume that condition \eqref{eq:assb}
is satisfied only in this region:
\begin{equation}\label{eq:assbS}
  \textstyle
  |b(x,\partial)v|\le
  \sigma
  (\bra{x}^{-2-\delta}|v|+\bra{x}^{-1-\delta}|\nabla v|)
  \quad\text{for}\quad |x|>S.
\end{equation}
Let $\chi_{0}$ be a real valued radial cutoff equal 
to $0$ for $|x|\le 1$
and equal to 1 for $|x|\ge 2$, with a non negative radial
derivative $\chi_{0}'\ge0$. Set $\chi(x)=\chi_{0}(x/S)$
with the parameter $S>1$. Note that
\begin{equation*}
  |\nabla \chi|\lesssim S^{-1}\one{|x|\sim S},
  \qquad
  |\Delta \chi|\lesssim S^{-2}\one{|x|\sim S}
\end{equation*}
where $|x|\sim S$ is a shortcut for $S\le|x|\le 2S$.
We consider $w=\chi v$, $L=\Delta+az^{2}+b$,
$z^{2}=\lambda+i \eta$, and
\begin{equation*}
  f=Lv,
  \qquad
  g=Lw=\chi f+[L,\chi]v,
  \qquad
  [L,\chi]v=
  2 \nabla \chi \cdot \nabla v+ \Delta \chi v+[b,\chi]v.
\end{equation*}
Assumption \eqref{eq:assbS} yields
\begin{equation*}
  |[b(x,\partial),\chi]v|\le
  |v|\sigma \bra{x}^{-1-\delta}|\nabla \chi|
  \lesssim
  \sigma S^{-2-\delta}|v|\one{|x|\sim S},
\end{equation*}
where we can assume w.l.o.g. $\sigma\le1$. 
We thus obtain
\begin{equation}\label{eq:L-chi}
  |[L,\chi]v|\le cS^{-1}(|v|+|\nabla v|)\one{|x|\sim S}
\end{equation}
for some constant $c=c(\sigma,M)$. We now prove a version of 
\eqref{eq:gensmoo} for $v_S=\chi v$.

\smallskip

1) It is sufficient to consider $\eta\ge0$ as
the case $\eta<0$ follows by conjugation.
First, let $-\infty<\lambda\le C_{1}\eta<+\infty$. 
We can here apply estimate \eqref{eq:lambdaep} with $w$ in
place of $v$, i.e.,
\begin{equation*}
  \textstyle
  \|\nabla w\|_{Y}^{2}+
  \|zw\|_{Y}^{2}+
  \|w\|_{\dot X}^{2}
  \le c\||x|^{-1}\overline{w}Lw\|_{L^{1}}.
\end{equation*}
(Since $w=0$ for $|x|\le S$, it is sufficient
to assume the localized condition \eqref{eq:assbS}
on the lower order terms.)
Writing $Lw=\chi Lv+[L,\chi]v$ and using the estimate
\eqref{eq:L-chi}, we compute
\begin{equation*}
  \||x|^{-1}\overline{w}\chi Lw\|_{L^{1}}\le
  \||x|^{-1}\overline{w}\chi Lv\|_{L^{1}}
  +cS^{-1}\|(|v|+|\nabla v|)v\|_{L^{1}(|x|\sim S)}
\end{equation*}
for some $c=c(\sigma,M)$. The space $\ell^{\infty}L^{\infty}L^{2}$
was introduced after \eqref{eq:dyadic}. Analogously, we define
$\ell^{1}L^{1}L^{2}$ and control its norm by
\[\|u\|_{\ell^{1}L^{1}L^{2}}
:= \sum_{j\ge0} \int_{2^j}^{2^{j+1}}\Big(\int_{|x|=r} |u|^2 dS\Big)^{\frac12}dr
\le  \sum_{j\ge0} 2^{\frac{j}2} \Big(\int_{A_j} |u|^2 dx\Big)^{\frac12}
\lesssim \|u\|_{Y^*},
\]
employing \eqref{eq:dyadic} in the last step.
By means of a variant of \eqref{eq:dyadic}, we thus obtain
\begin{equation*}
  \||x|^{-1}\overline{w}\chi Lv\|_{L^{1}}\le
  \||x|^{-1}w\|_{\ell^{\infty}L^{\infty}L^{2}}
  \|\chi Lv\|_{\ell^{1}L^{1}L^{2}}\lesssim
  \|w\|_{\dot X}\|\chi Lv\|_{Y^{*}},
\end{equation*}
We conclude
\begin{equation}\label{eq:firstfinw}
  \textstyle
  \|\nabla w\|_{Y}^{2}+
  \|zw\|_{Y}^{2}+
  \|w\|_{\dot X}^{2}
  \le c\|\chi Lv\|^{2}_{Y^{*}}
  +cS^{-1}\|(|v|+|\nabla v|)v\|_{L^{1}(|x|\sim S)}
\end{equation}
for $-\infty<\lambda\le C_{1}\eta<+\infty$.

\smallskip

2) Let now  $\lambda\ge C_{1}\eta\ge0$.
For this case we resort to \eqref{eq:intermest}
with $w=\chi v$ in place of $v$ and $h=(\Delta+z^{2}a(x))w$
in place of $f$ which gives
\begin{equation}\label{eq:intermestw}
  \textstyle
  \|w\|_{\dot X}^{2}+
  \lambda \|w\|_{Y}^{2}
  +\|\nabla w\|_{Y}^{2}
  \lesssim
  \sup_{R>0}\int \widetilde{I}_{h}
  +
  \sup_{R>0}\int \widetilde{I}_{\eta}
\end{equation}
where $\psi$ and $\phi$ are given by \eqref{eq:ourpsi} as well as
\begin{equation*}
  \widetilde{I}_{h}=
  -\Re \big((2 \nabla \psi \cdot \nabla\overline{w}
  + \Delta \psi \overline{w}
  +\phi \overline{w})h\big),      
  \qquad
  \widetilde{I}_{\eta}=
  -2 \eta a(x)\Im(\overline{w}\,\nabla \psi \cdot\nabla w). 
\end{equation*}
By \eqref{eq:Apsifi} we have $|\nabla \psi|\le1$ 
and hence
\begin{equation}\label{eq:Iepint}
  \textstyle
  \int \widetilde{I}_{\eta}\le 2M \eta\|\overline{w}\nabla w\|_{L^{1}}
  \lesssim
  2M \eta\|\overline{w}\nabla v\|_{L^{1}}
  +2M \eta S^{-1}\||v|^{2}\|_{L^{1}(|x|\sim S)}
\end{equation}
 for all $R>0$. Next, identities \eqref{eq:firstst} and
\eqref{eq:secest} imply the estimates 
\begin{equation*}
  \textstyle
  \eta\|v\|^{2}\le 
    \alpha^{-1}\|f \overline{v}\|_{L^{1}},
  \qquad
  \|\nabla v\|^{2}\le
  M\lambda\|v\|^{2}+
  \|f \overline{v}\|_{L^{1}}
\end{equation*}
where $\|\cdot\|=\|\cdot\|_{L^{2}}$.
Taking into account $S>1$ and $\lambda\ge C_{1}\eta$, we infer
\begin{align*}
  \textstyle
  \eta S^{-1}
  \|\one{|x|\sim S}v\|^{2}&\le
  \alpha^{-1}\|f \overline{v}\|_{L^{1}},\\
  \eta\|\overline{v}\nabla v\|_{L^{1}}
  &\le
  \eta \lambda^{1/2}\|v\|^{2}+
  \eta \lambda^{-1/2}\|\nabla v\|^{2}\le
  \alpha^{-1}\lambda^{1/2}\|f \overline{v}\|_{L^{1}}+
  M \eta \lambda^{1/2}\|v\|^{2}+
  \eta \lambda^{-1/2}\|f \overline{v}\|_{L^{1}}\\
&\le(\alpha^{-1}+M\alpha^{-1}+C_{1}^{-1})
  \lambda^{1/2}\|f \overline{v}\|_{L^{1}}.
\end{align*}
So \eqref{eq:Iepint} leads to
\begin{equation}\label{eq:Iepquasi}
  \textstyle
  \int \widetilde{I}_{\eta}\le
  C(\alpha,M)(1+\lambda)^{1/2}\|f \overline{v}\|_{L^{1}}\le
  C(\alpha,M,\rho)\|Lv\|_{Y^{*}}^{2}
  +\rho^{2}(1+\lambda)\|v\|_{Y}^{2}
\end{equation}
for all $\rho>0$.
On the other hand, $\widetilde{I}_{h}$ can be written as
\begin{align*}
  \widetilde{I}_{h}&=
  -\Re\big((2 \chi \psi'\overline{v}_{r}               
  +2 \chi'\psi'\overline{v}
  +(\Delta \psi+\phi) \overline{w})
  \cdot
  (2 \chi' v_{r}+\Delta \chi v+\chi Lv-\chi b(x,\partial)v)\big)\\
  &=N+I+II+III+IV
\end{align*}
for the summands
\begin{align*}
  N&=-4 \chi \chi' \psi'|v_{r}|^{2},\\
  I&=
  -\Re(2 \chi'\psi'\overline{v}
  +(\Delta \psi+\phi) \overline{w})
  \cdot
  2 \chi' v_{r},\\
  II&=            
  -\Re\big((2 \nabla \psi \cdot \nabla \overline{w}+  
    (\Delta \psi+\phi) \overline{w})
    \Delta\chi v\big),\\
  III&=-\Re\big((2 \nabla \psi \cdot \nabla \overline{w}+
  (\Delta \psi+\phi) \overline{w})
  \chi Lv\big),\\
  IV&=
  \Re\big((2 \nabla \psi \cdot \nabla \overline{w}+
    (\Delta \psi+\phi) \overline{w})
    (b(x,\partial) w-[b,\chi]v)\big).
\end{align*}
The term $N$ is negative and can be dropped.
For the remaining terms, we recall from  \eqref{eq:Apsifi}  
that $|\nabla \psi|\le1$ and
$|\Delta \psi+\phi|\le 2/\bra{x}$ on the support of $\chi$,
independently of $R>0$. Moreover, the definition of $\chi$ yields
\[\chi'\le cS^{-1}\one{|x|\sim S} \le c\bra{x}^{-1} 
\qquad\text{and}\qquad
|\Delta \chi|\le cS^{-2}\one{|x|\sim S}\le c\bra{x}^{-2}\]
for $S>1$. We thus obtain
\begin{align*}
  I+II&\lesssim S^{-2}|v|(|\nabla v|+|v|)\one{|x|\sim S},\\
  III&\lesssim
  |\chi Lv|(|\nabla w|+\bra{x}^{-1}|w|),\\
  \textstyle
  \|III\|_{L^{1}}&\le
  c'\|\chi Lv\|_{Y^{*}}^{2}+
  \frac{1}{10}\|\bra{x}^{-1}w\|_{Y}^{2}
  +\frac{1}{10}\|\nabla w\|_{Y}^{2}.
\end{align*}
Note that $\|\bra{x}^{-1}w\|_{Y}\le\|w\|_{\dot X}$, cf.\ \eqref{eq:XY-bounds}.
For $IV$ we use \eqref{eq:assbS} and get
\begin{align*}
  IV&\le
  c(|\nabla w|+\bra{x}^{-1}|w|)
  \sigma(\bra{x}^{-1-\delta}|\nabla w|+\bra{x}^{-2-\delta}|w|
  +S^{-2-\delta}|v|\one{|x|\sim S})\\
  \textstyle
  &\le
  c \sigma\bra{x}^{-1-\delta}(|\nabla w|+\bra{x}^{-1}|w|)^{2}
  +c \sigma S^{-2}|v|^{2}\one{|x|\sim S},
\end{align*}
 estimating $[b,\chi]$ as in \eqref{eq:L-chi}.
Invoking \eqref{eq:XY-bounds}, it follows
\begin{equation*}
  \textstyle
  \|IV\|_{L^{1}}\lesssim
   \sigma
  \|w\|_{\dot X}^{2}
  +
   \sigma
  \|\nabla w\|_{Y}^{2}
  +
   \sigma
  S^{-2}\|v^{2}\|_{L^{1}(|x|\sim S)}.
\end{equation*}
Thus if $\sigma$ is small enough we derive
\begin{equation*}
  \textstyle
  \int \widetilde{I}_{h}\le
  c\|\chi Lv\|_{Y^{*}}^{2}
  +
  cS^{-2}\|(|\nabla v|+|v|)v\|_{L^{1}(|x|\sim S)}
  +
  \frac{1}{5}\|w\|_{\dot X}^{2}
  +\frac{1}{5}\|\nabla w\|_{Y}^{2}.
\end{equation*}
Plugging this estimate and \eqref{eq:Iepquasi}
in \eqref{eq:intermestw}
and absorbing some terms at the LHS, we arrive at
\begin{equation}\label{eq:laststep}
  \textstyle
  \|w\|_{\dot X}^{2}+
  \lambda \|w\|_{Y}^{2}
  +\|\nabla w\|_{Y}^{2}
  \le
  C\|Lv\|_{Y^{*}}^{2}
  +CS^{-1}\|(|v|+|\nabla v|)v\|_{L^{1}(|x|\sim S)}
  +\rho^{2}(1+\lambda)\|v\|_{Y}^{2}.
\end{equation}
By the condition $0\le\eta\le C_{1}\lambda$
we can replace $\lambda$ by $|\lambda+i \eta|=|z|^{2}$ 
on the LHS of the inequality.
Combining \eqref{eq:firstfinw} and \eqref{eq:laststep},
we have proved the following uniform resolvent estimate for 
functions localized outside a ball, provided that the lower order
coefficients are small in that region.

\begin{proposition}[]\label{pro:smooestcutoff}
  Let $M,\alpha,\sigma,\delta>0$ and $S>1$.
  Assume that $a(x)$ satisfies \eqref{eq:assa1},
  while the first-order operator
  $b(x,\partial)$ satisfies
  \begin{equation}\label{eq:assb2}
    \textstyle
    |b(x,\partial)v|\le
    \sigma
    (\bra{x}^{-2-\delta}|v|+\bra{x}^{-1-\delta}|\nabla v|)
    \quad\text{for all $|x|\ge S$.}\quad 
  \end{equation}
  Let $\sigma>0$ be sufficiently small with respect to
  $\alpha$ and $M$. Then for all $z\in \mathbb{C}$ the 
  function $v_{S}=v\one{|x|\ge 2S}$ satisfies
  \begin{equation}\label{eq:smoocutoff2}
    \textstyle
    \|v_{S}\|_{X}
    + \|z v_{S}\|_{Y}
    +\|\nabla v_{S}\|_{Y}
    \le
    C
    \|Lv\|_{Y^{*}}
    + \frac{C}{S}\|(|v|+|\nabla v|)v\|_{L^{1}(|x|\sim S)}
    ^{1/2}
    +
    \rho(1+\lambda)^{1/2}\|v\|_{Y}
  \end{equation}
  for all $\rho>0$,
  where $\lambda=\Re z^2$, $L(z)=\Delta+z^{2}a(x)+b(x,\partial)$
  and $C=C(\alpha,M,\sigma,\delta,\rho)$.
\end{proposition}

\subsection{Carleman estimate}\label{sub:carl_esti}

We shall combine estimate \eqref{eq:smoocutoff2} with
a Carleman estimate in a compact subset of $\mathbb{R}^{3}$,
in order to handle coefficients which may be large on a bounded 
subset of  $\mathbb{R}^{3}$.
Our goal is an estimate for (large) frequencies $z^{2}=\lambda+i \eta$ 
belonging to a suitable parabolic region, which is needed for our 
later investigations. In the following computations we consider
functions $u\in H^2$ which decay fast enough, actually the result 
will be applied to functions with compact support.

First, let $\Re z^{2}=\lambda<0$. Integration by parts yields
\begin{equation*}
  \textstyle
  \int|\nabla v|^{2}-\Re z^{2}\int a(x)|v|^{2}=
  -\Re\int((\Delta v+z^{2}av)\overline{v}),
  \qquad
  \Im z^{2}\int a(x)|v|^{2}=
  \Im \int((\Delta v+z^{2}av)\overline{v}).
\end{equation*}
These identities lead to
\begin{align*}
  \|\nabla v\|_{L^2}^{2}+|\lambda|\|a^{1/2}v\|_{L^2}^{2}&\le
  \|\Delta v+z^{2}av\|_{L^2}\|v\|_{L^2}\le
  \tfrac 1{|\lambda|}\|\Delta v+z^{2}av\|_{L^2}^{2}   
  +
  \tfrac {|\lambda|}4 \|a^{1/2}v\|_{L^2}^{2}\\    
  |\eta|\|a^{1/2}v\|^{2}_{L^{2}}&\le               
  \tfrac 1{|\lambda|}\|\Delta v+z^{2}av\|_{L^2}^{2}
  +
  \tfrac {|\lambda|} 4\|a^{1/2}v\|_{L^2}^{2}.
\end{align*}
Using $\alpha\le a(x)$, we obtain the elliptic estimate
\begin{equation*}
  \textstyle
  \|\nabla v\|_{L^{2}}^{2}+ \|zv\|_{L^{2}}^{2}\le
  \|\nabla v\|_{L^{2}}^{2}+ \tfrac1{\alpha }\|za^{1/2}v\|_{L^{2}}^{2}\le
  \frac {C(\alpha)}{|\lambda|}\|(\Delta+z^{2} a)v\|_{L^{2}}^{2}.
\end{equation*}
For any first-order operator $b(x,\partial)$ with bounded coefficients,
the above inequality implies 
\begin{align}
  \|\nabla v\|_{L^{2}}^{2}+
  \|zv\|_{L^{2}}^{2} &\le
  C(\alpha,\lambda_{0})
  \|(\Delta+z^{2} a+b(x,\partial))v\|_{L^{2}}^{2}, \notag\\
  \label{eq:ellest}
  \|v\|_X^2+ \|\nabla v\|_Y^{2}+
  \|zv\|_Y^{2} &\le
  C(\alpha,\lambda_{0})
  \|(\Delta+z^{2} a+b(x,\partial))v\|_{Y^*}^{2}
\end{align}
for all $\Re z^{2}\le -\lambda_{0}(a,b)$,
where $\lambda_{0}(a,b)>0$ depends only on 
$\alpha=\inf a(x)$
and the supremum of the coefficients of $b(x,\partial)$.
In the second line we employ (3.18) from  \cite{CacciafestaDAnconaLuca16}
and \eqref{eq:dyadic}.

We thus focus on the case $\lambda>0$, starting with the main part $\Delta+az^2$.
We use the notations
\begin{equation*}
  \textstyle
  r=|x|,\qquad
  \widehat{x}=\frac{x}{|x|},\qquad
  \partial_{r}=\widehat{x}\cdot \nabla,\qquad
  \Omega=r \nabla -x \partial_{r},\qquad
  \widetilde{\Omega}=\Omega-2\widehat{x}.
\end{equation*}
As above, we denote the radial derivative of a radially
symmetric function with an apex, i.e., $\phi'(r)=\partial_{r}\phi$.
The vector fields $\Omega$ and $\widetilde{\Omega}$
satisfy the relations
\begin{equation*}
  \textstyle
  \widehat{x}\cdot \Omega=0,
  \qquad
  \int_{\mathbb{S}^2}\widetilde{\Omega}f dS=0
\end{equation*}
and we have
\begin{equation*}
  \textstyle
  \Omega^{2}=\Delta_{\mathbb{S}^2},\qquad
  \Delta=\partial_{r}^{2}+\frac{2}{r}\partial_{r}+
    \frac{1}{r^{2}}\Omega^{2},\qquad
  |\nabla v|^{2}=|\partial_{r}v|^{2}+\frac{1}{r^{2}}|\Omega v|^{2}.
\end{equation*}

Fix two radially symmetric, real valued functions
$\phi$ and $\gamma$. We introduce the transformed operator
\begin{equation*}
  Q=
  r e^{\phi}(\Delta+z^{2} a(x))e^{-\phi}r^{-1},
  \qquad
  z^{2}=\lambda+i \eta,\quad
  \lambda,\eta\in \mathbb{R},
\end{equation*}
or more explicitly
\begin{equation*}
  \textstyle
  Q=
  \partial_{r}^{2}+\frac{1}{r^{2}}\Omega^{2}
  +z^{2} a(x)+\phi'^{2}-\phi''-2 \phi' \partial_{r}.
\end{equation*}
It is straighforward  to check
\begin{equation}\label{eq:idbase}
  \begin{split}
  \textstyle
  \partial_{r}\{\gamma A_{0}\}
  +\widetilde{\Omega} \cdot \{\gamma Z_{0}\}
  &=
  2\gamma\cdot\Re[Qv \cdot \overline{v_{r}}]
  +2\gamma \eta a(x)\Im(v \cdot \overline{v_{r}})+
  \\
  &\quad \textstyle
  +(\gamma'+4 \phi'\gamma)|v_{r}|^{2}
  -(\frac{\gamma}{r^{2}})'|\Omega v|^{2}
  +[(\lambda a(x)+\phi'^{2}-\phi'') \gamma]'
    |v|^{2},
  \end{split}
\end{equation}
where $\lambda=\Re z^{2}$, $\eta=\Im z^{2}$ and
\begin{equation*}
  \textstyle
  A_{0}=
    |\partial_{r}v|^{2}
    -\frac{1}{r^{2}}|\Omega v|^{2}
    +(\lambda a(x)+\phi'^{2}-\phi'')|v|^{2},
    \qquad\
  Z_{0}=2\Re(r^{-2}\Omega v\cdot\overline{v_{r}}).
\end{equation*}

\begin{lemma}[]\label{lem:carlemanest}
  Assume $a(x)$ satisfies 
  \begin{equation}\label{eq:assax2}
    0<a(x)\le M,\qquad
    (\nu+r) a'_{-}\le  2a-\nu,
   \end{equation} 
   for some $M>0$ and $\nu\in(0,1]$. Let  $\lambda=\Re z^{2}$, 
   $\eta=\Im z^{2}$, $\nu \lambda\ge 2 \eta^{2}$, and $\tau\ge M^2+4$.  
   Then we have the estimate
  \begin{equation}\label{eq:firstcarl}
    \textstyle
    \|e^{\phi}\bra{x}^{-1/2}\nabla u\|^{2}_{L^{2}}+
    (\Re z^{2}+\tau^{2})\|e^{\phi}u\|^{2}_{L^{2}}
    \le
    10 \nu^{-4} \tau^{-1}         
    \|e^{\phi}(\nu+r)(\Delta+za)u\|^{2}
  \end{equation}
  where $\phi(r)=\tau(r^{2}+r)$.
\end{lemma}

\begin{proof}%[Proof of ...]
  Identity \eqref{eq:idbase} implies
  \begin{equation*}
  \begin{split}
    \partial_{r}&\{\gamma A_{0}\}
    +\widetilde{\Omega} \cdot \{\gamma Z_{0}\}
    + \tau^{-1}\gamma|Qv|^{2}
    \\
    &
    \ge
    \textstyle
    (\gamma'+(4\phi'-\tau)\gamma-M^2\gamma^{\frac 32})|v_{r}|^{2}   
    -(\frac{\gamma}{r^{2}})'|\Omega v|^{2}
    +[(\lambda a(x)+\phi'^{2}-\phi'') \gamma]'|v|^{2}
    -\eta^{2}\gamma^{\frac 12}|v|^{2} .
  \end{split}
  \end{equation*}
  We make the choices
  \begin{equation*}
    \gamma(r)=(\nu+r)^{2},\qquad
    \phi(r)=\tau(r^{2}+r)
  \end{equation*}
  with the parameters $\tau\ge M^2+4$ and $\nu\in(0,1]$. We obtain  
  \begin{equation*}
  \begin{split}
   \ell:= \partial_{r}\{\gamma A_{0}\}
    +\widetilde{\Omega} \cdot \{\gamma Z_{0}\}
    + \tau^{-1}\gamma|Qv|^{2}
    &\ge
    \textstyle
    2 \tau(r+1)\gamma|v_{r}|^{2}
    +\frac{2 \nu(\nu+r)}{r^{3}}|\Omega v|^{2}
    \\
    &\quad
    +(\lambda(a \gamma)'+\tau^{2}(\nu+r)^{3}
    - \eta^{2}\gamma^{\frac 12})|v|^{2}.
    \end{split}
  \end{equation*}
  Condition \eqref{eq:assax2} yields
$(a \gamma)'\ge \nu \gamma^{1/2}$, and 
 $\nu\lambda-\eta^{2}\ge \frac 12 \nu \lambda$
follows from the assumption on $z$.
We can thus continue the previous inequality as
  \begin{equation*}
    \ell\ge
    \textstyle
    2 \tau(r+1)\gamma|v_{r}|^{2}
    +\frac{2 \nu(\nu+r)}{r^{3}}|\Omega v|^{2}
    +(\frac 12 \nu \lambda(\nu+r)+\tau^{2}(\nu+r)^{3})|v|^{2}.
  \end{equation*}
  Now we integrate over the cylinder
  $\Pi=[0,+\infty)\times \mathbb{S}^2$
  and use the notation
  \begin{equation*}
    \textstyle
    \|v\|_{\Pi}^{2}:=
    \int_{0}^{+\infty}\int_{\mathbb{S}^2}
      |v|^{2}d S dr.
  \end{equation*}
  So the above lower bound leads to
  \begin{equation*}
    \textstyle
    \tau
    \|v_{r}\|_{\Pi}^{2}
    +\|\frac{1}{r^{3/2}}\Omega v\|_{\Pi}^{2}
    +\lambda\|v\|_{\Pi}^{2}+ 2\tau^{2}\|(\nu+r)v\|_{\Pi}^{2} 
    \le
    2\tau^{-1}\nu^{-2}\|(\nu+r)Qv\|_{\Pi}^{2}.   
  \end{equation*}
  Setting $v=re^{\phi}u$, we have   
  \begin{align*}
    \|v\|_{\Pi}&=\|e^{\phi}u\|_{L^{2}(\mathbb{R}^{3})},
    \qquad
     \|r^{-3/2}\Omega v\|_{\Pi}=
    \|r^{-3/2}e^{\phi}\Omega u\|_{L^{2}(\mathbb{R}^{3})},\\
    \|(\nu+r)Qv\|_{\Pi}&=
    \|e^{\phi}(\nu+r)(\Delta+z^{2}a(x))u\|_{L^{2}(\mathbb{R}^{3})},
  \end{align*}
 which implies the first partial estimate
  \begin{equation}\label{eq:partial}
  \begin{split}
    \textstyle
    \tau
    \|v_{r}\|_{\Pi}^{2}+
    \|\frac{e^{\phi}}{r^{3/2}}\Omega u\|_{L^{2}}^{2}
    +
    \lambda\|e^{\phi}u\|_{L^{2}}
    &
    +2\tau^{2}                               
    \|(\nu+r)e^{\phi}u\|_{L^{2}}^2
    \\
    &\le
    2\nu^{-2}\tau^{-1}     
    \|e^{\phi}(\nu+r) (\Delta+z^{2}a(x))u\|^{2}_{L^{2}}.
    \end{split}
\end{equation}
  In order to handle the $v_{r}$ term, we first define
  $v=rw$, i.e., $w=e^{\phi}u$. Observe that
  \begin{align*}
    \|v_{r}\|_{\Pi}^{2}&=\textstyle
    \int|w_{r}+\tfrac{w}{r}|^{2}dx=
    \int(|w_{r}|^{2}+\frac{|w|^{2}}{r^{2}}
    +2\Re w_{r}\frac{\overline{w}}{r})dx,\\
    \textstyle
    \int2\Re w_{r}\frac{\overline{w}}{r}dx
    & \textstyle
    =\int \frac{1}{r}\widehat{x}\cdot \nabla|w|^{2}dx
    =
    -\int |w|^{2}\nabla \cdot (\frac{\widehat{x}}{r})dx
    =
    -\int \frac{|w|^{2}}{r^{2}}dx,
  \end{align*}
  and hence
  \begin{equation*}
    \textstyle
    \|v_{r}\|_{\Pi}=\|w_{r}\|_{L^{2}(\mathbb{R}^{3})}
    =
    \|e^{\phi}(u_{r}+\phi' u)\|_{L^{2}(\mathbb{R}^{3})},
    \qquad
    \phi'(r)=\tau(2r+1).
  \end{equation*}
  We deduce
  \begin{equation*}
    \|e^{\phi}u_{r}\|^{2}_{L^{2}(\mathbb{R}^{3})}
    \le
    2\|v_{r}\|^{2}_{\Pi}+
    8\nu^{-2}             
    \tau^{2}\|e^{\phi}(\nu+r)u\|^{2}_{L^{2}(\mathbb{R}^{3})}   .
  \end{equation*}
So  estimate \eqref{eq:partial} gives
  \begin{equation*}
    \textstyle
    \|e^{\phi}u_{r}\|^{2}_{L^{2}}
    +
    \|\frac{e^{\phi}}{r^{3/2}}\Omega u\|^{2}_{L^{2}}
    +
    (\lambda+\tau^{2})\|e^{\phi}u\|^{2}_{L^{2}}
    \le
	10\nu^{-4} \tau^{-1}                  
    \|e^{\phi}(\nu+r)(\Delta+z^{2}a)u\|
      ^{2}_{L^{2}}.
  \end{equation*}
  Inserting $|\nabla u|^{2}=|u_{r}|^{2}+\frac{1}{r^{2}}|\Omega u|^{2}$,
  the assertion \eqref{eq:firstcarl} follows.
\end{proof}

We now take a first-order operator $b(x,\partial)$ 
and let $L=\Delta+z^{2}a+b$. Note that
\begin{equation*}
  \|e^{\phi}(\nu+r)(\Delta+z^{2}a)u\|_{L^{2}}
  \le
  \|e^{\phi}(\nu+r)Lu\|_{L^{2}}
  +
  \|e^{\phi}(\nu+r)bu\|_{L^{2}}.
\end{equation*}
Assume that $u$ has support in the ball $|x|\le K$ for some
$K\ge1$ and that $b(x,\partial)$ satisfies
\begin{equation}\label{eq:assLOT}
  |b(x,\partial)v|\le N(|v|+\bra{x}^{-1/2}|\nabla v|).
\end{equation}
We can then estimate
\begin{equation*}
  \|e^{\phi}(\nu+r)bu\|_{L^{2}}^{2}\le                  
  2N^2(K+1)^2(\|e^{\phi}u\|_{L^{2}}^{2}+
  \|e^{\phi}\bra{x}^{-1/2}\nabla u\|_{L^{2}}^{2}).
\end{equation*}
Taking a large parameter $\tau\ge1$, the lower order terms
on the RHS of \eqref{eq:firstcarl} can be absorbed yielding our Carleman estimate.

\begin{proposition}[]\label{pro:carlemanvc}
  Assume $a(x)$ satisfies \eqref{eq:assax2} and
  $b(x,\partial)$ satisfies \eqref{eq:assLOT}.
  Take $z\in \mathbb{C}$ with $\lambda=\Re z^{2}$,
  $\eta=\Im z^{2}$ and   $\nu \lambda\ge2 \eta^{2}$.
  Let $\phi(r)=\tau(r^{2}+r)$, $u\in H^2$ have
  support in $|x|\le K$ for some $K\ge1$, 
  and $\tau\ge\max\{4+M^2, 80\nu^{-4}N^2(K+1)^2\}$.  
 Then the following estimate holds
  \begin{equation}\label{eq:carl3}
    \textstyle
    \|e^{\phi}\bra{x}^{-1/2}\nabla u\|_{L^{2}}^2+
    (\Re z^2+\tau^{2})
    \|e^{\phi}u\|_{L^{2}}^2
    \le
    40\nu^{-4}\tau^{-1}\|e^{\phi}Lu\|_{L^{2}}^2.   
  \end{equation}
\end{proposition}

Since $\bra{x}^{-1}\ge (2K)^{-1}$ on the support of $u$,
choosing $\tau$ sufficiently large we deduce from
\eqref{eq:carl3} the estimate
\begin{equation}\label{eq:carl4}
  \|u\|_{X}+ \|zu\|_{Y}+\|\nabla u\|_{Y}
  \le c(K,M,N,\nu)\|(\Delta+z^{2}a(x)+b)u\|_{Y^{*}}
\end{equation}
provided $u$ is supported in $|x|\le K$ and
$z^{2}=\lambda+i \eta$ lies in the parabolic region
$\nu \lambda\ge 2 \eta^{2}$.

\section{The complete resolvent estimate}
\label{sec:comp_esti}

We are ready to patch the previous estimates and deduce
a global one valid for all frequencies $z^{2}=\lambda+i \eta$
in a region of the form
\begin{equation}\label{eq:omega}
  \Omega=\Omega(\nu,\lambda_{1})=
  \{\lambda\le -\lambda_{1}/2\}
  \cup\{\lambda^{2}+\eta^{2}\le \lambda_{1}^{2}\}
  \cup
  \{\nu \lambda\ge 2 \eta^{2}\}
\end{equation}
for suitable $\nu,\lambda_{1}>0.$ Recall that
\begin{itemize}
  \item if $\Re z^{2}\le-\lambda_{0}$ for a sufficiently
  large $\lambda_{0}>0$,
  we can use the elliptic estimate \eqref{eq:ellest};
  \item if $z^{2}$ belongs to an arbitrarily large
  (but fixed) ball $|z^{2}|\le \lambda_{1}$, we can use
  Proposition~\ref{pro:fullopres}.
\end{itemize}
Thus to cover the entire region $\Omega(\nu,\lambda_{1})$
it remains to consider
frequencies $z^{2}=\lambda+i \eta$ in the parabolic
region given by $\lambda\ge\lambda_{0}'$ and 
$\nu \lambda\ge 2 \eta^{2}$
for a sufficiently large $\lambda_{0}'>0$.

To this aim, we combine estimates \eqref{eq:smoocutoff2} 
and \eqref{eq:carl3}
for functions vanishing inside, resp. outside, balls.
The assumptions on $a(x)$ are
\begin{equation}\label{eq:assax3}
  \textstyle
  0<\alpha\le a(x)\le M,\qquad
  \|\bra{x}a^{-1}a'_{-}\|_{\ell^{1}L^{\infty}}\le \frac 14,
  \qquad
  (\nu+r) a'_{-}\le  2a-\nu
\end{equation}
for some $\nu\in(0,1]$. For $b(x,\partial)$ we require
\begin{equation}\label{eq:assbpa}
  |b(x,\partial)v|\le
  C_{b}(\bra{x}^{-2-\delta}|v|+\bra{x}^{-1-\delta}|\nabla v|)
\end{equation}
for some $C_{b},\delta>0$, which is the same as 
\eqref{eq:assbpacom}
in Proposition~\ref{pro:fullopres}.
Note that \eqref{eq:assax3} contains
both \eqref{eq:assa1} and \eqref{eq:assax2},
and \eqref{eq:assbpa}
implies \eqref{eq:assLOT} (after possibly increasing $N$).
On the other hand, if we take $S_{0}>1$ sufficiently large
(and possibly decrease $\delta$), we see
that \eqref{eq:assbpa} implies \eqref{eq:assb2} for
$|x|\ge S$ for any $S\ge S_{0}$. From now on, $S_0$ is fixed.
Thus the assumptions of both Propositions \ref{pro:smooestcutoff}
and \ref{pro:carlemanvc} are verified.

Fix a radial cutoff function $\chi_{0}$ such that
$\chi_{0}=0$ for $|x|\le1$ and $\chi_{0}=1$ for $|x|\ge2$.
Set $S= 2S_{0}$ and $\chi(x)=\chi_{0}(S^{-1}x)$. We then decompose 
\begin{equation}\label{eq:start}
  \|u\|_{X}+ \|zu\|_{Y}+\|\nabla u\|_{Y}\le I+II
\end{equation}
with
\begin{align*}
  I&=\|\chi u\|_{X}+ \|\chi zu\|_{Y}
    +\|\nabla (\chi  u)\|_{Y},\\
  II&=\|(1-\chi) u\|_{X}+ \|(1-\chi)z u\|_{Y}
    +\|\nabla ((1-\chi)  u)\|_{Y}.
\end{align*}
Writing $L=L(z)=\Delta+z^{2}a+b$,
we can apply \eqref{eq:carl4} to $II$ since
$(1-\chi)u$ is compactly supported in $|x|\le 2S$, obtaining   
\begin{equation*}
  II \lesssim
  \|L((1-\chi)u)\|_{Y^{*}}
  \lesssim
  \|Lu\|_{Y^{*}}+
  \||u|+|\nabla u|\|_{L^{2}(S\le |x|\le 2S)}  
\end{equation*}
The last term at the right is supported in $|x|\ge 2S_0$. It 
can thus be estimated via \eqref{eq:smoocutoff2} 
in  Proposition \ref{pro:smooestcutoff} with $S_{0}$ 
instead of $S$, and hence
\begin{equation*}
  \textstyle
  II \le
  C\|Lu\|_{Y^{*}}
  + \rho (1+|z|)\|u\|_{Y}
  + C
  \|(|u|+|\nabla u|)u\| _{L^{1}(|x|\sim S_{0})}^{1/2}
\end{equation*}
where $\rho>0$ is arbitrarily small
and $C=C(\alpha,M,\sigma,\delta,\rho,S_{0})$.
We next treat $I$ again using \eqref{eq:smoocutoff2}
with $S_{0}$ instead of $S$ (recall that we have $S= 2S_{0}$), 
which yields
\begin{equation*}
  I \le
  C\|Lu\|_{Y^{*}} + \rho (1+|z|)\|u\|_{Y}   
    + C\|(|u|+|\nabla u|) u\|_{L^{1}(|x|\sim S_{0})}^{1/2}.
\end{equation*}
Summing up, we get
\begin{equation*}
  \textstyle
  I+II\le
  C\|Lu\|_{Y^{*}}+
  \rho (1+|z|)\|u\|_{Y}+
  C\|(|u|+|\nabla u|) u\|_{L^{1}(|x|\sim S_{0})}^{1/2}
\end{equation*}
For every $\rho>0$, the last summand is bounded by
\begin{equation}\label{eq:estcom}
  \|(|u|+|\nabla u|)u\| _{L^{1}(|x|\sim S_{0})}^{1/2}
  \le
  \rho\|\nabla u\|_{Y}+
  C(S_{0},\rho)\|u\|_{Y},
\end{equation}
leading to
\begin{equation*}
  \|u\|_{X}+ \|zu\|_{Y}+\|\nabla u\|_{Y}
  \le
  C\|Lu\|_{Y^{*}}+
  \rho(\|zu\|_{Y}+\|\nabla u\|_{Y})+
  C\|u\|_{Y}.
\end{equation*}
Here $\rho>0$ is arbitrary and 
$C=C(\alpha,M,\sigma,\delta,\rho,S_{0})$.
Taking $\rho=1/2$ and absorbing two terms by the LHS, we infer
\begin{equation*}
  \|u\|_{X}+ \|zu\|_{Y}+\|\nabla u\|_{Y}
  \le
  C\|Lu\|_{Y^{*}}+
  C\|u\|_{Y}.
\end{equation*}
If we assume $|z|\ge 2C$, we can also absorb the last summand
and we obtain
\begin{equation*}
  \|u\|_{X}+ \|zu\|_{Y}+\|\nabla u\|_{Y}
  \le
  C\|Lu\|_{Y^{*}}
\end{equation*}
for all $z$ in the region 
$|z|\ge 2C(\alpha,M,\sigma,\delta,S_{0})$ such that
$\nu \lambda\ge 2 \eta^{2}$.
We now choose a sufficiently large $\lambda_{1}>0$ in the
definition \eqref{eq:omega} of $\Omega$ 
and employ \eqref{eq:ellest} and Proposition~\ref{pro:fullopres}
as indicated after \eqref{eq:omega}.
In this way, the following main resolvent estimate is proved.

\begin{proposition}[]\label{pro:resesttot}
  Assume $a(x)$ and $b(x,\partial)$ satisfy
  \eqref{eq:assax3}, \eqref{eq:assbpa}, 
  $|x|^{2}\bra{x}^{\delta}(a-1)\in L^{\infty},$
  and the spectral assumption (S).
  Then we can find $\lambda_{1}>0$ such that
  for all $z^{2}=\lambda+i \eta\in \mathbb{C}$
  in the region $\Omega=\Omega(\nu,\lambda_{1})$ 
  defined in \eqref{eq:omega}, the operator
  $L(z)=\Delta+z^{2}a(x)+b(x,\partial)$ satisfies
  the estimate
  \begin{equation}\label{eq:largefreqfin}
    \|u\|_{X}+ \|zu\|_{Y}+\|\nabla u\|_{Y}\lesssim 
    \|L(z)u\|_{Y^{*}}
  \end{equation}
  with a constant uniform in $z$.
\end{proposition}

The same proof applies to a matrix operator of the special form
\begin{equation*}
  L(z)=I_{3}\Delta+I_{3}a(x)z^{2}+b(x,\partial).
\end{equation*}

\begin{remark}[]\label{rem:onass}
  The last condition in \eqref{eq:assax3} is implied by
  \begin{equation*}
    \frac{a'_{-}}{a}\le \frac{1}{\nu+r}
  \end{equation*}
  (provided $\nu$ is small enough). Thus we see that
  the following assumption
  \begin{equation}
    a'_{-}(x)\le \nu_{0}a(x)\bra{x}^{-1-\delta}
  \end{equation} \label{eq:assax3-var}
  implies the last two conditions in \eqref{eq:assax3},
  provided $\nu_{0}$ is small enough.
\end{remark}

\section{Smoothing estimates}\label{sec:smoo_esti}

We shall now convert estimate \eqref{eq:largefreqfin}
into a smoothing estimate for the wave equation.
First, we repackage \eqref{eq:largefreqfin} in
a weaker form in terms of weighted $L^{2}$ norms,
in order to apply the Laplace transform. 
Recall from Propositions~\ref{pro:spectralass1}--\ref{pro:spectralass3}
that hypothesis (S) is valid for our Maxwell system, under mild extra 
decay conditions.

\begin{corollary}[]\label{cor:resestweaker}
  Let $L(z)=I_{3}\Delta+I_{3}a(x)z^{2}+b(x,\partial)$ be
  a matrix operator such that
  \begin{enumerate}
    \item 
    $\alpha=\inf a(x)>0$,
    $\bra{x}^{2+\delta}(a-1)\in L^{\infty}$,
    and
    $a_{-}'\le \frac 14 (1-2^{-\delta})^{-1} a \bra{x}^{-1-\delta}$,   
    \item 
    $|b(x,\partial)v|\lesssim
    \bra{x}^{-2-\delta}|v|+\bra{x}^{-1-\delta}|\nabla v|$,
    \item the spectral assumption (S) holds
  \end{enumerate}
   for some $\delta>0$.
  Then there exists $\lambda_{1}>0$ such that
  for any $z$ with $z^{2}\in\Omega(1 \wedge \alpha,\lambda_{1})$ 
  we have
  \begin{equation}\label{eq:globresL2s}
    \textstyle
    \|v\|_{L^{2}_{-3/2-}}
    + \|zv\|_{L^{2}_{-1/2-}}
    +\|\nabla v\|_{L^{2}_{-1/2-}}\lesssim 
    \|L(z)v\|_{L^{2}_{1/2+}}
  \end{equation}
  where we use the notation 
  $\|u\|_{L^{2}_{s}}=\|\bra{x}^{s}u\|_{L^{2}(\mathbb{R}^{3})}$.
\end{corollary}

\begin{proof}%[Proof of ...]
  It is easy to check that assumption (1)
  implies \eqref{eq:assax3}, with $\nu=1 \wedge \alpha$. 
  In view of \eqref{eq:XY-bounds}, estimate \eqref{eq:largefreqfin}
  implies \eqref{eq:globresL2s}.
\end{proof}

Let $u:\mathbb{R}_{t}\times \mathbb{R}_{x}^3\to\mathbb{C}^{3}$ 
 be a function with $u(t,x)=0$ for $t<0$ and such that the maps
 $\partial_t^k u:\mathbb{R}\to H^{2-k}(\mathbb{R}^3)$
are  continuous and grow sub-exponentially for $k=0,1,2$. 
For $z=\alpha+i\beta$ in the upper half plane $\Im z>0$, then 
 the `damped' Fourier transform
\begin{equation*}
  \textstyle
  v(z,\cdot):= 
  \int_{-\infty}^{+\infty}e^{i t z} u(t,\cdot)dt
\end{equation*}
is defined in $L^2(\mathbb{R}^3)$. It satisfies
\begin{equation*}
  \textstyle
  -z^{2}v(z,\cdot)=
  \int_{-\infty}^{+\infty}e^{itz}\partial_{t}^{2}u(t,\cdot)dt,
  \qquad
  (\Delta+b) v(z,\cdot)=
  \int_{-\infty}^{+\infty}e^{itz}(\Delta+b)u(t,\cdot)dt
\end{equation*}
 so that
\begin{equation*}
  \textstyle
  (\Delta+az^{2}+b)v(z,x)=
  \int_{-\infty}^{+\infty}e^{itz}(\Delta+b-a \partial^{2}_{t})
  u(t,x)dt
\end{equation*}
for a.e.\ $x\in \mathbb{R}^3$.
Plancherel's formula thus yields
\begin{equation*}
  \textstyle
 \int|(\Delta+(\alpha+i\beta)^{2}a+b)   
  v(\alpha+i\beta,x)|^{2}d \alpha=
 2\pi  \int e^{-2 \beta t}
  |(\Delta+b-a \partial^{2}_{t})u(t,x)|^{2}dt.
\end{equation*}
We multiply by the weight $\bra{x}^{2s}$ and
integrate also in $x$, obtaining
\begin{equation*}
  \|(\Delta+ (\cdot+i \beta)^{2}a+b)v\|
    _{L^{2}(d\alpha)L^{2}_{s}(\mathbb{R}^{3})}
  \eqsim
  \|e^{-\beta t}
  (\Delta+b-a \partial^{2}_{t})u\|
  _{L^{2}(dt)L^{2}_{s}(\mathbb{R}^{3})}
\end{equation*}
for any $s\in\mathbb{R}$, though the norms could be infinite. 
In a similar way we deduce
\begin{align*}
  \|v(\cdot+i\beta)\|
  _{L^{2}(d\alpha)L^{2}_{s}(\mathbb{R}^{3})}
  &\eqsim
  \|e^{-\beta t}u\|
  _{L^{2}(dt)L^{2}_{s}(\mathbb{R}^{3})},\\
  \|\nabla v(\cdot+i\beta)\|
  _{L^{2}(d\alpha)L^{2}_{s}(\mathbb{R}^{3})}
  &\eqsim
  \|e^{-\beta t}\nabla u\|
  _{L^{2}(dt)L^{2}_{s}(\mathbb{R}^{3})},\\
  |z|\|v(\cdot+i\beta)\|
  _{L^{2}(d\alpha)L^{2}_{s}(\mathbb{R}^{3})}
  &\eqsim
  \|e^{-\beta t}\partial_{t}u\|
  _{L^{2}(dt)L^{2}_{s}(\mathbb{R}^{3})}.
\end{align*}
Note that if $z=\alpha+i\beta$ with $\beta>0$
sufficiently small, then $z^{2}=\lambda+i \eta$
lies in the parabolic region $\Omega$. 
We assume that $G=(\Delta+b-a \partial^{2}_{t})u$ belongs to 
$L^2L^2_{1/2+}$. Estimate \eqref{eq:globresL2s} thus implies that
\begin{equation}\label{eq:partwave}
\begin{split}
  \|e^{- \beta t}u\|_{L^{2}(dt)L^{2}_{-3/2-}}
  +
  \|  &e^{- \beta t}\partial_{t}u\|_{L^{2}(dt)L^{2}_{-1/2-}}
  +
  \|e^{- \beta t}\nabla u\|_{L^{2}(dt)L^{2}_{-1/2-}}\\
  &\lesssim
  \|e^{-\beta t}
  (\Delta+b-a \partial^{2}_{t})u\|
  _{L^{2}(dt)L^{2}_{1/2+}}.
\end{split}
\end{equation}
for  sufficiently small $\beta>0$. (In particular, the involved norms are  
finite.) Here the implicit constant does not depend on $\beta$, so that 
one can let $\beta\to0$ by Fatou's lemma.
As usual, no modification is necessary in the matrix case.

We apply \eqref{eq:partwave} to a solution of the problem
\begin{equation}\label{eq:wavevarF}
  (a \partial_{t}^{2}-\Delta-b(x,\partial))U=G(t,x),
  \qquad
  U(0,x)=
  \partial_tU(0,x)=0.
\end{equation}

\begin{proposition}[]\label{pro:smooWE}
  Let $U(t,x):\mathbb{R}_{t}\times \mathbb{R}_{x}^3\to\mathbb{C}^{3}$ 
  be a solution of the Cauchy problem
  \eqref{eq:wavevarF} subject to the above growth conditions, 
  where $a(x)$ and $b(x,\partial)$ are as in 
  Corollary \ref{cor:resestweaker}
  and $\bra{x}^{1/2+} G\in L^{2}L^{2}$.
  Then the following estimate holds:
  \begin{equation}\label{eq:smooU}
    \|U\|_{L^{2}L^{2}_{-3/2-}}
    +
    \|\partial_{t} U\|_ {L^{2}L^{2}_{-1/2-}}
    +
    \|\nabla U\|_ {L^{2}L^{2}_{-1/2-}}
    \lesssim
    \|G\|_ {L^{2}L^{2}_{1/2+}}.
  \end{equation}
\end{proposition}

\begin{proof}%[Proof of ...]
  Assume $G=0$ for $t\le0$, so that
  $U=0$ for $t\le0$ and we can apply \eqref{eq:partwave}.
  Letting $\beta \downarrow0$ we obtain \eqref{eq:smooU}.
  The same estimate is valid if $G=0$ for $t\ge0$
  (just by time reversal $t\to -t$).
  By linearity, estimate \eqref{eq:smooU} holds
  for arbitrary $G$.
\end{proof}

We next focus on the actual Maxwell equations
\begin{equation*}
  \textstyle
   \partial_{t}^2\E+
   \frac 1 \epsilon\nabla \times\frac 1\mu \nabla \times \E=0,
  \qquad
  \nabla \cdot(\epsilon \E)=0,
\end{equation*}
or equivalently (with $\D= \epsilon\E$)
\begin{equation*}
  \textstyle
   \partial_{t}^2\D+
   \nabla \times\frac 1\mu \nabla \times \frac1{\epsilon}\D=0,
  \qquad
  \nabla \cdot\D=0.
\end{equation*}
Let $\mathcal{H}$ be the Hilbert space
\begin{equation}\label{eq:HS-D}
  \mathcal{H}=\{u\in L^{2}(\mathbb{R}^{3};\mathbb{C}^{3})
  \colon
  \nabla \cdot u=0\}
\end{equation}
endowed with the scalar product 
$(u,v)_{\mathcal{H}}=\int \epsilon^{-1} u \cdot \overline{v}dx$
and the corresponding
norm $\|u\|_{\mathcal{H}}=(u,u)_{\mathcal{H}}^{1/2}$,
and let $H=H(x,\partial)$ be the operator
\begin{equation}\label{eq:opH}
  \textstyle
  H(x,\partial)U=
  \nabla \times\frac 1\mu \nabla \times \frac1{\epsilon}U
\end{equation}
which is selfadjoint and non negative on $\mathcal{H}$. 
The spectral theorem implies that the flow 
$e^{it \sqrt{H}}$ is well defined, bounded and
continuous on $\mathcal{H}$.
Let $U(t,x)$ be the solution to
\begin{equation}\label{eq:maxwHinhom}
 \partial_{t}^2 U+HU=F(t,x),
  \qquad
  U(0,x)=0,
  \qquad
 \partial_{t} U(0,x)=0,
\end{equation}
where $F$ is $\mathcal{H}$--valued and hence
\begin{equation*}
  \nabla \cdot F=0.
\end{equation*}
By Duhamel's formula $U$ is given by
\begin{equation*}
  \textstyle
  U(t,x)=\int_{0}^{t}H^{-1/2}\sin((t-s)\sqrt{H})F(s)ds.
\end{equation*}
We thus have
\begin{equation}\label{eq:duh1}
  \textstyle
    \sqrt{H}U=\int_{0}^{t}\sin((t-s)\sqrt{H})F(s)ds,
    \quad
    \partial_{t}U=\int_{0}^{t}\cos((t-s)\sqrt{H})F(s)ds.
\end{equation}
%so that
%\begin{equation}\label{eq:reprexp}
%  \textstyle
%  \int_{0}^{t}e^{i(t-s)\sqrt{H}}F(s)ds=
%  i\sqrt{H}U+\partial_{t}U.                      
%\end{equation}
Note that also $U$
is divergence free and hence
$HU$ is given by 
\begin{equation*}
  HU=
  \textstyle
  -\frac{1}{\epsilon \mu}\Delta U
  +
  \nabla \times \frac 1\mu \nabla \times \frac1\epsilon U
  - \frac{1}{\epsilon \mu}\nabla \times \nabla \times U.
\end{equation*}
Therefore problem \eqref{eq:maxwHinhom} can be written
in the form \eqref{eq:wavevarF} with the choices
\begin{equation}\label{eq:choiab}
  \textstyle
  a=\epsilon \mu,
  \qquad
  b(x,\partial)U=
  \nabla \times \nabla \times U
  -
  \epsilon \mu
  \nabla \times \frac 1\mu \nabla \times \frac1\epsilon U,
  \qquad 
  G=aF.
\end{equation}

We collect in the next lemma some basic estimates
involving $\sqrt{H}$ and $H$. Observe that a map  $u\in\dot{H}^2$
satisfies $\nabla u\in L^6$ and $|u|\lesssim \bra{x}^{1/2}$. 
Hence, $H:\dot{H}^2 \to L^2$ is bounded if
\begin{equation}\label{eq:decayem0}
|\nabla \epsilon|+|\nabla \mu|\lesssim\bra{x}^{-1-\delta},\qquad
    |D^{2}\epsilon|\lesssim \bra{x}^{-2-\delta}. 
\end{equation}

\begin{lemma}[]\label{lem:riesz}
  Let $H$ be the operator in \eqref{eq:opH}
  and $c,\delta>0$.
  Assume that the coefficients $\epsilon,\mu$ and their first and second derivatives
  are bounded and that $\epsilon,\mu\ge c$.
  We take divergence free functions $f$ from $\dot{H}^1$ in (i) and (ii), 
  from $H^2$ in  (iii), and from $\dot{H^2}$ in (iv).
  Then  the following estimates hold.
  \begin{enumerate}[label=(\roman*)]
    \item If $|\nabla \epsilon|\lesssim \bra{x}^{-1-\delta}$,
    then
    \begin{equation}\label{eq:eqpowH0}
      \|\sqrt {H}f\|_{L^2}\lesssim  \|\nabla f\|_{L^2}.
    \end{equation}
    \item If 
     \begin{equation}\label{eq:extra-decay}
  |\nabla \epsilon|+|\nabla \mu|\lesssim \bra{x}^{-\frac32-\delta}, \qquad
  |D^{2}\epsilon|\lesssim \bra{x}^{-\frac52-\delta},
  \end{equation}
     %$|\nabla \epsilon|\lesssim \langle x\rangle^{-1-\delta}$ 
    %and $|D^{2}\epsilon|\lesssim \bra{x}^{-2-\delta}$, 
    then
    \begin{equation}\label{eq:eqpowH1}
      \|\sqrt {H}f\|_{L^2}\eqsim  \|\nabla f\|_{L^2}.
    \end{equation}
    \item We have
    \begin{equation}\label{eq:mah}
      \|\bra{x}^{-\frac12-\delta}\Delta f\|_{L^{2}}
      \lesssim
      \|\bra{x}^{-\frac12-\delta}H f\|_{L^{2}}
      +
      \|\bra{x}^{-\frac12-\delta} f\|_{L^{2}},
    \end{equation}
    \begin{equation}\label{eq:mah2}
      \|\bra{x}^{-\frac12-\delta}H f\|_{L^{2}}
      \lesssim
      \|\bra{x}^{-\frac12-\delta}\Delta f\|_{L^{2}}
      +
      \|\bra{x}^{-\frac12-\delta} f\|_{L^{2}}.
    \end{equation}
    \item If \eqref{eq:extra-decay} is true,
%    \begin{equation}\label{eq:extra-decay}
%  |\nabla \epsilon|+|\nabla \mu|\lesssim \bra{x}^{-\frac32-\delta}, \qquad
%  |D^{2}\epsilon|\lesssim \bra{x}^{-\frac52-\delta},
% \end{equation}
    then for $\sigma\in(0,\delta)$
    \begin{equation}\label{eq:eqpowH2}
    %  \|\bra{x}^{-\frac 12-\sigma}\nabla f\|_{L^{2}}
    %  \lesssim  \|\bra{x}^{-\frac 12-\sigma}\sqrt{H} f\|_{L^2},
	%\qquad
      \|\bra{x}^{-\frac 12-\sigma}\Delta f\|_{L^{2}}
      \lesssim  \|\bra{x}^{-\frac 12-\sigma}H f\|_{L^2}.
    \end{equation}
  \end{enumerate}
\end{lemma}

\begin{proof}
  \textsc{Proof of  \eqref{eq:eqpowH0}}.
  Integrating by parts we have
  \begin{align}\label{eq:form}% avvicina equazioni centrate
    \|\sqrt {H}f\|_{L^2}^2 &\eqsim 
    \|\sqrt {H}f\|_{\mathcal{H}}^2 = (Hf,f)_{\mathcal{H}}
    \textstyle
      = \int \tfrac1\mu \,|\nabla\times \tfrac1\epsilon f|^2 dx
    \lesssim \|\nabla \times f\|_{L^2}^2 + 
    \|\nabla\epsilon\times f\|_{L^2}^2
  \\
    &\lesssim \|\nabla f\|_{L^2}^2 
    + \|\bra{x}^{-1-\delta}f  \|_{L^2}^2
       \lesssim \|\nabla f\|_{L^2}^2 
       + \|f\|_{L^6}^2  \lesssim \|\nabla f\|_{L^2}^2.  \notag
  \end{align}
  This computation is valid for $f\in H^2\cap \mathcal{H}$, 
  and extends to $f\in \dot H^1\cap \mathcal{H}$ by approximation.

  \textsc{Proof of \eqref{eq:mah} and \eqref{eq:mah2}}. 
  Integration by parts yields
  (all norms are $L^{2}$)
  \begin{equation*}
    \|\bra{x}^{-s}\nabla f\|^{2}
    \le
    \|\bra{x}^{-s}\Delta f\|
    \|\bra{x}^{-s}f\|
    +
    2s\|\bra{x}^{-s}\nabla f\|
    \|\bra{x}^{-s-1}f\|   
  \end{equation*}
  for every $s>0$.
  By the Cauchy--Schwartz inequality and absorbing one term
  at the right, we obtain
  \begin{equation*}
    \|\bra{x}^{-s}\nabla f\|^{2}
    \le
    2\|\bra{x}^{-s}\Delta f\|
    \|\bra{x}^{-s}f\|
    +
    2s^{2}\|\bra{x}^{-s-1}f\|^{2}  
  \end{equation*}
  and then, for arbitrary $\rho>0$,
  \begin{equation*}
    \|\bra{x}^{-s}\nabla f\|
    \le
    \rho\|\bra{x}^{-s}\Delta f\|
    +C(\rho,s)\|\bra{x}^{-s}f\|.
  \end{equation*}
  Since $\nabla \cdot f=0$,
  we can write $Hf$ in the form
  \begin{equation*}
    \epsilon \mu Hf=\Delta f+ b_{1}(x)\cdot \nabla f+b_{0}(x)f
  \end{equation*}
  for suitable bounded matrices $b_{j}$.
  Taking $\rho$ small in the previous estimate, we deduce
  \begin{align*}
    \|\bra{x}^{-s}Hf\|_{L^{2}}
    &\gtrsim
    \|\bra{x}^{-s}\Delta f\|_{L^{2}}-
    \|\bra{x}^{-s}\nabla f\|_{L^{2}}-
    \|\bra{x}^{-s}f\|_{L^{2}}\\
    \textstyle
    &\gtrsim \|\bra{x}^{-s}\Delta f\|_{L^{2}}
    -C\|\bra{x}^{-s}f\|_{L^{2}}.
  \end{align*} 
  The proof of \eqref{eq:mah2} is similar.

  \textsc{Proof of \eqref{eq:eqpowH2}}.
  1) Assume by contradiction the existence of a sequence
  $(f_{n})$ such that
  $\nabla \cdot f_{n}=0$,
  $\|\bra{x}^{-1/2-\sigma}\Delta f_{n}\|_{L^{2}}=1$
  and $\|\bra{x}^{-1/2-\sigma}H f_{n}\|_{L^{2}}\to0$.
  (We may assume that $f_n$ is a Schwartz function
  vanishing at 0 together with its derivatives.)
  By compact embedding we can extract a subsequence
  (again denoted by $f_{n}$) which converges in
  $H^{1}_{loc}$ to a limit function $f$ such that
  $\nabla \cdot f=0$,
  $\|\bra{x}^{-1/2-\sigma}\Delta f\|_{L^{2}}\le1$
  and $\|\bra{x}^{-1/2-\sigma}H f\|_{L^{2}}=0$.

  We first prove that $f\neq0$. 
  Note that for this step it is enough to assume \eqref{eq:decayem0}.
 % $|\nabla \epsilon|+|\nabla \mu| \lesssim \bra{x}^{-1-\delta}$
 % and  $|D^{2}\epsilon| \lesssim \bra{x}^{-2-\delta}$.
  Recalling \eqref{eq:expbu}, for a sufficiently regular $v$ we have
  \begin{equation*}
    |\nabla \times \nabla \times v|\lesssim|Hv|+|b(x,\partial)v|,
    \quad
    |b(x,\partial)v|
    \lesssim
    (|\nabla \mu|^{2}+|\nabla \epsilon|^{2}+|D^{2}\epsilon|)|v|+
    (|\nabla \epsilon|+|\nabla \mu|)|\nabla v|.
  \end{equation*}
   The decay of the coefficients 
  thus implies (all norms are $L^{2}$)
  \begin{equation*}
    \|\bra{x}^{-s}\nabla \times \nabla \times v\|
    \lesssim
    \|\bra{x}^{-s}Hv\|
    +
    \|\bra{x}^{-s-1-\delta}\nabla v\|
    +
    \|\bra{x}^{-s-2-\delta}v\|.
  \end{equation*}
  for $s=\frac 12+\sigma$.
  As in the proof of \eqref{eq:mah},
  we integrate by parts and get
  \begin{equation*}
    \|\bra{x}^{-s-1-\delta}\nabla v\|^{2}
    \le
    \|\bra{x}^{-s-\delta}\Delta v\|
    \|\bra{x}^{-s-2-\delta} v\|
    +C(s,\delta)
    \|\bra{x}^{-s-1-\delta}\nabla v\|
    \|\bra{x}^{-s-2-\delta} v\|.
  \end{equation*}
Cauchy--Schwartz allows us to  absorb a term at the left and hence
  \begin{equation*}
    \|\bra{x}^{-s-1-\delta}\nabla v\|
    \le
    \rho\|\bra{x}^{-s-\delta}\Delta v\|
    +
    C(\rho)\|\bra{x}^{-s-2-\delta} v\|
  \end{equation*}
  where $\rho>0$ can be taken arbitrarily small.
  In conclusion we obtain
  \begin{equation*}
    \|\bra{x}^{-s}\nabla \times \nabla \times v\|
    \le
    \rho\|\bra{x}^{-s-\delta}\Delta v\|
    +
    C\|\bra{x}^{-s}Hv\|
    +
    C\|\bra{x}^{-s-2-\delta}v\|.
  \end{equation*}
  Now take $\chi_{R}(x)=\chi(x/R)$ as above
  and suppose $v=(1-\chi_{R})w$, with $w$ divergence free.
  Then we have
  \begin{equation*}
    \|\bra{x}^{-s}\Delta ((1-\chi_{R})w)\|
    \lesssim
    \|\bra{x}^{-s}\nabla \times \nabla \times((1-\chi_{R})w)\|
    +
    \||w|+|\nabla w|\|_{L^{2}(R\le|x|\le 2R)}.
  \end{equation*}
  We combine this inequality with the previous one. For
  sufficiently small $\rho$  we can absorb a term at the left
  and derive, for $v=(1-\chi_{R})w$ with $\nabla \cdot w=0$,
  \begin{equation}\label{eq:finDH}
    \|\bra{x}^{-s}\Delta v\|
    \lesssim
    C\|\bra{x}^{-s}Hv\|
    +
    C\|\bra{x}^{-s-2-\delta}v\|
    +
    \||w|+|\nabla w|\|_{L^{2}(R\le|x|\le 2R)}.
  \end{equation}

  Now we split
  \begin{equation*}
    1=\|\bra{x}^{-s}\Delta f_{n}\|
    \le 
    \|\bra{x}^{-s}\Delta(\chi_{R} f_{n})\|+
    \|\bra{x}^{-s}\Delta((1-\chi_{R}) f_{n})\|.
  \end{equation*}
For the first term we write
  \begin{equation*}
  \begin{split}
    \|\bra{x}^{-s}\Delta(\chi_{R} f_{n})\|
    &\le 
    \|\bra{x}^{-s}H(\chi_{R} f_{n})\|+
    \|\bra{x}^{-s}b(x,\partial)(\chi_{R} f_{n})\|
    \\
    &\lesssim 
    \|\bra{x}^{-s}H  f_{n}\|+
    \||f_{n}|+|\nabla f_{n}|\|_{L^{2}(|x|\le 2R)}.
    \end{split}
  \end{equation*}
For the second one we use \eqref{eq:finDH}
  with $w=f_{n}$. Summing up, we infer
  \begin{equation}\label{eq:beforeALR}
    1 \lesssim
    \|\bra{x}^{-s}H  f_{n}\|+
    \||f_{n}|+|\nabla f_{n}|\|_{L^{2}(|x|\le 2R)}+
    R^{-\delta}\|\bra{x}^{-s-2}(1-\chi_{R})f_{n}\|.
  \end{equation}
  We recall the Allegretto--Rellich inequality
  \begin{equation}\label{eq:ALR}
    \textstyle
    \||x|^{-a-2}u\|_{L^{2}}\lesssim\||x|^{-a}\Delta u\|_{L^{2}},
    \end{equation}
  if $\inf_{k\in\mathbb{N}_0} |4k(k+1)- (2a+3)(2a+1)|>0$
  which can be applied to functions that vanish in a neighborhood of 0
  and decay fast enough at $\infty$.
  (See Theorem~6.4.1 and Remark 6.4.2  in \cite{BalinskyEvansLewis15-a}.)
  Then the last term can be estimated by
  \begin{equation*}
    CR^{-\delta}
    \|\bra{x}^{-s}\Delta((1-\chi_{R})f_{n})\|_{L^{2}}
    \lesssim
    R^{-\delta}
    \|\bra{x}^{-s}\Delta f_{n}\|_{L^{2}}+
    R^{-\delta}
    \||f_{n}|+|\nabla f_{n}|\|_{L^{2}(|x|\le 2R)}.
  \end{equation*}
  Using $\|\bra{x}^{-s}\Delta f_{n}\|_{L^{2}}=1$ 
  and fixing a large $R$,   we arrive at
  \begin{align*}
    1 &\lesssim
    \|\bra{x}^{-s}H  f_{n}\|+
    \||f_{n}|+|\nabla f_{n}|\|_{L^{2}(|x|\le 2R)}+
    R^{-\delta},\\
     1 &\lesssim
    \|\bra{x}^{-s}H  f_{n}\|+
    \||f_{n}|+|\nabla f_{n}|\|_{L^{2}(|x|\le 2R)}.
  \end{align*}
  Since $\|\bra{x}^{-s}H  f_{n}\|\to0$ and $f_n\to f$ in $H^1_{loc}$ 
  as $n\to \infty$,   we conclude that $f\neq0$. 
  Moreover, $\Delta f$ belongs to $L^{2}_{-1/2-\sigma}$. 

 2)  We finally prove that $f=0$ and so deduce the required contradiction.
 % Indeed, since $\nabla \cdot f=0$ and $Hf=0$ we have
%    $\Delta f=-b(x,\partial)f=:F$ and
%    \begin{equation}\label{eq:estF}
%      |F|\lesssim \bra{x}^{-2-\delta}|f|
%      +\bra{x}^{-1-\delta}|\nabla f|
%    \end{equation}
%   just employing \eqref{eq:decayem0}. 
To this aim we use that   $\Delta f=-b(x,\partial)f=:F$ since $\nabla \cdot f=0$ 
and $Hf=0$. We want to proceed as in Proposition \ref{pro:spectralass3} for which we will
need to establish $F\in L^{2}_{3/2+\delta}$. The above argument shows that 
  the functions $(1-\chi_1)f_n$ are uniformly bounded in $L^{2}_{-5/2-\sigma}$, 
  and so  $f$ belongs to $L^{2}_{-5/2-\sigma}$. Interpolation yields
  $\nabla f\in L^{2}_{-3/2-\sigma}$. We now use  the additional decay
  \eqref{eq:extra-decay} of the coefficients to deduce  that
  \begin{equation}\label{eq:delta-sigma}
  \bra{x}^{\delta-\sigma}|\Delta f| =  \bra{x}^{\delta-\sigma}|F|
  \lesssim \bra{x}^{-\frac52-\sigma}|f| +  \bra{x}^{-\frac32-\sigma}|\nabla f| \ \in  L^2.
  \end{equation}
  We  take numbers $ \delta>\gamma'>\gamma>\sigma$, where we may assume 
  that $\delta\le\frac16$. Setting 
 \[\frac1{p_1}=\frac{\gamma'-\sigma}3 +\frac12<1,\]
  H\"older's inequality implies that $\Delta f \in L^{p_1}$. 
  From Sobolev's embedding we then infer
  \[ \nabla f\in  L^{q_1}, \qquad 
  \frac1{q_1}=\frac1{p_1} - \frac13 = \frac{\gamma'-\sigma}3 +\frac16.\]
  as well as 
  \[\bra{x}^{-\alpha} f\in L^\infty, \qquad \alpha=2-\frac3{p_1}= \frac12- \gamma'+\sigma>0.\]
The functions   $\bra{x}^{\gamma-\sigma-2}f$ and 
  $\bra{x}^{\gamma-\sigma-1}\nabla f$ are thus contained in $L^2$ 
  due to H\"older's inequality.
 Assumption \eqref{eq:extra-decay} then yields
 \[\bra{x}^{\delta+ \gamma-\sigma+\frac12} |\Delta f| 
 \lesssim  \bra{x}^{\gamma-\sigma-2}|f| +  \bra{x}^{\gamma-\sigma-1}|\nabla f| \ \in  L^2. \]
 
  As a result, $\Delta f$ is an element of  $L^{p_2}$ with 
  \[\frac1{p_2}=\frac{\gamma' + \gamma-\sigma+\frac12}3 +\frac12<1.\]
Employing Sobolev's embedding again, we obtain
  \[ \nabla f\in L^{q_2} \text{  \  with \ } 
    \frac1{q_2}= \frac{\gamma'+\gamma-\sigma}3 +\frac13\in \Big(\frac13,\frac12\Big), \qquad
      f\in L^{r_2} \text{  \  with \ }
	\frac1{r_2}= \frac{\gamma'+\gamma-\sigma}3.\]
Hence  $\bra{x}^{2\gamma-\sigma-\frac32}f\in L^2$ and 
  $\bra{x}^{2\gamma-\sigma-\frac12}\nabla f\in L^2$ so that
  $\bra{x}^{\delta +2\gamma-\sigma+1}\Delta f\in L^2$
  by  \eqref{eq:extra-decay}. 
  
 Repeating these steps, we gain another factor $\bra{x}^{\gamma+\frac12}$. 
  We thus obtain $f\in L^2_{-1}$, $\nabla f\in L^2$, and 
  $\Delta f=F\in L^{2}_{3/2+\delta}$, also interpolating with  $\nabla f\in L^{q_2}$.
  Moreover,  we  have $\nabla f\in L^6$ since $\Delta f\in L^2$, as well as
 \[|\nabla (\bra{x}^{-1} f)|\lesssim  \bra{x}^{-1} |\nabla f|+ \bra{x}^{-2} |f|  \ \in  L^2\]
 so that $\bra{x}^{-1} f\in L^6$. Using \eqref{eq:extra-decay} once more, we derive
 \[ \bra{x}^{\frac32 + \delta} |\Delta f| \lesssim  \bra{x}^{-1} |f| + |\nabla f|\ \in  L^6.\]
 The argument used in
  Proposition \ref{pro:spectralass3} now leads to $f=0$.

  \textsc{Proof of \eqref{eq:eqpowH1}}.
   The converse inequality is proved  by interpolation with the inequality 
   \begin{equation}\label{eq:Delta-H}
   \|\Delta f\|_{L^{2}}
      \lesssim  \|H f\|_{L^2}.
   \end{equation}
   for divergence free $f\in\dot{H}^2$. One shows \eqref{eq:Delta-H}
   by contradiction, assuming the existence of  $(f_{n})$ with
  $\nabla \cdot f_{n}=0$,   $\|\Delta f_{n}\|_{L^{2}}=1$
  and $\|H f_{n}\|_{L^{2}}\to0$. Here we may assume that $f_n$ is a 
  Schwartz function vanishing at 0 together with its derivatives.
  
   We proceed as the proof of \eqref{eq:eqpowH2} above.
   As in step 1)  of this proof we deduce that $f_n$ tends 
in $H^1_{loc}$  to a function  $f\neq 0$ with  $\nabla \cdot f=0$, $Hf=0$, and
$\|\Delta f\|_{L^{2}}\le1$. One only has to modify the last summand in \eqref{eq:beforeALR} 
with $s=0$ to
$R^{-\frac\delta2}\|\bra{x}^{-\frac\delta2-2}(1-\chi_{R})f_{n}\|$,
in order to use the Rellich-Allegretto inequality \eqref{eq:ALR} with $a=\delta/2$.
As in step 2) of the proof of \eqref{eq:eqpowH2}, one also sees that
$f$ belongs to $L^{2}_{-2-\sigma}$ and  $\nabla f$ to $L^{2}_{-1-\sigma}$ for some $\sigma>0$.
Just using assumption \eqref{eq:decayem0} we deduce \eqref{eq:delta-sigma}
and can proceed as before to conclude the contradiction $f=0$.

\end{proof}

We are now in position to apply Corollary \ref{cor:resestweaker},
combined with Propositions~\ref{pro:spectralass1}--\ref{pro:spectralass3}
concerning hypothesis (S).

\begin{proposition}[]\label{pro:estMaxw}
  Let $\epsilon,\mu:\mathbb{R}^3\to \mathbb{R}$ and
  $b(x,\partial)$ as in \eqref{eq:choiab}.
  For a $\delta>0$ assume that
  \begin{enumerate}
    \item 
    $\inf\epsilon \mu>0$
    and
    $(\epsilon \mu)'_{-}\le \frac 14 (1-2^{-\delta})^{-1}\epsilon \mu      
     \bra{x}^{-1-\delta}$,        
    \item 
    $|\nabla\epsilon|+|\nabla \mu|\lesssim \bra{x}^{-1-\delta}$
    and
    $|\epsilon-1|+|\mu-1|+|D^{2}\epsilon|+|D^{2}\mu|
      \lesssim \bra{x}^{-2-\delta}$,
      \item 
    either 0 is not a resonance (i.e., if $(\Delta+b)u=0$ with
    $u=\Delta^{-1}f$ for some 
    $\bra{x}^{\frac 12+\delta} f\in L^{2}$ then $u=0$) or we strengthen (2) by
    \begin{equation}\label{eq:decay-ep-mu}
        |\nabla \epsilon|+|\nabla \mu|\lesssim
        \bra{x}^{-\frac32-\delta}.
   %     \qquad  |D^{2}\epsilon| \lesssim \bra{x}^{-\frac52-\delta}.
    \end{equation}
  \end{enumerate}
  Then for any divergence free forcing term
  $F\in L^2L^2_{1/2+}$ we have
  \begin{equation}\label{eq:estMaxwinh}
    \textstyle
    \|\bra{x}^{-\frac 12-}
    \int_{0}^{t}\cos((t-s)\sqrt{H})F(s)ds\|_{L^{2}(dt)L^{2}}
    \lesssim
    \|\bra{x}^{1/2+}F\|_{L^{2}L^{2}},
  \end{equation}
  \begin{equation}\label{eq:estMaxwinhunt}
    \textstyle
    \|\bra{x}^{-\frac 12-}
    \int_{\mathbb{R}}\cos((t-s)\sqrt{H})F(s)ds\|_{L^{2}(dt)L^{2}}
    \lesssim
    \|\bra{x}^{1/2+}F\|_{L^{2}L^{2}}.
  \end{equation}
\end{proposition}

\begin{proof}%[Proof of ...]
  We check the hypotheses in Corollary \ref{cor:resestweaker}
on the coefficients.
  Assumption (1) in the corollary follows from conditions
  (1) and (2) here, and assumption (2) in the corollary is an easy
  consequence of (2) here (compare with \eqref{eq:decayem}).
  The spectral assumption (S) reduces to (3) here in view of
  Propositions~\ref{pro:spectralass1}--\ref{pro:spectralass3}.
  
We can approximate $F$ in $L^2L^2_{1/2+}$ by divergence free functions 
$F_n\in C^1 L^2_{1/2+}$ with compact support in time. The corresponding
solutions $U_n$ to \eqref{eq:maxwHinhom} then satisfy the conditions
of Proposition~\ref{pro:smooWE}. By density we can thus assume that 
$U$ has the required regularity and growth.  Since for divergence free solutions
  the wave equations $\epsilon\mu U_{tt}=(\Delta+b)U+\epsilon\mu F$ and
  $U_{tt}=HU+F$ coincide,  we can apply estimate 
  \eqref{eq:smooU}. Hence the solution $U$ of 
  problem \eqref{eq:maxwHinhom} satisfies
  \begin{equation*}
    \|\bra{x}^{-\frac 32-}U\|_{L^{2}L^{2}}+
    \|\bra{x}^{-\frac 12-}\nabla U\|_{L^{2}L^{2}}+
    \|\bra{x}^{-\frac 12-}\partial_{t} U\|_{L^{2}L^{2}}
    \lesssim
    \|\bra{x}^{\frac 12+}F\|_{L^{2}L^{2}}.
  \end{equation*}
Because of \eqref{eq:duh1}, we obtain
  \begin{equation*}
    \textstyle
    \|\bra{x}^{-\frac 12-}
    \int_{0}^{t}\cos((t-s)\sqrt{H})F(s)ds\|_{L^{2}L^{2}}
    \lesssim
    \|\bra{x}^{\frac 12+}F\|_{L^{2}L^{2}}
  \end{equation*}
  i.e., \eqref{eq:estMaxwinh}.
  By time translation invariance, we can replace the
  integral $\int_{0}^{t}$ with $\int_{T}^{t}$
  for an arbitrary $T<0$, where the implicit constant
  does not dependent of $T$,  and hence with
  $\int_{-\infty}^{t}$ by letting $T\to-\infty$.
  By time reversal, the estimate is then valid also with
  the integral $\int_{-\infty}^{t}$ replaced by
  $\int_{t}^{+\infty}$. Summing the two, we arrive at
  \eqref{eq:estMaxwinhunt}.
\end{proof}

By a modification of the standard $TT^{*}$ method we obtain
the corresponding homogeneous estimates.

\begin{proposition}[]\label{pro:estMaxwhom}
  Under the assumptions of Proposition \ref{pro:estMaxw}
  we have for any divergence free data $f$ in the respective spaces
  \begin{align}\label{eq:estMaxwhom}
    \|\bra{x}^{-1/2-}e^{it \sqrt{H}}f\|_{L^{2}(dt)L^{2}}
    &\lesssim\|f\|_{L^{2}},
    \quad
    \|\bra{x}^{-1/2-}\partial_{t} e^{it \sqrt{H}}f\|_{L^{2}(dt)L^{2}}
    \lesssim\|f\|_{\dot H^{1}},\\
  \label{eq:estMaxwhomdx}
    \|\bra{x}^{-1/2-}\nabla e^{it \sqrt{H}}f\|
      _{L^{2}(dt)L^{2}}
    &\lesssim\|f\|_{H^{1}},
    \quad
    \|\bra{x}^{-1/2-}\Delta e^{it \sqrt{H}}f\|
      _{L^{2}(dt)L^{2}}
    \lesssim\|f\|_{H^{2}}.
  \end{align}
\end{proposition}

\begin{proof}%[Proof of ...]
 1)  By the change of variable $s\to -s$,
  estimate \eqref{eq:estMaxwinhunt} implies
  \begin{equation*}
    \textstyle
    \|\bra{x}^{-1/2-}
    \int_{\mathbb{R}}\cos((t+s)\sqrt{H})F(s)ds\|_{L^{2}(dt)L^{2}}
    \lesssim
    \|\bra{x}^{1/2+}F\|_{L^{2}L^{2}}.
  \end{equation*}
  Summing the two inequalities, we get
  \begin{equation}\label{eq:coscos}
    \textstyle
    \|\bra{x}^{-1/2-}
    \int_{\mathbb{R}}
    \cos(t\sqrt{H}) \cos(s\sqrt{H})
    F(s)ds\|_{L^{2}(dt)L^{2}}
    \lesssim
    \|\bra{x}^{1/2+}F\|_{L^{2}L^{2}}.
  \end{equation}
  By subtraction we obtain this estimate with 
  $\sin(t\sqrt{H}) \sin(s\sqrt{H})$ instead of
  $\cos(t\sqrt{H}) \cos(s\sqrt{H})$.
  
To exploit the above bounds, we consider the duality
  \begin{equation*}
    ((F,G))=
    \textstyle
    \int_{\mathbb{R}}(F(t),G(t))_{\mathcal{H}}dt=
    \int_{\mathbb{R}}\int_{\mathbb{R}^{3}}
    F(t,x) \overline{G}(t,x)\epsilon^{-1}dt
  \end{equation*}
  and the weighted space $Z^*$ of divergence free $F$ with finite norm
  \begin{equation*}
    \|F\|_{Z^*}=\|\bra{x}^{1/2+}F\|_{L^{2}L^{2}}.
  \end{equation*}
  Define $(Tf)(t)=\cos(t \sqrt{H})f$ for $f\in \mathcal{H}$ 
  and $t\in \mathbb{R}$, as well as 
  $T^*F=\int_{\mathbb{R}}\cos(s \sqrt{H}) F(s)\,ds\in  \mathcal{H}$
  at first for $F\in Z^*$ with compact support in time.
   (Recall $H$ is selfadjoint for the $\epsilon$--product.) We obtain
   \begin{equation*}
    TT^{*}F=
    \textstyle
    \int \cos(t\sqrt{H}) \cos(s\sqrt{H})F(s)ds,
  \end{equation*} 
   as well as $((F, Tf))=(T^*F,f)_\mathcal{H}$ and 
   $((TT^*F,G))= (T^*F,T^*G)_\mathcal{H}$ for such $f$, $F$ and $G$.
 Estimate \eqref{eq:coscos} yields
 \begin{equation*}
    |(( T^*F, T^*G))|= |((\bra{x}^{-1/2-} TT^{*}F,\bra{x}^{1/2+}G))|\lesssim\|F\|_{Z^*}\|G\|_{Z^*}.
  \end{equation*}
   Taking $F=G$ we deduce
  \begin{equation*}
    \textstyle
    \|T^{*}F\|^{2}_{\mathcal{H}}\lesssim\|F\|_{Z^*}^2=
    \|\bra{x}^{1/2+}F\|_{L^{2}L^{2}}^{2}.
  \end{equation*}
  By density,  the operator $T^{*}:Z^*\to \mathcal{H}$ is bounded.
  Duality implies 
  \begin{align*}
     %&=\|\bra{x}^{-1/2-}Tf\|_{L^{2}L^{2}}
   \|\bra{x}^{-1/2-}\cos(t \sqrt{H})f\|_{L^{2}L^{2}}
   =\|Tf\|_{Z} =\sup_{\|F\|_{Z^*}\le1} |((F, Tf))| 
   = \sup_{\|F\|_{Z^*}\le1}(T^*F,f)_\mathcal{H}
    \lesssim\|f\|_{\mathcal{H}} 
  \end{align*}
   where the suprema are taken over $F$ with compact support in time.
    A similar argument gives
  \begin{equation*}
    \|\bra{x}^{-1/2-}\sin(t \sqrt{H})f\|_{L^{2}L^{2}}
    \lesssim\|f\|_{\mathcal{H}}.
\end{equation*}
Combining the two estimates we get the first of \eqref{eq:estMaxwhom}.
  By the estimate already proved and \eqref{eq:eqpowH0}
  we have
  \begin{equation*}
    \|\bra{x}^{-1/2-}\partial_{t}e^{it \sqrt{H}}f\|_{L^{2}L^{2}}
    =
    \|\bra{x}^{-1/2-}e^{it \sqrt{H}}\sqrt{H}  f\|_{L^{2}L^{2}}
    \lesssim\|\sqrt{H} f\|_{L^2}
    \lesssim\|\nabla f\|_{L^{2}}
  \end{equation*}
  and this concludes the proof of \eqref{eq:estMaxwhom}.

  2) Applying \eqref{eq:mah} to $u=e^{it \sqrt{H}}f$ 
  and using the first inequality in \eqref{eq:estMaxwhom} we have
  \begin{equation*}
  \begin{split}
    \|\bra{x}^{-1/2-}\Delta u\|_{L^{2} L^{2}}
    \lesssim &
    \|\bra{x}^{-1/2-}H u\|_{L^{2}L^{2}}
    +
    \|\bra{x}^{-1/2-} u\|_{L^{2}L^{2}}
    \\
    =&\|\bra{x}^{-1/2-}e^{it \sqrt{H}}Hf\|_{L^{2}L^{2}}
    +\|\bra{x}^{-1/2-}e^{it \sqrt{H}}f\|_{L^{2}L^{2}}
    \\
    \lesssim & \|Hf\|_{L^{2}}+\|f\|_{L^{2}}
    \end{split}
  \end{equation*}
  and by \eqref{eq:mah2} we obtain
  the second part of \eqref{eq:estMaxwhomdx}.
  The first estimate in \eqref{eq:estMaxwhomdx} then follows
  by complex interpolation of weighted Sobolev spaces
  with the first inequality in \eqref{eq:estMaxwhom}.
\end{proof}

Under more restrictive decay conditions on the coefficients,
we can prove a variant of \eqref{eq:estMaxwhomdx} 
in homogeneous norms that is needed below.

\begin{proposition}[]\label{pro:homhom}
  Let $\epsilon,\mu:\mathbb{R}^3\to \mathbb{R}$ and
  $b(x,\partial)$ as in \eqref{eq:choiab}.
  For some $\delta>0$ assume that
  \begin{enumerate}
    \item 
    $\inf\epsilon \mu>0$
   \  and \
    $(\epsilon \mu)'_{-}\le \frac 14 (1-2^{-\delta})^{-1}\epsilon \mu
     \bra{x}^{-1-\delta}$,   
    \item 
    $|\epsilon-1|+|\mu-1|+|D^{2}\mu|\lesssim \bra{x}^{-2-\delta}$, \
    $|\nabla\epsilon|+|\nabla \mu|\lesssim 
        \bra{x}^{-\frac32-\delta}$,
   \ and \
    $|D^{2}\epsilon|
      \lesssim \bra{x}^{-\frac52-\delta}$.
  \end{enumerate}
  Let $f$  be divergence free.
  Then in addition to \eqref{eq:estMaxwhom} and
  \eqref{eq:estMaxwhomdx},  we have the estimates
  \begin{equation}\label{eq:homhom}
    \|\bra{x}^{-1/2-}\nabla e^{it \sqrt{H}}f\|
      _{L^{2}(dt)L^{2}}
    \lesssim\|\nabla f\|_{L^2},
    \quad
    \|\bra{x}^{-1/2-}\Delta e^{it \sqrt{H}}f\|
      _{L^{2}(dt)L^{2}}
    \lesssim\|Hf\|_{L^{2}},
  \end{equation}
\end{proposition}

\begin{proof}%[Proof of ...]
  Note that under these assumptions, the spectral condition
  (S) is satisfied due to Propositions~\ref{pro:spectralass1}--\ref{pro:spectralass3}.
  By \eqref{eq:eqpowH2} in Lemma \ref{lem:riesz}
  combined with \eqref{eq:estMaxwhom} we have
  \begin{equation*}
    \|\bra{x}^{-\frac 12-}\Delta e^{it \sqrt{H}}f\|_{L^{2}L^{2}}
    \lesssim
    \|\bra{x}^{-\frac 12-}H e^{it \sqrt{H}}f\|_{L^{2}L^{2}}=
    \|\bra{x}^{-\frac 12-} e^{it \sqrt{H}}Hf\|_{L^{2}L^{2}}
    \lesssim\|Hf\|_{L^{2}}.
  \end{equation*}
  This proves the second estimate in \eqref{eq:homhom}.
  Complex interpolation with  \eqref{eq:estMaxwhom} then yields
  \begin{equation*}
    \|\bra{x}^{-\frac 12-}\nabla e^{it \sqrt{H}}f\|_{L^{2}L^{2}}
    \lesssim
    \|\sqrt{H}f\|_{L^{2}},
  \end{equation*}
  and recalling \eqref{eq:eqpowH0} we obtain also the first
  estimate.
\end{proof}

\section{Strichartz estimates}\label{sec:stri}

We first deduce from the results in
\cite{MetcalfeTataru12} a conditional Strichartz
estimate for the wave equation 
\begin{equation}\label{eq:wavevarcompl}
  (a \partial_{t}^{2}-\Delta-b(x,\partial))U=F,
  \qquad
  U(0,\cdot)=U_{0},
  \qquad
  \partial_{t}U(0,\cdot)=U_{1}.
\end{equation}
Recall that a couple $(p,q)\in[2,\infty]^{2}$ is
\emph{wave admissible} in dimension $n=3$ if
\begin{equation*}
  \textstyle
  \frac 1p+\frac 1q=\frac 12,
  \qquad
  p\in[2,\infty],
  \qquad
  q\in[2,\infty).
\end{equation*}
We often use that multiplication by $\epsilon$, $\epsilon^{-1}$, 
$\mu$ or $\mu^{-1}$ is continuous on the Strichartz spaces, 
as shown in the next lemma.
\begin{lemma}\label{lem:mult}
Let $m\in W^{1,\infty}$ be positive with $\frac1m\in L^\infty$ and 
$|\nabla m|\lesssim\bra{x}^{-1-}$ and let $(p,q)$ be wave admissible. 
Then the operator $f\mapsto mf$ is bounded on spaces
$\dot{H}^{2/p}_{q'}$, $\dot{H}^{-2/p}_{q}$ and $\dot{H}^{1-2/p}_{q}$.
In addition, assume that  $|D^2 m|\lesssim\bra{x}^{-2-}$. Then $f\mapsto mf$
is also bounded on $\dot{H}^{1+2/p}_{q'}$.
\end{lemma}
\begin{proof}
Observe  that   $|\,|D|^{2/p}m|\lesssim\langle x\rangle^{-2/p -}$ 
  by interpolation so that  $|D|^{2/p}m$  belongs to $L^{3p/2}$. 
  Let $\frac1\kappa= \frac1{q'}-\frac2{3p} \in[ \tfrac12,\tfrac{1}{q'}]$.  
  Sobolev's inequality yields $\dot H^{2/p}_{q'}\hookrightarrow L^\kappa$. 
  We thus obtain
  \[\|\,|D|^{2/p}(m f)\|_{L^{q'}}
    \lesssim \|m\|_{L^\infty}\,\|\,|D|^{2/p}f\|_{L^{q'}} 
     + \||D|^{2/p}m\|_{L^{3p/2}}\,\|f\|_{L^\kappa}
    \lesssim \|f\|_{\dot H^{2/p}_{q'}}\,.\]

    Duality then implies the boundedness on $\dot{H}^{-2/p}_{q}$.
We further have $|\,|D|^{1-2/p}m|\lesssim \langle x\rangle^{-2/q -}$ 
and  $\dot H^{1-2/p}_{q}\hookrightarrow L^{3q}$. As above, 
one now shows that $f\mapsto mf$ is continuous on $\dot H^{1-2/p}_{q}$.

For the last claim, let $f\in \dot{H}^{1+2/p}_{q'}$. The first
step allows us to bound $ |D|^{2/p}(m\nabla f)$ in $L^{q'}$.
For the term $ |D|^{2/p}(\nabla m f)$, we note that
$\dot{H}^{1+2/p}_{q'}\hookrightarrow L^{3q'}$ since 
$1+\frac2p-\frac3{q'}= -\frac1{q'}$ and that 
$|D|^{2/p}\nabla m\in L^{3q'/2}$ since $(1+\frac2p)\frac{3q'}2=3$.
Hence  $|D|^{2/p}\nabla m f$ belongs to $L^{q'}$ by H\"older's inequality.
The remaining summand $\nabla m |D|^{2/p} f$ is contained in $L^{q'}$ 
because $\nabla m\in L^3$ and $\dot{H}^{1}_{q'}\hookrightarrow L^{\frac{3q'}{3-q'}}$.
\end{proof}

\begin{proposition}[]\label{pro:strichcond}
  Assume that the coefficients $a(x)>0$ and $b(x,\partial)$ satisfy 
  \begin{enumerate}
    \item 
    $\bra{x}^{2+\delta}|D^{2}a|+\bra{x}^{1+\delta}|Da|+
    \bra{x}^{\delta}|a-1|\in L^{\infty}$,
    \item 
    $|b(x,\partial)v|\lesssim
    \bra{x}^{-2-\delta}|v|+\bra{x}^{-1-\delta}|\nabla v|$
  \end{enumerate}
  for some $\delta>0$. Let $(p,q)$ and $(r,s)$ be wave admissible. 
  Then there exists $R_{0}>0$ such that, for any
  $R\ge R_{0}$ and any solution 
  $U$ of problem \eqref{eq:wavevarcompl},
  we have the estimate
  \begin{equation}\label{eq:condstrich}
    \||D|^{-\frac 2p}D_{t,x}U\| _{L^{p}L^{q}}
    \lesssim
    \|D_{t,x}U(0)\|_{L^{2}}+ \||D|^{\frac2r} F \|_{L^{r'}L^{s'}}+
    \|D_{t,x}U\|_{L^{2}L^{2}(|x|\le R+1)}
  \end{equation}
  with an implicit constant depending on $R$.
\end{proposition}

\begin{proof}%[Proof of ...]
  Let $R$ be a  large parameter to be chosen below and
  $\chi(x)=\chi_R(x)$ be a smooth cutoff equal to 1 on a
  ball $B(0,R)$ and vanishing outside $B(0,R+1)$, 
  whose derivatives are bounded independently of $R$.
  We decompose $U=v+w$ with $v=\chi U$ and $w=(1-\chi)U$.
  Let $(p,q)$ and $(r,s)$ be wave admissible. 
 
 1) The piece $w$ is supported in $|x|\ge R$ and solves the problem
  \begin{equation*}
    (a \partial_{t}^{2}-\Delta-b(x,\partial))w=G+ (1-\chi)F,
    \qquad
    w(0,\cdot)=(1-\chi) U_{0},
    \qquad
    w_{t}(0,\cdot)=(1-\chi) U_{1},
  \end{equation*}
  where $G$ is the commutator $G=[\chi,\Delta+b]U$.  
  Note that $G$ is supported in $R\le |x|\le R+1$
  and satisfies, for some constant depending on the coefficients,
  \begin{equation}\label{eq:estR}
    |G|\le C(|U|+|D_{x} U|)\one{R\le |x|\le R+1}.
  \end{equation}
  Moreover, by choosing $R$ large enough, we see that
  in the region $|x|\ge R$ the coefficients fulfill 
  assumptions (8), (9) and (10) of 
  \cite{MetcalfeTataru12}
  (modified as in Remark 1 of that paper) with a constant
  $\varepsilon>0$ which can be made arbitrarily small as
  $R\to \infty$. 
  
  We want to apply Theorem~2 of \cite{MetcalfeTataru12} with $s=0$
  for $R\ge R_{0}$ and some sufficiently large $R_{0}\ge 2$.
However this theorem does directly apply to zero-order terms
in our situation. So we use it with the modified inhomogeneity 
$G+  (1-\chi)F+ b_0w$, where $bw= b_1\cdot \nabla w +b_0w$.
We combine Theorem~2 with estimates (12) and  (16) 
of \cite{MetcalfeTataru12}, all for $s=0$.
These estimates allow to bound the $X^0$ and $Y^0$ norms of 
\cite{MetcalfeTataru12} from below and above, respectively,
 by weighted $L^2$-based norms.
  Using also Lemma~\ref{lem:mult}, it follows
  \begin{align*}
  \sup_{j\ge j_R} &\|\bra{x}^{-\frac12} \nabla w\|_{L^2L^2(A_j)}
  +\||D|^{\sigma}D_{t,x}w\|_{L^{p}L^{q}}\\
   & \lesssim
    \|D_{t,x}w(0)\|_{L^{2}}
    + \||D|^\rho ((1-\chi)F) \|_{L^{r'}L^{s'}}
    + \sum_{j\ge j_R} \|\bra{x}^{\frac 12}(G+ b_0w)\|_{L^{2}L^{2}(A_j)} \\
    &\lesssim  \|D_{t,x}w(0)\|_{L^{2}}+
      \||D|^\rho F \|_{L^{r'}L^{s'}}+
    \|\bra{x}^{\frac 12+}G\|_{L^{2}L^{2}}
    +\sum_{j\ge j_R} \|\bra{x}^{-\frac32-\delta}w\|_{L^{2}L^{2}(A_j)},
  \end{align*}
  where   $R\eqsim 2^{j_R}$, $\sigma=-\frac2p$, $\rho=\frac2r$
  and $A_j=\{2^{j-1}\le |x|\le 2^{j}\}$.
H\"older's and Sobolev's inequalities imply
\begin{align*} 
\sum_{j\ge j_R} \|\bra{x}^{-\frac32-\delta}w\|_{L^{2}L^{2}(A_j)}
&\lesssim \sum_{j\ge j_R} 2^{-j(1/2+\delta)}\|w\|_{L^{2}L^{6}(A_j)}
\lesssim \sum_{j\ge j_R} 2^{-j(1/2+\delta)}\|\nabla w\|_{L^{2}L^{2}(A_j)}\\
&\lesssim R_0^{-\delta} \sup_{j\ge j_R} \|\bra{x}^{-\frac12}\nabla w\|_{L^{2}L^{2}(A_j)}.
\end{align*}
Taking a large $R_0$, we infer
  \begin{equation*}
    \||D|^{\sigma}D_{t,x}w\|_{L^{p}L^{q}}
    \lesssim
    \|D_{t,x}w(0)\|_{L^{2}}+
    \||D|^\rho F \|_{L^{r'}L^{s'}}+
    \|\bra{x}^{\frac 12+}G\|_{L^{2}L^{2}}.
  \end{equation*}

Since $G$ is compactly supported in $R\le|x|\le R+1$, we can
  bound
  \begin{equation*}
    \|\bra{x}^{\frac 12+}G\|_{L^{2}L^{2}}\lesssim
    \|G\|_{L^{2}L^{2}}
  \end{equation*}
  with an implicit constant $\eqsim R^{1/2+}$,
  and Sobolev's embedding further yields
  \begin{equation*}
    \|D_{t,x}w(0)\|_{L^{2}}\le C(R)\|D_{t,x}U(0)\|_{L^{2}}.
  \end{equation*}
  Recalling the definition of $G$,
  the previous estimate can thus be simplified to
  \begin{equation}\label{eq:str1}
    \||D|^{\sigma}D_{t,x}w\|_{L^{p}L^{q}}
    \lesssim
    \|D_{t,x}U(0)\|_{L^{2}}+
    \||D|^\rho F \|_{L^{r'}L^{s'}}+
    \||U|+|D_{x}U|\|_{L^{2}L^{2}(R\le |x|\le R+1)}.
  \end{equation}

 2) Next we consider the remaining piece $v$ supported in $|x|\le R+1$, 
 which solves
  \begin{equation*}
    (a \partial_{t}^{2}-\Delta-b(x,\partial))v=-G+ \chi F,
    \qquad
    v(0,x)=\chi U_{0},
    \qquad
   \partial_{t} v(0,x)=\chi U_{1}.
  \end{equation*}
  In the region $|x|\le R+1$ the coefficients
  satisfy assumptions (8), (9) and (10) of 
  \cite{MetcalfeTataru12}
  (again modified as in Remark 1 of the paper) but with a possibly 
 large constant $\epsilon$ there. We want to apply Theorem~3 of this 
 paper. As above, we have to generalize this result to the case of 
 potentials. Moreover, in this theorem only treats inhomogenities
  in $Y^0$ and not in  $L^{r'}\dot{H}^{-2/r}_{s'}+ Y^0$, as needed by us.

We first extend this result to a forcing term of the form
$f+g \in L^{r'}\dot{H}^{-2/r}_{s'}+ L^2L^2_{1/2+}$
using the parametrix $K$ from  Theorem~3 of
\cite{MetcalfeTataru12}. 
(We note that $L^2L^2_{1/2+}\hookrightarrow Y^0$ by (16) 
of this paper).
Let $P=a\partial_t^2 - \Delta- b_1\cdot\nabla$. 
We consider a  function $u$ with $Pu=f+g$. 
Set $\tilde u=u-Kf$ so that
$P\tilde u= (I-PK)f+g$. 
(See the proof of Lemma~9 in  \cite{MetcalfeTataru12}.)
We restrict the time interval to $t\in[0,\tau]$ 
for some $\tau\in(0,2]$ and let $L^p_\tau=L^p_{(0,\tau)}$.
Estimates (24) and (25) of  \cite{MetcalfeTataru12} then yield
\begin{align*}
&\||D|^{-\rho}D_{t,x}u\|_{L^{r}_\tau L^{s}} 
+ \||D|^{\sigma}D_{t,x}u\|_{L^{p}_\tau L^{q}} \\
&\le \||D|^{-\rho}D_{t,x}\tilde u\|_{L^{r}_\tau L^{s}} 
+  \||D|^{-\rho}D_{t,x}Kf\|_{L^{r}_\tau L^{s}}
+ \||D|^{\sigma}D_{t,x}\tilde u\|_{L^{p}_\tau L^{q}} 
+\||D|^{\sigma} D_{t,x}Kf\|_{L^{p}_\tau L^{q}} \\
&\lesssim \|D_{t,x}(u-Kf)\|_{L^{\infty}_\tau L^{2}} +
   \|(I-PK)f+g\|_{Y^0}  + \||D|^{\rho} f\|_{L^{r'}_\tau L^{s'}}\\
&\lesssim \|D_{t,x} u\|_{L^{\infty}_\tau L^{2}}
   +\||D|^{\rho} f\|_{L^{r'}_\tau L^{s'}} + \|g\|_{Y^0}.
\end{align*}
Here we also use that, on $(0,\tau)$, 
the $X^{0}$ norm (modified as in Remark 1 of
\cite{MetcalfeTataru12}) is controlled by the $L^{\infty}L^{2}$ norm.
The implicit constant is uniform in $\tau\le 2$, 
but depends on $R$.

Let $t\in[0,\tau]$ and $D_{t,x}u(0)=0$. 
A standard energy estimate yields
\begin{align*}
\|D_{t,x}u(t)\|_{L^2}^2 
&\lesssim  \int_0^t \int_{\mathbb{R}^3} 
(|f|+ |g|+ |\nabla u|)|\partial_t u|\, dx\,ds\\
& \lesssim 
\||D|^{\rho} f\|_{L^{r'}_\tau L^{s'}}\, 
\||D|^{-\rho}\partial_t u\|_{L^{r}_\tau L^{s}}
+ \|g\|_{L^1_\tau L^2} \,  \|\partial_t u\|_{L^\infty_\tau L^2}
+\int_0^t \|D_{t,x}u(s)\|_{L^2}^2 ds.
\end{align*}
By means of Gronwall's inquality we infer
\begin{align*}
\|D_{t,x}u\|_{L^\infty_\tau L^2} 
 &\le \|D_{t,x}u(0)\|_{L^2}
 +\kappa \||D|^{-\rho}\partial_t u\|_{L^{r}_\tau L^{s}}
 +\kappa \|\partial_t u\|_{L^\infty_\tau L^2}\\
&\qquad + c(\kappa)(\||D|^{\rho} f\|_{L^{r'}_\tau L^{s'}} 
 +\|\bra{x}^{{1/2}+}g\|_{L^2_\tau L^2}).
\end{align*}
We can absorb the second and third term on the right
choosing a small $\kappa>0$ so that
\begin{equation}\label{eq:thm3-var}
\||D|^{\sigma}D_{t,x}u\|_{L^{p}_\tau L^{q}}
\lesssim \|D_{t,x}u(0)\|_{L^2}
  +\||D|^{\rho} f\|_{L^{r'}_\tau L^{s'}} 
 +\|\bra{x}^{{1/2}+}g\|_{L^2_\tau L^2}.
\end{equation}

3) We put the term $-b_0v$  to the RHS as before, 
now applying \eqref{eq:thm3-var} with $f=\chi F$ and $g=b_0v-G$.
% We can now apply Theorem 3 of \cite{MetcalfeTataru12}
%  with $s=0$ and restricting the time interval to $t\in[0,\tau]$ 
%  for some $\tau\in(0,2]$. 
  To deal with  the zero-order part,
  we also involve the trivial Strichartz pair $(\infty, 2)$, obtaining
  \begin{align*}
   \|\nabla v\|_{L^\infty_\tau L^2} 
    +\||D|^{\sigma}D_{t,x}v\|_{L^{p}_ \tau L^{q}}
    &\lesssim
  \|D_{t,x} v(0)\|_{L^{2}}+
    \||D|^{\rho} (\chi F)\|_{L^{r'}_\tau L^{s'}}+
    \|\bra{x}^{\frac 12+}G\|_{L^{2}_\tau L^{2}}\\
   &\qquad  + \|\bra{x}^{-\frac32-}v\|_{L^{2}_\tau L^{2}}.
 % + \sum_{j\ge0} \|\bra{x}^{-\frac32-}v\|_{L^{2}_\tau L^{2}(A_j)}.
  \end{align*}
 % Here we used the fact that, on a bounded time interval, 
  %the $X^{0}$ norm (modified as in Remark 1 of
 % \cite{MetcalfeTataru12}) is controlled
 % by the $L^{\infty}L^{2}$ norm. 
 Employing H\"older's and Sobolev's inequalities as in step 1), 
 we control the last summand by 
  \[ \|\bra{x}^{-\frac32-}v\|_{L^{2}_\tau L^{2}}
  \lesssim \|\nabla v\|_{L^{2}_\tau L^{2}}
  \le   \tau^{\frac12}  \|\nabla v\|_{L^{\infty}_ \tau L^{2}}.\]
 % \sum_{j\ge0} \|\bra{x}^{-\frac32-}v\|_{L^{2}_ \tau L^{2}(A_j)}
  %\lesssim \tau^{\frac12}  \sum_{j\ge0} 2^{-j/2} \|\nabla v\|_{L^{\infty}_ \tau L^{2}(A_j)}
  For a fixed small $\tau>0$ we can absorb this term by the LHS.
  As before we  then simplify the estimate to
  \begin{equation*}
    \||D|^{\sigma}D_{t,x}v\|
    _{L^{p}_ \tau L^{q}}
    \lesssim
    \|D_{t,x} v(0)\|_{L^{2}}+
    \||D|^{\rho} F\|_{L^{r'}_\tau L^{s'}}+
    \|G\|_{L^{2}_\tau L^{2}}.
  \end{equation*}
% The standard local in time energy estimate then yields 
%  \begin{equation*}
%    \|D_{t,x} v\|_{L^{\infty}_ \tau L^{2}}\lesssim
%    \|D_{t,x}U(0)\|_{L^{2}}+
%    \||D|^{\rho} F\|_{L^{r'}_\tau L^{s'}}+
%    \|G\|_{L^{2}_\tau L^{2}}.
%  \end{equation*}
  This inequality is invariant under time translations. By a finite iteration,
  we conclude 
  \begin{equation}\label{eq:str2}
    \||D|^{\sigma}D_{t,x}v\|
    _{L^{p}_{[0,2]}L^{q}}
    \lesssim
    \|D_{t,x} v(0)\|_{L^{2}}+
    \||D|^{\rho} F\|_{L^{r'}_{[0,2]} L^{s'}}+
    \|G\|_{L^{2}_{[0,2]}L^{2}}
  \end{equation}
  controlling the initial values by means of the Strichartz pair $(\infty,2)$.
  Observe that this estimate is valid on any time interval of length 2.
 
  We now use a reduction which (to our knowledge) 
  originates in \cite{Burq03-a}.
  Let $\mathcal{I}$ be the sequence of intervals
  $I=[k,k+2]$ for $k\in \mathbb{Z}$ and 
  $\{\phi_{I}\}_{I\in \mathcal{I}}$ be a smooth partition of
  unity adapted to $\mathcal{I}$. The cutoffed function
  $v_{I}=\phi_{I}(t)v(t,x)$ solves
  \begin{equation*}
    (a \partial_{t}^{2}-\Delta-b(x,\partial))v_{I}=F_I + G_{I},
    \qquad
    v_I(k,x)= \partial_{t} v_I(k,x)=0,
  \end{equation*}
 where $F_I=\chi\phi_IF$ and
  $G_{I}=a[\partial_{t}^{2},\phi_{I}]v-\phi_{I} G$ 
  are supported in $\{|x|\le R+1, t\in I\}$. We have  
  \begin{equation}\label{eq:estG}
    |G_{I}(t,x)|\le C(|U|+|D_{t,x} U|)\one{|x|\le R+1}(x)   
    \one{I}(t).
  \end{equation}
% Note also that $v_{I}$ has 0 initial data at $t=\inf I$.
  Estimate \eqref{eq:str2} on the time interval $I$ yields
  \begin{equation*}
    \||D|^{\sigma}D_{t,x}v_{I}\|
    _{L^{p}_{I}L^{q}}
    \lesssim
    \||D|^{\rho} F_I\|_{L^{r'}_{I} L^{s'}}+
    \|G_{I}\|_{L^{2}_{I}L^{2}}.
  \end{equation*}$\|D_{t,x}\D^h\|_{L^{2}L^{2}_{-1/2-}} + \|D_{t,x}\D^h\|_{L^{2}L^{2}_{-1/2-}}$.
  We now raise both sides to the power $p\ge2$ and sum over 
  $I\in \mathcal{I}$, obtaining
  \begin{equation*}
    \textstyle
    \||D|^{\sigma}D_{t,x}v\| _{L^{p}L^{q}}^{p}
    \lesssim                           
    \sum_{I}\||D|^{\sigma}D_{t,x}v_{I}\| _{L^{p}_{I}L^{q}}^{p}
    \lesssim
    \sum_{I}\||D|^{\rho} F_I\|_{L^{r'}_{I} L^{s'}}^p
     +\sum_I \|G_{I}\|_{L^{2}_{I}L^{2}}^{p}.
  \end{equation*}
  Since   $\sum c_{I}^{p}\le(\sum c_{I}^{\kappa})^{p/\kappa}$
  for $\kappa\in[1,p]$, we deduce
  \begin{equation*}
    \textstyle
    \||D|^{\sigma}D_{t,x}v\| _{L^{p}L^{q}}^{p}
    \lesssim
    (\sum_{I}\||D|^{\rho} F_I\|_{L^{r'}_{I} L^{s'}}^{r'})^{p/r'}
    +(\sum_{I}\|G_{I}\|_{L^{2}_{I}L^{2}}^{2})^{p/2}.
  \end{equation*}
  Inquality \eqref{eq:estG} then leads to
  \begin{equation*}
    \||D|^{\sigma}D_{t,x}v\| _{L^{p}L^{q}}
    \lesssim
    \||D|^{\rho} (\chi F)\|_{L^{r'} L^{s'}}+
    \||U|+|D_{t,x}U|\|_{L^{2}L^{2}(|x|\le R+1)}.    
  \end{equation*}
  We also use Sobolev's inequality  to estimate $U$ by $D_{x}U$
  with a constant depending on $R$.
  Together with Lemma~\ref{lem:mult} and \eqref{eq:str1}, the assertion follows.
\end{proof}

Using estimate (19) of \cite{MetcalfeTataru12} one checks easily
that the results in this paper, and hence
Proposition \ref{pro:strichcond},
are valid more generally for the system of wave equations
\begin{equation}\label{eq:wavemat}
  (aI_{3}\partial_{t}^{2}-I_{3}\Delta-b(x,\partial))U=F,
  \qquad
  U(0,\cdot)=U_{0},
  \qquad
  \partial_{t}U(0,\cdot)=U_{1},
\end{equation}
with diagonal principal part,
where $b(x,\partial)$ is a matrix first-order operator
which satisfies decay assumptions as in Proposition \ref{pro:strichcond}.
We now apply \eqref{eq:condstrich} to the Maxwell system
\begin{equation}\label{eq:maxwDhom}
  \textstyle
  \partial_{t}^2\D+\nabla \times \frac 1\mu \nabla 
  \times \frac1{\epsilon} \D=F,
  \qquad \D(0,\cdot)=\D_{0},
  \ \ \partial_{t}\D(0,\cdot)=\D_{1},
  \quad \nabla \cdot\D_{0}=\nabla \cdot\D_{1}
           =\nabla\cdot F=0.
\end{equation}
Recall that the solution to the above problem is given by
\begin{equation}\label{eq:maxwDhom-sol}
\D(t)=\cos (t\sqrt{H})\D_0 + \sin (t\sqrt{H})H^{-1/2}\D_1 
   + H^{-1/2}\int_0^t  \sin ((t-s)\sqrt{H})F(s)ds.
\end{equation}
We denote by $\D^h$  the solution to the problem with $F=0$ and by $\D^i$ 
that one with $\D_0=\D_1=0$. 

For $F=0$, the conditional Strichartz
estimate in Proposition~\ref{pro:strichcond} and the smoothing
estimates in Propositions~\ref{pro:estMaxwhom} and \ref{pro:homhom}
easily yield the Strichartz inequality for $\D^h$. The usual $TT^*$
argument then allows us to bound $\partial_t \D^i$ and
$\sqrt{H} \D^i$ in $L^p\dot{H}^{-2/p}_q$. To replace here $\sqrt{H}$ 
by $\nabla$, one would need a variant of Lemma~\ref{lem:riesz} in 
$\dot{H}^{-2/p}_q$ which should require a substantial effort.
We by-pass this difficulty by means of a modified $TT^*$ argument
that also uses ideas from Proposition~\ref{pro:estMaxwhom}.
This is possible since we only have to control an error term 
in $L^2L^2_{-1/2-}$ arising from Proposition~\ref{pro:strichcond}. 
Here it turns out to be enough to estimate $\nabla H^{-1/2}$ in
$L^2_{-1/2-}$ just using Lemma~\ref{lem:riesz}. 
%Since we employ 
%\eqref{eq:eqpowH1}, we have to strenghten the assumptions of 
%Proposition \ref{pro:homhom} a bit.

\begin{theorem}[]\label{the:strichM}
 Under the assumptions of Proposition \ref{pro:homhom},
  Then  the solution $\D(t,x)$ to problem \eqref{eq:maxwDhom}
  satisfies for any wave admissible $(p,q)$ and $(r,s)$ the estimate
  \begin{equation}\label{eq:strichM}
    \||D|^{-\frac 2p}D_{t,x}\D\| _{L^{p}L^{q}}
    \lesssim
    \|\nabla\D_{0}\|_{L^2}+\|\D_{1}\|_{L^{2}}+\||D|^{\frac2r}  F\|_{L^{r'} L^{s'}} .
  \end{equation}
\end{theorem}

\begin{proof}%[Proof of ...]
% We assume that $F$ has support in $t\ge0$. The general case is then treated 
% by time reversal and linearity.
  As in the proof of Proposition \ref{pro:estMaxw} we can recast
  \eqref{eq:maxwDhom} in the form \eqref{eq:wavemat}.
  Since the conditions in Proposition \ref{pro:strichcond}
  are satisfied, estimate \eqref{eq:condstrich} yields
  \begin{equation*}
    \||D|^{-\frac 2p}D_{t,x}\D\| _{L^{p}L^{q}}
    \lesssim
    \|\nabla\D_{0}\|_{L^2}+\|\D_{1}\|_{L^{2}}
    + \||D|^{\frac2r}  F\|_{L^{r'} L^{s'}}
    +\|D_{t,x}\D\|_{L^{2}L^{2}(|x|\le R_0+1)}.
  \end{equation*}
  for some fixed radius $R_0\ge1$.
  The last term is bounded by a constant times
  \[  \|D_{t,x}\D^h\|_{L^{2}L^{2}_{-1/2-}} + \|D_{t,x}\D^i\|_{L^{2}L^{2}_{-1/2-}}.\]
  %   + \|\sqrt{H}\D^i\|_{L^{2}L^{2}_{-1/2-}}, \]
 % where we have used \eqref{eq:eqpowH2}. 
 In view of \eqref{eq:maxwDhom-sol},
  Propositions \ref{pro:estMaxwhom} and \ref{pro:homhom} and  \eqref{eq:eqpowH1} yield
 \[\|D_{t,x}\D^h\|_{L^{2}L^{2}_{-1/2-}} \lesssim \|\nabla\D_{0}\|_{L^2}+\|\D_{1}\|_{L^{2}}.\]
We thus have shown \eqref{eq:strichM} for $F=0$.

Consider now the case $F\neq0$. To complete the proof it is
sufficient to prove the estimate
\begin{equation*}
  \left\|D_{t,x}\int_{0}^{t}
  H^{-\frac12}e^{i(t-s)\sqrt{H}} F(s)ds\right\|_{Z}\le C\|F\|_{E^{*}}
\end{equation*}
where $F$ is divergence free and we set $ E= L^r\dot H^{-2/r}_s$ and
 $ Z=L^2L^2_{-1/2-}.$ First, we notice that by the 
 Christ--Kiselev Lemma this estimate follows from the analogous 
 unretarded one (since $r>2$)
\begin{equation*}
  \left\|D_{t,x}\int
  H^{-\frac12}e^{i(t-s)\sqrt{H}} F(s)ds\right\|_{Z}\le C\|F\|_{E^{*}}.
\end{equation*}
Next, we split
$D_{t,x}H^{-\frac12}\int e^{i(t-s)\sqrt{H}}F(s)ds
  =D_{t,x}e^{it \sqrt{H}}H^{-1/2}\int e^{-is \sqrt{H}}F(s)ds$
and we recall from Propositions \ref{pro:estMaxwhom}
and \ref{pro:homhom} the inequalities
\begin{equation*}
  \|e^{it \sqrt{H}}f\|_{Z}\lesssim\|f\|_{L^{2}},
  \qquad
  \|\nabla e^{it \sqrt{H}}f\|_{Z}\lesssim\|\nabla f\|_{L^{2}}.
\end{equation*}
Combined with \eqref{eq:eqpowH1}, these estimates yield 
\begin{align}
\hspace*{-0.3cm}  \left\|D_{t,x}\int
  H^{-\frac12}e^{i(t-s)\sqrt{H}} F(s)ds\right\|_{Z}&\lesssim
  \left\|
    \int e^{-is \sqrt{H}}F(s)ds
  \right\|_{L^{2}}
  +
  \left\|
    \nabla H^{-\frac12}\int e^{-is \sqrt{H}}F(s)ds
  \right\|_{L^{2}}\notag\\
%\end{equation*}
%which by \eqref{eq:eqpowH1} is equivalent to
%\begin{equation}
  &\lesssim
  \left\|
    \int \cos(s \sqrt{H}) F(s)ds
  \right\|_{L^{2}}+
  \left\|
    \int \sin(s \sqrt{H}) F(s)ds
  \right\|_{L^{2}}.\label{eq:interduh}
\end{align}
In the first part of the proof we have seen
\begin{equation*}
  \|D_{t,x}H^{-\frac12}\sin(t \sqrt{H})f\|_{E}
  \lesssim\|f\|_{L^{2}},
  \qquad
  \|D_{t,x}\cos(t \sqrt{H})f\|_{E}
  \lesssim\|\nabla f\|_{L^{2}}.
\end{equation*}
Considering only $D_{t}$, the first  inequality implies
\begin{equation*}
  \|\cos(t \sqrt{H})f\|_{E}
  \lesssim\|f\|_{L^{2}},
\end{equation*}
while the second one gives
\begin{equation*}
  \|\sin(t \sqrt{H})f\|_{E} =\|D_t\cos(t \sqrt{H})H^{-\frac12}f\|_{E}
  \lesssim\|\nabla H^{-\frac12} f\|_{L^{2}}\lesssim\|f\|_{L^{2}}
\end{equation*}
using again \eqref{eq:eqpowH1}. Applying the dual estimates we see
that both terms in \eqref{eq:interduh} can be estimated
by $\|F\|_{E^{*}}$, and this concludes the proof.
\end{proof}

The above theorem now easily implies our main result.

\begin{proof}[Proof of Theorem~\ref{the:strichmaxw}]
%We can now easily extend Theorems~\ref{the:strichM} and \ref{the:strichM-in}
%to  $\B$, $\E$, and $\H$. The final results are stated in the introduction
%in the framework of the (first-order) Maxwell system. 
In view of \eqref{eq:waveD}, the main results for $\D$ is an immediate consequence
of the above theorem. The magnetic  field $\B$ solves \eqref{eq:waveB} 
which is of the same form as equation \eqref{eq:waveD} for $\D$ except that 
$\epsilon$ and $\mu$ are interchanged.
So Theorem~\ref{the:strichM} is true 
with $\B$ instead of $\D$ if we  replace the condition on second derivatives 
in (2) of Proposition~\ref{pro:homhom} by
\[ |D^{2}\epsilon|\lesssim \bra{x}^{-2-\delta}, \qquad 
    |D^{2}\mu|      \lesssim \bra{x}^{-\frac52-\delta}.\]
% and interchange   $\epsilon$ and $\mu$ in the definition of $H$.
 Theorem~\ref{the:strichmaxw} for $\B$ again follows
easily taking into account   \eqref{eq:waveB} and Lemma~\ref{lem:mult}.
 
% For the electric field $\E=\epsilon^{-1}\D$, the proof of 
% Theorem~\ref{the:strichM} implies that the multiplication operators 
% induced by $\epsilon$, $\mu$, $\epsilon^{-1}$, or $\mu^{-1}$ are bounded 
% on $\dot H^{2/p}_{q'}$. By duality, they also act continuously on $\dot H^{-2/p}_{q}$. 
% We further have the embedding $\dot H^{1-2/p}_{q}\hookrightarrow L^{3q}$. As above, 
% one then  shows that these operators are bounded on $\dot H^{1-2/p}_{q}$.
% Special cases are $L^2$ and $\dot H^{1}$, where $p=\infty$ and $q=2$.
By Lemma~\ref{lem:mult} we can  replace
$\D$ by $\E$ in Theorem~\ref{the:strichmaxw}, with the divergence conditions
$\nabla\cdot(\epsilon \E_0)=\nabla\cdot(\epsilon \E_1)=0$. 
% The inhomogeneous Strichartz estimate for $\E$ does not follow in the same way
% because of the presence of $\sqrt{H}$ on the left-hand side of \eqref{eq:strichM-in}.
% Instead we  employ the
% operator $H_{\E}\E= \frac1\epsilon\nabla \times\frac 1\mu \nabla \times \E$ which 
% is self adjoint and non negative on the Hilbert space
% \begin{equation}\label{eq:HS-E}
%   \mathcal{H}_{\E}=\{u\in L^{2}(\mathbb{R}^{3};\mathbb{C}^{3})
%   \colon
%   \nabla \cdot (\epsilon u)=0\}
% \end{equation}
% endowed with the scalar product 
% $(u,v)_{\mathcal{H}_{\E}}=\int \epsilon u \cdot \overline{v}dx$. As in the proof 
% Theorem~\ref{the:strichM-in}, see \eqref{eq:SC},
% the variant for $\E$ of the homogeneous estimate 
% \eqref{eq:strichM} then implies the boundedness $L^2\to E$ of the sine and cosine 
% functions for $H_{\E}$. We can then deduce the inhomogeneous 
% Strichartz inequality \eqref{eq:strichM-in} with $\E$ instead of $\D$ as before, 
% now assuming that $\nabla(\epsilon F)=0$. As before, in view of \eqref{eq:waveE}
% we have shown Theorems~\ref{the:strichmaxw} and \ref{the:strichmaxw-in}
% also for $\E$.
One can pass from $\B$ to $\H$ in same way as from $\D$ to $\E$.
\end{proof}

\printbibliography %needed
%%% >>> BIBTEX: (cancellare \usepackage{biblatex})
% \bibliography{/Users/piero/.bib/bibliodatabase.bib}
% \bibliographystyle{abbrv}
%%% >>> entrambe si possono usare contemporaneamente a:
% \begin{\thebibliography}{10}  DATI  \end{thebiblography}
\end{document}